\documentclass[10pt, reqno]{amsart}
\usepackage[margin=1in]{geometry}

\usepackage{amsmath,amsthm,wasysym,bbold,tensor,stmaryrd,epigraph,mathdots,stmaryrd}
\usepackage{amssymb}
\usepackage{amsfonts}
\usepackage{mathscinet}
\usepackage[dvipsnames]{xcolor}
\usepackage{tikz,keyval,pgfsys}
\usepackage{graphicx}
\usepackage{mathtools, enumerate}

\usetikzlibrary{arrows, decorations.pathreplacing}
\usepgflibrary{arrows}
\newtheorem{thm}{Theorem}[section]
\newtheorem*{main}{Main Theorem}
\newtheorem{prop}[thm]{Proposition}
\newtheorem{lem}[thm]{Lemma}
\newtheorem{cor}[thm]{Corollary}

\theoremstyle{definition}
\newtheorem{defn}[thm]{Definition}

\theoremstyle{remark}
\newtheorem{rem}[thm]{Remark}
\theoremstyle{remark}

\allowdisplaybreaks

\begin{document}

\binoppenalty=10000
\relpenalty=10000

\numberwithin{equation}{section}

\newcommand{\QQ}{\mathbb{Q}}
\newcommand{\RR}{\mathbb{R}}
\newcommand{\ZZ}{\mathbb{Z}}
\newcommand{\Hilb}{\mathrm{Hilb}}
\newcommand{\CC}{\mathbb{C}}
\newcommand{\PP}{\mathcal{P}}
\newcommand{\Hom}{\mathrm{Hom}}
\newcommand{\Rep}{\mathrm{Rep}}
\newcommand{\gl}{\mathfrak{gl}}
\newcommand{\hh}{\mathfrak{h}}
\newcommand{\ppp}{\mathfrak{p}}
\newcommand{\VV}{\mathcal{V}}
\newcommand{\NN}{\mathbb{N}}
\newcommand{\OO}{\mathcal{O}}
\newcommand{\UU}{\mathcal{U}}
\newcommand{\DD}{\mathcal{D}}
\newcommand{\WW}{\mathcal{W}}
\newcommand{\AAA}{\mathcal{A}}
\newcommand{\WWo}{\WW^\circ}
\newcommand{\MM}{\mathcal{M}}
\newcommand{\SHH}{\mathcal{S}\ddot{\mathcal{H}}}
\newcommand{\HH}{\ddot{\mathcal{H}}}
\newcommand{\git}{/\kern-.35em/}
\newcommand{\KK}{\mathbb{K}}
\newcommand{\Proj}{\mathrm{Proj}}
\newcommand{\quot}{\mathrm{quot}}
\newcommand{\core}{\mathrm{core}}
\newcommand{\Sym}{\mathrm{Sym}}
\newcommand{\FF}{\mathbb{F}}
\newcommand{\Span}[1]{\mathrm{span}\{#1\}}
\newcommand{\UTor}{U_{\qqq,\ddd}(\ddot{\mathfrak{sl}}_\ell)}
\newcommand{\Sss}{\mathcal{S}}
\newcommand{\Hecke}{\mathcal{H}}
\newcommand{\ee}{\widetilde{S}}
\newcommand{\xx}{\mathbf{x}}
\newcommand{\del}{\partial}
\newcommand{\ad}{\mathrm{ad}}
\newcommand{\tr}{\mathrm{tr}}
\newcommand{\rrr}{\mathfrak{r}}
\newcommand{\zzz}{\mathsf{z}}
\newcommand{\ZZZ}{\mathsf{Z}}
\newcommand{\qqq}{\mathsf{q}}
\newcommand{\ttt}{\mathsf{t}}
\newcommand{\KKK}{\mathbb{K}}
\newcommand{\LL}{\mathbb{L}}
\newcommand{\RRR}{\mathbf{R}}
\newcommand{\CCC}{\mathfrak{C}}
\newcommand{\XX}{\mathbf{X}}
\newcommand{\XXu}{\underline{\mathbf{X}}}
\newcommand{\DDD}{\mathbf{D}}
\newcommand{\YY}{\mathbf{Y}}
\newcommand{\YL}{ {}^L\YY}
\newcommand{\YLu}{ {}^L\underline{\YY}}
\newcommand{\YR}{ {}^R\YY}

\title{Harish-Chandra isomorphism for cyclotomic double affine Hecke algebras}
\author{Joshua Jeishing Wen}
\address{Department of Mathematics, Northeastern University, Boston, MA, USA}
\address{Fakult\"{a}t f\"{ur} Mathematik, Universit\"{a}t Wien, Vienna, Austria}
\email{joshua.jeishing.wen@univie.ac.at}
\maketitle

\begin{abstract}
We confirm a conjecture of Braverman--Etingof--Finkelberg that the spherical subalgebra of their cyclotomic double affine Hecke algebra (DAHA) is isomorphic to a quantized multiplicative quiver variety for the cyclic quiver, as defined by Jordan.
The isomorphism is constructed as a $\qqq$-analogue of Oblomkov's cyclotomic radial parts map for the rational case.
In the appendix, we also prove that the spherical cyclotomic DAHA is isomorphic to the image of a shifted quantum toroidal algebra under Tsymbaliuk's GKLO homomorphism.
\end{abstract}

\section{Introduction}
This paper aims to shine a speck of light on the largely-unexplored world of integrable systems of Ruijsenaars--Schneider type for complex reflection groups.
The Hamiltonians for the quantizations of such systems should be $\qqq$-difference operators.
For Calogero--Moser systems and their generalizations to \textit{real} reflection groups, the Olshanetsky--Perelomov systems \cite{OlshPer}, the quantum Hamiltonians are instead differential operators, and their study can be facilitated by \textit{Dunkl operators} \cite{DunklRef}.
The latter are commuting permutation-differential operators with explicit formulas, and one obtains the quantum Calogero--Moser--Olshanetsky--Perelomov Hamiltonians by symmetrizing them.

Thus, one can define Calogero--Moser systems for complex reflection groups by generalizing the Dunkl operators, which was done by Dunkl--Opdam \cite{DunkOpComp}.
Moreover, an entire algebraic structure for studying such systems can be generalized: these are the \textit{rational Cherednik algebras}.
The double affine Hecke algebra (DAHA) was originally defined by Cherednik for Weyl groups, who applied it with great success to solving Macdonald's conjectures \cite{ChereAnnals,ChereEval}.
In type $A$, the DAHA provides the algebraic structure for the trigonometric Ruijsenaars--Schneider system; its rational degenerations, which includes the rational Cherednik algebra, in turn yielded many new insights to Calogero--Moser systems.
Rational Cherednik algebras for complex reflection groups were defined by Etingof--Ginzburg within the larger context of \textit{symplectic reflection algebras} \cite{EtGinz}, and the cases for certain groups have been studied extensively.

One should be careful in removing the ``rational'' in ``rational Cherednik algebra for complex reflection groups'' because the DAHA for such groups has yet to be defined.
While Hecke algebras have been defined in this generality \cite{BMR}, the naive definition of an affine Hecke algebra requires the group to preserve a lattice, which is only true in the real case.
Thus, it is not straightforward to define such an object.
For the group $\Sigma_n\wr\ZZ/\ell\ZZ$, Braverman--Etingof--Finkelberg \cite{BEF} gave a definition of \textit{cyclotomic DAHA} that sidesteps the issue.
Their cyclotomic DAHA is like the subalgebra obtained from a fictional DAHA for $\Sigma_n\wr\ZZ/\ell\ZZ$ by sphericalizing away the cyclic group factors; in fact, they proved that their cyclotomic DAHA degenerates to the analogous such subalgebra for the rational Cherednik algebra.

It is thus sensible to expect results for the spherical subalgebra in the rational case to lift to the cyclotomic DAHA, and this paper confirms one of these expectations, conjectured in \cite{BEF}.
Namely, in the rational case, results of Oblomkov \cite{ObCyc} and Gordon \cite{GorCycDiff} combine to show that the spherical subalgebra of the rational Cherednik algebra for $\Sigma_n\wr\ZZ/\ell\ZZ$ is isomorphic to a \textit{quantized Nakajima quiver variety} (cf. \cite{NakaKM}) for the $\widetilde{A}_{\ell-1}$ quiver.
This isomorphism opened the door to the geometric representation theory of the cyclotomic rational Cherednik algebra \cite{GorCyc}, which in turn was crucial in the proof of Haiman's wreath Macdonald positivity conjecture \cite{BezFinkMc,HaimCSH}. 

For the cyclotomic DAHA, the candidate replacement for the Nakajima quiver variety is the \textit{multiplicative} quiver variety (cf. \cite{CBShaw}), also for the $\widetilde{A}_{\ell-1}$ quiver.
The multiplicative quiver varieties for such quivers were studied by Chalykh--Fairon \cite{ChalFai} as phase spaces for generalized Ruijsenaars--Schneider integrable systems.
Their quantizations were defined by Jordan \cite{JordanMult}, although the case $\ell=1$ requires a new definition, given in this paper.
The cyclotomic DAHA depends on three parameters $(\ZZZ,\qqq,\ttt)$.
Our main result is:
\begin{main}
The spherical cyclotomic DAHA for $\Sigma_n\wr\ZZ/\ell\ZZ$ is isomorphic to a quantized multiplicative quiver variety for the $\widetilde{A}_{\ell-1}$ quiver (\ref{CyclicQ}) when:
\begin{itemize}
\item $\ttt=\qqq^k$ for $k>2n\ell$ and $\qqq$ is generic
\item or $(\qqq,\ttt)$ is generic.
\end{itemize}
\end{main}
\noindent In their work \cite{ChalFai}, Chalykh--Fairon identified three families of Poisson-commuting functions, and an analogous three families of commuting elements exist in the spherical cyclotomic DAHA.
Our isomorphism takes manifest quantizations of the former and sends them to the latter.

The proof of the main theorem basically follows that of the rational case; in fact, perhaps the biggest surprise of this paper is that there are no surprises.
Neither algebra in the theorem has a presentation via generators and relations, so the principal difficulty comes from constructing a homomorphism from one to the other.
We do this by constructing a $\qqq$-analogue of Oblomkov's cyclotomic radial parts map \cite{ObCyc}.
To do so, it is important to have an analogous isomorphism established between the usual (non-cyclotomic) spherical DAHA of $GL_n$ and the a quantized multiplicative quiver variety for the Jordan quiver (single vertex with a loop), which was proved in our earlier work \cite{QHarish}.
The cyclotomic DAHA is a subalgebra of $GL_n$-DAHA, and we define an embedding from the quantized $\widetilde{A}_{\ell-1}$ quiver variety into the quantized Jordan quiver variety; our radial parts map is then just the restriction of the isomorphism from \cite{QHarish} to these subalgebras.

Despite the lack of surprises, this paper can get technical.
Oblomkov's radial parts map is defined using natural constructions on rings of differential operators: localization, restriction, etc.
In our $\qqq$-deformed setting, there is a shortage of functors and thus we must build these homomorphisms by hand.
Additionally, showing that the sphericalized cyclotomic $\qqq$-Dunkl elements of \cite{BEF} lie in the image of the radial parts map requires some tricks; for this, we found much inspiration from the manipulations in \cite{DiKedqt} involving the Gaussian.
We also mention that, as in \cite{QHarish}, our radial parts map relies on the Etingof--Kirillov realization of Macdonald polynomials via intertwiners for $U_\qqq(\gl_n)$ \cite{EtKirQuant}.
The beautiful idea to use the Etingof--Kirillov realization comes from \cite{VarVassRoot}.

Finally, we mention that in the setting of \textit{3d mirror symmetry}, multiplicative quiver varieties are multiplicative \textit{Higgs branches}.
The authors of \cite{BEF} also conjectured an isomorphism between the spherical cyclotomic DAHA and a quantized \textit{$K$-theoretic Coulomb branch} \cite{BFN,FTMult, FTShift}.
In the appendix, we show that the spherical cyclotomic DAHA is isomorphic to a \textit{truncated shifted quantum toroidal algebra} \cite{TsymGKLO}.
The latter algebras are generally expected to be isomorphic to quantized $K$-theoretic Coulomb branches.

\subsection*{Acknowledgements}
I would like to thank Valerio Toledano Laredo for many stimulating discussions.
This work received support from NSF-RTG grant ``Algebraic Geometry and Representation Theory at Northeastern University'' (DMS-1645877) and ERC consolidator grant No. 101001159 ``Refined invariants in combinatorics, low-dimensional topology and geometry of moduli spaces''.

\subsection*{Notation}
With the intent of providing clarity, we have sought to assign symbols according to the following principles:
\begin{itemize}
\item ground rings are denoted using blackboard bold ($\CC$, $\LL$, $\KK$);
\item more complicated algebras are denoted using calligraphic letters ($\mathcal{R}$, $\HH_n$);
\item parameters are denoted using sans-serif fonts: ($\qqq$, $\ttt$, $\ZZZ_i$);
\item most DAHA elements are denoted using capital Roman letters ($T_i$, $X_i$, $Y_i$);
\item matrices will be denoted using bold capital letters ($\mathbf{X}$, $\mathbf{Y}$, $\mathbf{D}$).
\end{itemize}

\section{Cyclotomic DAHA} 
In what follows, we will work over the ring $\LL:=\CC[\qqq^{\pm 1},\ttt^{\pm 1}]$ and its field of fractions $\KKK:=\CC(\qqq,\ttt)$.

\subsection{$\mathbf{GL_n}$-DAHA}
The cyclotomic DAHA sits inside the usual DAHA of $GL_n$.
We begin by reviewing the definition and structures of the latter that will be relevant to our study.
Our reference is \cite{ChereDAHA}, although our conventions will be closer to those of \cite{KirLec}.

\subsubsection{Definition of $\HH_n^0(\qqq,\ttt)$ and PBW}\label{PBW}
The $GL_n$-DAHA $\HH_n^{0}(\qqq,\ttt)$ is the $\KKK$-algebra with generators
\[
\{T_i,
X_j^{\pm 1},
\pi^{\pm 1}\, |\,i=1,\ldots, n-1\hbox{ and } j=1,\ldots, n\}
\]
and relations
\begin{equation*}
\begin{aligned}
&(T_i-\ttt)(T_i+\ttt^{-1})=0;&&X_iX_j=X_jX_i; \\
&T_iT_{i+1}T_i=T_{i+1}T_iT_{i+1};&& T_iT_j=T_jT_i\hbox{ for }j\not=i, i+1;\\
&T_iX_j=X_jT_i\hbox{ for }j\not=i, i+1; && T_iX_iT_i=X_{i+1};  \\
&\pi T_i=T_{i+1}\pi; && \pi^n T_i=T_i\pi^n;\\
&\pi X_i=X_{i+1}\pi;&& \pi X_n=\qqq^{-2}X_1\pi
\end{aligned}
\end{equation*}
The extra superscript ``$0$'' is to place it in a family with the cyclotomic algebras, to be introduced in \ref{Cyclo} below.
We can also define this as an $\LL$-algebra, which we denote by $\HH^{0}_n(\qqq,\ttt)_\LL$.

There is also an alternative presentation via the elements
\begin{equation}
Y_i:= T_i\cdots T_{n-1}\pi^{-1} T_1^{-1}\cdots T_{i-1}^{-1}
\label{YDef}
\end{equation}
for $i=1,\ldots,n$ 
The elements $\{Y_1,\ldots, Y_n\}$ generate a polynomial subalgebra, and $\HH_n^{0}(\qqq,\ttt)$ has generators
\[
\left\{ T_i, X_j^{\pm 1}, Y_j^{\pm 1} \, | \, i=1,\ldots, n-1\hbox{ and }j=1,\ldots, n\right\}
\]
and relations
\begin{equation}
\begin{aligned}
&(T_i-\ttt)(T_i+\ttt^{-1})=0;\\
&T_iT_{i+1}T_i=T_{i+1}T_iT_{i+1};&& T_iT_j=T_jT_i\hbox{ for }j\not=i, i+1;\\
&X_iX_j=X_jX_i; && Y_iY_j=Y_jY_i;\\
&T_i X_i T_i=X_{i+1}\hbox{ for }i\not=n; && T_i^{-1}Y_iT_i^{-1}=Y_{i+1}\hbox{ for }i\not=n;\\
&T_iX_j=X_jT_i\hbox{ for }j\not=i, i+1;&&T_iY_j=Y_jT_i\hbox{ for }j\not=i, i+1;\\
&Y_1\cdots Y_nX_j=\qqq^2X_jY_1\cdots Y_n;&& X_1\cdots X_nY_j=\qqq^{-2}Y_jX_1\cdots X_n;\\
&X_1Y_2=Y_2T^2_1X_1
\end{aligned}
\label{DAHAY}
\end{equation}
We note that this presentation is also valid for $\HH_n^{0}(\qqq,\ttt)_\LL$.
%With respect to these generators, $\HH_n^0(\qqq,\ttt)$ and $\HH_n^0(\qqq,\ttt)_\mathbb{L}$ each have the \textit{Cherednik involution} $\varphi$:
%\begin{align*}
%\varphi(\qqq)&= \qqq^{-1},
%&\varphi(\ttt)&= \ttt^{-1},
%&\varphi(X_i)&= Y_i^{-1},
%&\varphi(Y_i)&= X_i^{-1},
%&\varphi(T_i)&= T_i^{-1}.
%\end{align*}
 
The elements $\{T_1,\ldots, T_{n-1}\}$ generate a Hecke algebra for the symmetric group $\Sigma_n$ with parameter $\ttt$.
Let $s_i\in \Sigma_n$ be the $i$th adjacent transposition.
For any $w\in\Sigma_n$ with reduced expression $w=s_{i_1}\cdots s_{i_k}$, we the element
\[
T_w:=T_{i_1}\cdots T_{i_k}
\] 
is well-defined (independent of the reduced expression).

\begin{thm}[\cite{ChereDAHA}]
$\HH_n^0(\qqq,\ttt)_{\LL}$ has an $\LL$-basis given by elements of the form
\[
M_XM_YT_w
\]
where $M_X$ and $M_Y$ are monomials in $\{X_i\}$ and $\{Y_i\}$, respectively, and $w\in\Sigma_n$.
\end{thm}

\subsubsection{Spherical subalgebra}\label{Spherical}
For $w\in\Sigma_n$, let $l(w)$ be its \textit{length}.
The \textit{Hecke symmetrizer} element
\[
\ee:=\frac{1}{[n]_{\ttt^2}!}\sum_{w\in\Sigma_n}\ttt^{l(w)}T_w
\]
is idempotent.
Here, for a parameter $\mathsf{s}$ and $k\in\ZZ_{\ge 0}$, we denote:
\begin{equation*}
\begin{aligned}
[k]_{\mathsf{s}}&:=\frac{\mathsf{s}^{n}-\mathsf{s}^{-n}}{\mathsf{s}-\mathsf{s}^{-1}}, & [n]_{\mathsf{s}}!&:=\prod_{k=1}^n[k]_{\mathsf{s}}
\end{aligned}
\end{equation*}

We define the \textit{spherical subalgebra} $\SHH_n^{0}(\qqq,\ttt)$ to be
\[
\SHH_n^{0}(\qqq,\ttt):=\ee\HH_n^{0}(\qqq,\ttt)\ee\subset\HH_n^{0}(\qqq,\ttt)
\]
Notice that $\ee$ is not an element of $\HH_n^{0}(\qqq,\ttt)_\LL$.
By abuse of notation, we define $\SHH_n^{0}(\qqq,\ttt)_\LL$ to be the corresponding subalgebra after base change:
\[
\SHH_n^{0}(\qqq,\ttt)_\LL:=\ee\left( \HH_n^{0}(\qqq,\ttt)_\LL\otimes \LL\left[\frac{1}{[n]_{\ttt^2}!}\right] \right)\ee
\] 
For $k\in\ZZ$, we define the specialization $\SHH_n^{0}(\qqq,\qqq^k)$ by
\[
\SHH_n^{0}(\qqq,\qqq^k):=\CC(\qqq)\otimes\left(\SHH_n^{0}(\qqq,\ttt)_{\LL}\bigg/(\ttt-\qqq^k)\right)
\]

Let $\HH_n^0(\qqq,\ttt)^+$, $\HH_n^0(\qqq,\ttt)^+_\LL$, and $\HH_n(\qqq,\qqq^k)$ be the subalgebras generated by positive powers of $\{X_i\}$ and $\{Y_i\}$, and let $\SHH_n^0(\qqq,\ttt)^+$, $\SHH_n^0(\qqq,\ttt)_\LL^+$, and $\SHH_n^0(\qqq,\qqq^k)^+$ be their respective spherical subalgebras.
By the DAHA relations, the powers of $\ee X_1\cdots X_n\ee$ and $\ee Y_1\cdots Y_n\ee$ generate an Ore set, and we obtain $\SHH_n^0(\qqq,\ttt)_\LL$ from $\SHH_n^0(\qqq,\ttt)_\LL^+$ by inverting them.
We have the following generation lemma:

\begin{lem}\label{GenLem0}
$\SHH_n^{0}(\qqq,\ttt)$ is generated by $\ee\KK[X^{\pm 1}_1,\ldots,X^{\pm 1}_n]\ee$ and $\ee\KK[Y^{\pm 1}_1,\ldots,Y^{\pm 1}_n]\ee$.
Similary, for $k>2n$, $\SHH_n^{0}(\qqq,\qqq^k)$ is generated by $\ee\CC(\qqq)[X^{\pm 1}_1,\ldots,X^{\pm 1}_n]\ee$ and $\ee\CC(\qqq)[Y^{\pm 1}_1,\ldots,Y^{\pm 1}_n]\ee$.
\end{lem}

\begin{proof}
It suffices to prove the analogous generation statement for $\SHH_n^0(\qqq,\ttt)^+$ and $\SHH_n^0(\qqq,\qqq^k)^+$.
If we assign 
\[
\deg(X_i)=(1,0),\,
\deg(Y_i)=(0,1),\,
\deg(T_i)=(0,0)
\]
then $\SHH_n^0(\qqq,\ttt)_\LL^+$ is bigraded with finite rank pieces.
For the case of $(\qqq,\ttt)$, we set $t=1$, and for the case of $(\qqq,\qqq^k)$, we set $\qqq= e^{2\pi i/k}$.
In both cases, the spherical DAHA becomes a symmetrized quantum torus.
The proof of Lemma A.15.2 of \cite{VarVassRoot} proves the result at these specific specializations ($k>2n$ is important here).
We obtain the result at generic parameters by applying Nakayama's Lemma to each (finite-rank) bigraded piece.
\end{proof}

\subsubsection{Macdonald polynomials}
Let $\LL[\mathbf{x}_n^\pm]:=\LL[x_1^{\pm 1},\ldots, x_n^{\pm 1}]$.
The following is well-known:

\begin{thm}[\cite{ChereAnnals}]\label{PolyRep}
There is a faithful representation $\mathfrak{r}$ of $\HH_n^0(\qqq,\ttt)_\LL$ on $\LL[\mathbf{x}_n^\pm]$ given by the formulas: for $f\in \LL[\mathbf{x}_n^\pm]$,
\begin{align}
\label{Demazure}
\mathfrak{r}(T_i)&= \ttt s_i+ (\ttt-\ttt^{-1})\frac{s_i -1}{x_{i}/x_{i+1}-1}\\
\nonumber\mathfrak{r}(X_i)\cdot f&= x_i f\\
\nonumber\mathfrak{r}(\pi)\cdot f(x_1,\ldots, x_n)&= f(x_2, x_3,\ldots, x_n, \qqq^{-2}x_1)
\end{align}
Moreover, this representation remains faithful at any specialization of $(\qqq,\ttt)$ where $\qqq$ is not a root of unity.
\end{thm}
\noindent We call $\mathfrak{r}$ the \textit{polynomial representation}.
By abuse of notation, we will also use $\mathfrak{r}$ to denote the representations obtained via base-change from $\LL$.

Let $\Lambda_n^{\pm}(\qqq,\ttt)_\LL$ denote the base-changed subring of symmetric polynomials:
\[
\Lambda_n^\pm (\qqq,\ttt)_\LL:=\LL[\mathbf{x}_n^\pm]^{\Sigma_n}\otimes \LL\left[ \frac{1}{[n]_{\ttt^2}!} \right]
\]
%Note that $\Lambda_n^\pm(q,t)_R$ has a natural $R$-basis given by \textit{monomial symmetric polynomials}: for a partition $\lambda=(\lambda_1\ge \lambda_2\ge\cdots\ge\lambda_n)\in\ZZ^n$,
%\[
%m_\lambda:=\sum_{w\in\Sigma_n/\mathrm{Stab}(\lambda)}x^{w\lambda}
%\]
%where $\Sigma$ acts on $\ZZ^n$ by permutation and for $v=(v_1,\ldots, v_n)\in\ZZ^n$,
%\[
%x^v:=x_1^{v_1}\cdots x_n^{v_n}
%\]
By (\ref{Demazure}), it is easily follows that $\SHH_n^0(\qqq,\ttt)_\LL$ preserves $\Lambda_n^\pm(\qqq,\ttt)_\LL$.
Because $\ee$ is idempotent, we have that this action is also faithful.
Let $\Lambda_n^\pm(\qqq,\ttt):= \Lambda_n^\pm(\qqq,\ttt)_\LL\otimes \KKK$.
\begin{thm}\label{MacThm}
We have the following:
\begin{enumerate}
\item For each partition $\lambda$, there exists a unique $P_\lambda(\qqq,\ttt)\in\Lambda_n^\pm(\qqq,\ttt)$ satisfying:
\begin{enumerate}
\item for all $f\in\Lambda_n^\pm(\qqq,\ttt)$, 
\[\mathfrak{r}\big(f(Y_1,\ldots, Y_n)\big)\cdot P_\lambda(\qqq,\ttt)= f(\qqq^{2\lambda_1}\ttt^{n-1}, \qqq^{2\lambda_2}t^{n-3},\ldots,\qqq^{2\lambda_n}t^{1-n})P_\lambda(\qqq,\ttt)\]
\item the coefficient of $x_1^{\lambda_1}\cdots x_n^{\lambda_n}$ is $1$.
\end{enumerate}
The collection $\{P_\lambda(\qqq,\ttt)\}$ forms a basis of $\Lambda_n^\pm(\qqq,\ttt)$.
\item The set of polynomials $\{P_\lambda(\qqq,\ttt)\}$ is defined over a localization of $\LL$ at a multiplicative set $M$ not containing $(\ttt-\qqq^k)$ for $k\ge 1$.
\end{enumerate}
\end{thm}
\noindent The functions $\{P_\lambda(\qqq,\ttt)\}$ are called the \textit{Macdonald polynomials}.
Part (1) is classical, and part (2) follows from the \textit{tableaux sum formula} for $P_\lambda(\qqq,\ttt)$ (see \cite{MacBook} for both).
From part (2), it follows that we can specialize the Macdonald polynomials at $\ttt=\qqq^k$ for $k\ge 1$.
By part (1), we have that $\{P_\lambda(\qqq,\qqq^k)\}$ forms a basis of $\Lambda_n^\pm(\qqq):=\CC(\qqq)[x_1,\ldots, x_n]^{\Sigma_n}$.

%\subsubsection{Gaussian}\label{Gaussian}
%We will make frequent use of the following element:
%\begin{equation}
%\gamma:=\frac{24\ttt^n\log(\ttt)^2}{n(n^2-1)\log(\qqq)}\exp\left( \sum_{i=1}^n\frac{\log(Y_i)^2}{\log(\qqq)} \right)
%\label{GaussForm}
%\end{equation}
%It is possible to make sense of $\gamma$ in a suitable completion of $\HH_n^0(\qqq,\ttt)$.
%%On the other hand, it has a nice action on $\Lambda_n^\pm(\qqq,\ttt)$, and thus we can view it as an operator in $\mathrm{End}(\Lambda^\pm(\qqq,\ttt))$:
%\begin{prop}\cite{DiKedqt}
%We have the following:
%\begin{enumerate}
%\item The action of $\gamma$ on $\Lambda_n^\pm(\qqq,\ttt)$ is given by
%\begin{equation}
%\rrr(\gamma)\cdot P_\lambda(\qqq,\ttt)=\prod_{i=1}^n\qqq^{\lambda_i^2}\ttt^{(n-2i)\lambda_i}P_\lambda(\qqq,\ttt)
%\label{GaussEigen}
%\end{equation}
%\item The adjoint action of $\gamma$ on $\HH_n^0(\qqq,\ttt)$ is given by:
%\begin{align*}
%\gamma T_i\gamma^{-1}&= T_i, & \gamma X_1\cdots X_i\gamma^{-1}&= \qqq^{-i}Y_1\cdots Y_iX_1\cdots X_i,&\gamma Y_i\gamma^{-1}&= Y_i
%\end{align*}
%\end{enumerate}
%\end{prop}
%\noindent Notice that we have no problem setting $\ttt=\qqq^k$ in (\ref{GaussForm}).
%The above results for $\gamma$ also hold for $k\ge 1$.
%
\subsection{Cyclotomic subalgebra}\label{Cyclo}
For $\ell\ge 0$, let $\mathsf{Z}_1,\ldots, \mathsf{Z}_\ell\in\CC^\times$ and set $\ZZZ=(\ZZZ_1,\ldots, \ZZZ_\ell)$.
When $\ell=0$, we treat $\ZZZ$ as an empty list.

\subsubsection{Cyclotomic $\qqq$-Dunkl operators}
The \textit{cyclotomic DAHA} of Braverman--Etingof--Finkelberg \cite{BEF} is a subalgebra of $\HH_n^0(\qqq,\ttt)$ that can be defined using certain \textit{$\qqq$-Dunkl operators}:
\begin{align*}
D_1&:= X_1^{-1}\left(\qqq^{-1}\ttt^{1-n}Y_1^{-1}-\ZZZ_1\right)\cdots\left(\qqq^{-1}\ttt^{1-n}Y_1^{-1}-\ZZZ_\ell\right)\\
D_i&:= T_{i-1}^{-1}\cdots T_1^{-1} D_1T_1^{-1}\cdots T_{i-1}^{-1}
\end{align*}
These elements can be obtained by conjugating $\{X_1^{-1},\ldots, X_n^{-1}\}$ by certain elements lying in a completing of $\HH_n^0(\qqq,\ttt)$.
Specifically, let
\begin{equation}
\begin{aligned}
\gamma_{\ZZZ_a}=\gamma_{\ZZZ_a}(Y_1,\ldots, Y_n;\qqq,\ttt)&:=\exp\left( \sum_{m>0}\sum_{i=1}^n\frac{(\qqq^{-1}\ttt^{1-n}\ZZZ_aY_i)^{-m}-\ttt^{m(n-2i)}}{(1-\qqq^{2m})m} \right)\\
\gamma_{\ZZZ}&:=\prod_{a=1}^\ell\gamma_{\ZZZ_a}
\end{aligned}
\label{GammaZ}
\end{equation}
We denote the \textit{$\qqq$-Pochhammer symbols} by
\begin{align*}
(x;\qqq)_\infty&:=\prod_{i=0}^\infty(1-\qqq^ix)\\
(x;\qqq)_m&:=\frac{(x;\qqq)_\infty}{(\qqq^mx;\qqq)_\infty}\hbox{ for }m\in\ZZ
\end{align*}
\begin{prop}\label{GammaZProp}
The following properties hold for $\gamma_\ZZZ$:
\begin{enumerate}
\item $\gamma_\ZZZ X_i^{-1}\gamma_\ZZZ^{-1}= (-1)^\ell(\ZZZ_1\cdots\ZZZ_\ell)^{-1} D_i$;
\item the action on the Macdonald polynomial $P_\lambda(\qqq,\ttt)$ is given by
\begin{equation}
\mathfrak{r}(\gamma_\ZZZ) \cdot P_\lambda(\qqq,\ttt)=\left( \prod_{a=1}^\ell\prod_{i=1}^n\left(\qqq^{-1}\ttt^{-2(n-i)}\ZZZ_a^{-1};\qqq^{2}\right)_{-\lambda_i} \right)P_\lambda(\qqq,\ttt)
\label{GammaEigen}
\end{equation}
\end{enumerate}
\end{prop}

\begin{proof}
Part (1) follows from the proof of Lemma 3.21 in \cite{BEF}.
Namely, one can check that
\[
\gamma_{\ZZZ}(X_1^{-1},\ldots,\qqq^{-1} X_i^{-1},\ldots, X_n;\qqq^{-1},\ttt^{-1})=(1-\qqq\ttt^{n-1}\ZZZ_i^{-1}X_i)\gamma_\ZZZ(X_1^{-1},\ldots, X_n^{-1};\qqq^{-1},\ttt^{-1})
\]
Part (2) is by direct calculation using Theorem \ref{MacThm}(1).
\end{proof}

\noindent The corollary below was proved in \cite{BEF}:
\begin{cor}
The elements $\{D_1,\ldots, D_n\}$ pairwise commute.
\end{cor}

\subsubsection{Definition of $\HH_n^\ell(\ZZZ,\qqq,\ttt)$ and PBW}
The cyclotomic DAHA $\HH_n^\ell(\ZZZ,\qqq,\ttt)$ is defined to be the subalgebra of $\HH_n^0(\qqq,\ttt)$ generated by
\[
\{T_i, Y_j^{\pm 1}, X_j, D_j\, |\, i=1,\ldots n-1\hbox{ and }j=1,\ldots n\}
\]
We emphasize that only $X_j$ is included, not $X_j^{-1}$.
The Dunkl elements $\{D_j\}$ can be defined in $\HH_n^0(\qqq,\ttt)_\LL$, and we define $\HH_n^\ell(\ZZZ,\qqq,\ttt)_\LL$ and $\HH_n^\ell(\ZZZ,\qqq,\qqq^k)$ analogously.
We likewise define $\HH_n^\ell(\ZZZ,\qqq,\ttt)^-$, $\HH_n^\ell(\ZZZ,\qqq,\ttt)_\LL^-$, and $\HH_n^\ell(\ZZZ,\qqq,\qqq^k)^-$ to be the subalgebras generated by
\[
\{T_i, Y_j^{-1}, X_j, D_j\, |\, i=1,\ldots n-1\hbox{ and }j=1,\ldots n\}
\]
i.e. without $\{Y_j\}$.
%Note that by the DAHA relations, the powers of $Y_1Y_2\cdots Y_n$ form an Ore set that yields $\{Y_j^{-1}\}$ upon localization.

\begin{rem}
In \cite{BEF}, their Dunkl operators are instead 
\begin{align*}
D_1^{\mathrm{BEF}}&:=X_1^{-1}(Y_1-\ZZZ_1)\cdots(Y_1-\ZZZ_\ell)\\
D_i^{\mathrm{BEF}}&= T_{i-1}^{-1}\cdots T_1^{-1} D_1^{\mathrm{BEF}}T_1^{-1}\cdots T_{i-1}^{-1}
\end{align*}
It follows that our $\HH_n^\ell(\ZZZ,\qqq,\ttt)$ coincides with theirs with rescaled and inverted $\ZZZ$-parameters.
%We have elected to define $D_1$ as we do due to the quantum moment map (cf. \ref{Moment}).
\end{rem}

\begin{prop}\label{CyclicPBW}
We have the following:
\begin{enumerate}
\item There exists a filtration on $\HH_n^\ell(\ZZZ,\qqq,\ttt)_\LL^-$ where
\[
\deg(X_i)=\ell,\, \deg(Y_i^{-1})=2,\, \deg(D_i)=\ell,\, \deg(T_i)=0
\]
\item $\HH_n^\ell(\ZZZ,\qqq,\ttt)^-_\LL$ has an $\LL$-basis given by elements of the form
\[
M_DM_XM_YT_w
\]
where:
\begin{itemize}
\item $M_D$ is a monomial in the $\{D_i\}$;
\item $M_X$ is a monomial in the $\{X_i\}$;
\item $M_Y$ is a monomial in the $\{Y_i\}$ where every variable has exponent $\le 0$ and $>-\ell$;
\item $w\in \Sigma_n$ and $T_w$ is defined as in \ref{PBW}.
\end{itemize}
\end{enumerate}
\end{prop}

\begin{proof}
Linear independence follows from the analogous statement over $\KK$, which is proved in \cite{BEF}.
The relations used in \textit{loc. cit.} to show the spanning property are defined over $\LL$.
This yields a presentation for $\HH_n^\ell(\qqq,\ttt)_\LL^-$; the relations in \textit{loc. cit.} shows the existence of the filtration.
\end{proof}

Next, let us consider the action of $\ZZ/\ell\ZZ$ on $\CC[z,w]$ where the generator $\sigma\in\ZZ/\ell\ZZ$ acts by
\begin{equation*}
\begin{aligned}
\sigma\cdot z&= e^{2\pi \sqrt{-1}/\ell}z,&
\sigma\cdot w&= e^{-2\pi \sqrt{-1}/\ell}w
\end{aligned}
\end{equation*}
Taking the tensor product, we obtain an action of $(\ZZ/\ell\ZZ)^n$ on $\CC[z_1,\ldots, z_n,w_1,\ldots, w_n]$.
We set
\begin{equation}
\mathcal{R}_\ell:=\CC[z_1,\ldots, z_n,w_1,\ldots, w_n]^{(\ZZ/\ell\ZZ)^n}
\label{RellDef}
\end{equation}

\begin{cor}\label{PBWCor}
$\HH_n^\ell(\ZZZ,\qqq,\ttt)_\LL^-$ is a flat deformation of a filtered algebra whose associated graded is isomorphic to $\mathcal{R}_\ell\rtimes\Sigma_n$.
\end{cor}

\begin{proof}
Setting $\qqq=\ttt=1$, the Hecke algebra generated by $\{T_i\}$ becomes the group algebra $\CC[\Sigma_n]$ while the elements $\{X_i\}\cup\{D_i\}\cup\{Y_i\}$ commute and satisfy
\begin{equation}
X_iD_i=(Y_i^{-1}-\ZZZ_1)\cdots(Y_i^{-1}-\ZZZ_\ell)
\label{NaiveCyc}
\end{equation}
In the associated graded of $\HH_n^\ell(\ZZZ,1,1)$ for the filtration from Proposition \ref{CyclicPBW}, (\ref{NaiveCyc}) becomes $X_iD_i=Y_i^\ell$.

Now consider $\mathcal{R}_\ell$.
If we denote the variables in the $i$th tensorand by $\{z_i, w_i\}$, then the map
\begin{equation}
\begin{aligned}
X_i&\mapsto z_i^\ell,  &
t^{1-n}Y_i^{-1}&\mapsto z_iw_i,&
D_i&\mapsto w_i^\ell, & 
T_i&\mapsto s_i
\end{aligned}
\label{DegenMap}
\end{equation}
defines a ring homomorphism $\mathrm{gr}\HH_n^\ell(\ZZZ,1,1)\rightarrow \mathcal{R}_\ell\rtimes\Sigma_n$.
It follows from Proposition \ref{CyclicPBW}, this homomorphism is an isomorphism.
\end{proof}

\subsubsection{Spherical generators}
The spherical subalgebra is defined to be $\SHH_n^\ell(\ZZZ,\qqq,\ttt):=\ee\HH_n^\ell(\ZZZ,\qqq,\ttt)\ee$.
We similarly define $\SHH_n^\ell(\ZZZ,\qqq,\ttt)_\LL$ and $\SHH_n^\ell(\ZZZ,\qqq,\qqq^k)$ as well as the `minus' versions $\SHH_n^\ell(\ZZZ,\qqq,\ttt)^-$, $\SHH_n^\ell(\ZZZ,\qqq,\ttt)_\LL^-$, and $\SHH_n^\ell(\ZZZ,\qqq,\qqq^k)^-$.
These minus subalgebras inherit the filtration from Proposition \ref{CyclicPBW} because $\deg(T_i)=0$.

\begin{lem}\label{CycGen}
$\SHH_n^\ell(\ZZZ,\qqq,\ttt)$ is generated by:
\[
\ee\KK[X_1,\ldots, X_n]\ee\cup\ee\KK[Y_1^{\pm 1},\ldots, Y_n^{\pm 1}]\ee\cup\ee\KK[D_1,\ldots, D_n]\ee
\]
For $\ell>0$ and $k>2n\ell$, $\SHH_n^\ell(\ZZZ,\qqq,\qqq^k)$ is generated by:
\[
\ee\CC(\qqq)[X_1,\ldots, X_n]\ee\cup\ee\CC(\qqq)[Y_1^{\pm 1},\ldots, Y_n^{\pm 1}]\ee\cup\ee\CC(\qqq)[D_1,\ldots, D_n]\ee
\]
\end{lem}

\begin{proof}
Because $\SHH_n^\ell(\ZZZ,\qqq,\ttt)^-_\LL$ is a direct summand of $\HH_n^\ell(\ZZZ,\qqq,\ttt)_\LL^-$, it follows from Corollary \ref{PBWCor} that $\SHH_n^\ell(\ZZZ,\qqq,\ttt)_\LL^-$ is a flat deformation of a filtered algebra whose associated graded is isomorphic to 
\[
\mathcal{R}_\ell^{\Sigma_n}=\CC[z_1,\ldots, z_n,w_1,\ldots, w_n]^{\Sigma_n\wr\ZZ/\ell\ZZ}
\]
By a theorem of Weyl \cite{WeylInv} (cf. \cite{GanChev}), $\mathcal{R}_\ell^{\Sigma_n}$ is generated by power sums
\begin{equation}
p_{a,k,b}:=\sum_{i=1}^n z_i^{k+a\ell}w_i^{k+b\ell}
\label{CycPow}
\end{equation}
where $a,k,b\in\ZZ_{\ge 0}$ and $0\le k<\ell$.
Under the $(\qqq,\ttt)\mapsto (1,1)$ degeneration followed by taking the associated graded and applying the map (\ref{DegenMap}), the element
\begin{equation}
P_{a,k,b}:=\ee\left( \sum_{i=1}^n X_i^aY_i^{-k}D_{i}^b \right)\ee
\label{Pakb}
\end{equation}
is sent to $p_{a,k,b}$.
Applying Nakayama's lemma to each piece of the filtration (which has finite-rank), it follows that $\{P_{a,k,b}\}$ generates $\SHH_n^\ell(\ZZZ,\qqq,\ttt)^-$ and $\SHH_n^\ell(\ZZZ,\qqq,\qqq^k)^-$.

To prove the lemma, the steps are similar to the proof of Lemma \ref{GenLem0}.
Namely, 
\begin{enumerate}
\item for $\HH_n^\ell(\ZZZ,\qqq,\ttt)^-$, we consider $\HH_n^\ell(\ZZZ,\qqq,1)_\LL^-$;
\item for $\HH_n^\ell(\ZZZ,\qqq,\qqq^k)^-$, we consider $\HH_n^\ell(\ZZZ,e^{2\pi i/k}, 1)^-$.
\end{enumerate}
After these specializations, we take the associated graded, wherein we have the relations
\begin{align*}
[Y_i^{-1},X_i]&= (\qqq^{-2}-1)X_iY_i^{-1}\\
[D_i,Y_i^{-1}]&= (\qqq^{-2}-1)Y_i^{-1}D_i\\
[D_i, X_i]&= (\qqq^{-2\ell}-1)X_i D_i
\end{align*}
The same steps as in the proof of Lemma A.15.2 in \cite{VarVassRoot} allow us to produce every $P_{a,k,b}$ at these specializations, and the lemma follows from applying Nakayama's lemma to each filtered piece.
\end{proof}

\section{Quantum groups}
The following is an abbreviated recap of Section 3 from our previous work \cite{QHarish}.

\subsection{Quantized enveloping algebra}
In this subsection, we review the algebra $\UU:=U_\qqq(\gl_n)$ and properties of its category of finite-dimensional representations.

\subsubsection{Definition}
First, let us establish some basic notation for roots and weights of $\gl_n$:
\begin{itemize}
\item Let $\{\epsilon_i\}\subset \RR^n$ be the coordinate basis and let $P\cong\ZZ^n$ be the lattice spanned by $\{\epsilon_i\}$.
\item The set $R:=\{\epsilon_i-\epsilon_j\}$ is the set of \textit{roots}, $R^+:=\{\epsilon_i-\epsilon\}_{i<j}\subset R$ is the set of \textit{positive} roots, and the adjacent difference $\alpha_i:=\epsilon_i-\epsilon_{i+1}\subset R^+$ is the \textit{$i$th simple root}.
The sublattice $Q\subset \RR^n$ spanned by $R$ is the \textit{root lattice}.
\item Under the standard pairing $\langle -,-\rangle$ on $\RR^n$ whereby $\{\epsilon_i\}$ forms an orthonormal basis, we identify $P$ with the \textit{weight lattice} (for $\gl_n$).
\item Note that $P$ has another basis $\{\omega_i\}$, where:
\[
\omega_i=\epsilon_1+\cdots+\epsilon_i
\]
We call $\omega_i$ the \textit{$i$th fundamental weight}.
\item We call $\lambda\in P$ \textit{dominant} if it pairs nonnegatively with all elements of $R_+$.
Let $P^+\subset P$ be the subset of dominant weights.
\item We denote by $\delta$ the \textit{staircase partition}:
\[
\delta:=(n-1, n-2,\ldots, 0)\in Q^+
\]
A somewhat similar element is the \textit{Weyl vector}
\[
\rho:=\left( \frac{n-1}{2}, \frac{n-3}{2},\ldots, -\frac{n-1}{2} \right)=\delta-\frac{n-1}{2}\left( 1,\ldots, 1 \right)
\]
Both satisfy $\langle \delta,\alpha_i\rangle=\langle\rho,\alpha_i\rangle=1$ for all $i$.
%pairs integrally with $R^+$, but it does not lie in $P$ if $n$ is even.
%We denote by $P_\rho$ the lattice generated by $P$ and $\rho$ and define $P_\rho^+\subset P_\rho$ to be the subset of elements that pair nonnegatively with $R^+$.
\end{itemize}

The \textit{quantized enveloping algebra} $\UU=U_\qqq(\gl_n)$ is the $\CC(\qqq)$-algebra with generators
\[
\{E_i, F_i, \qqq^h\, |\, i=1,\ldots, n-1\hbox{ and }h\in P\}
\]
and relations
\begin{gather*}
\qqq^{\vec{0}}=1,\,\qqq^{h_1}\qqq^{h_2}=\qqq^{h_1+h_2},\\
\qqq^hE_i\qqq^{-h}=\qqq^{\langle\alpha_i,h\rangle}E_i,\,\qqq^hF_i\qqq^{-h}=\qqq^{-\langle\alpha_i,h\rangle}F_i,\\
E_iF_j-F_jE_i=\delta_{ij}\dfrac{\qqq^{\alpha_i}-\qqq^{-\alpha_i}}{\qqq-\qqq^{-1}},\\
E_i^2E_{i+1}-(\qqq+\qqq^{-1})E_iE_{i+1}E_i+E_iE_{i+1}^2=0,\\
F_i^2F_{i+1}-(\qqq+\qqq^{-1})F_iF_{i+1}F_i+F_iF_{i+1}^2=0,
\end{gather*}
where $\delta_{ij}$ is the Kronecker delta.
$\UU$ is a Hopf algebra with the coproduct $\Delta_\UU$, counit $\epsilon_\UU$, and antipode $\iota_\UU$ given by
\begin{gather*}
\Delta_\UU(E_i)=E_i\otimes \qqq^{\alpha_i}+1\otimes E_i,\hspace{.5cm}\Delta_\UU(F_i)=F_i\otimes 1+\qqq^{-\alpha_i}\otimes F_i,\hspace{.5cm}\Delta_{\UU}(\qqq^h)=\qqq^h\otimes \qqq^h;\\
\epsilon_\UU(E_i)=\epsilon_\UU(F_i)=0,\hspace{.5cm}\epsilon_\UU(\qqq^h)=1;\\
\iota_\UU(E_i)=-E_i\qqq^{-\alpha_i},\hspace{.5cm}\iota_\UU(F_i)=-\qqq^{\alpha_i}F_i,\hspace{.5cm}\iota_\UU(\qqq^h)=\qqq^{-h}.
\end{gather*}
We will use \textit{Sweedler notation} to denote coproducts:
\[\Delta_\UU(x)=x_{(1)}\otimes x_{(2)}\]
With this notation, the \textit{adjoint action} of $\UU$ on itself can be written as:
\[
\ad_x(y)=x_{(1)}y\iota_\UU(x_{(2)})
\]
Finally, we note here that for $x\in\UU$,
\begin{equation}
\iota_\UU^2(x)=\qqq^{2\rho}x\qqq^{-2\rho}
\label{AntiSquare}
\end{equation}

\subsubsection{R-matrix}\label{RMatrix}
The \textit{universal $R$-matrix} $\widetilde{\RRR}$ is an invertible element of a suitable completion of $\UU^{\otimes 2}$ that satisfies the following properties:
\begin{align}
\label{RMatrixCo1}
(\Delta_\UU\otimes 1)\widetilde{\RRR}&= \widetilde{\RRR}_{13}\widetilde{\RRR}_{23}\\
\label{RMatrixCo2}
(1\otimes\Delta_\UU)\widetilde{\RRR}&= \widetilde{\RRR}_{13}\widetilde{\RRR}_{12}\\
\label{RCoprod}
\widetilde{\RRR}(x_{(1)}\otimes x_{(2)})&= (x_{(2)}\otimes x_{(1)})\widetilde{\RRR}\hbox{ for }x\in \UU
\end{align}
Here, $\widetilde{\RRR}_{ab}$ means that the first tensorand of $\widetilde{\RRR}$ is placed in the $a$th position and the second tensorand in the $b$th position.
Some consequences of these properties are:
\begin{gather}
\nonumber
(\epsilon_\UU\otimes 1)\widetilde{\RRR}=(1\otimes\epsilon_\UU)\widetilde{\RRR}=1\\
\nonumber
(\iota_\UU\otimes 1)\widetilde{\RRR}=(1\otimes \iota_\UU)\widetilde{\RRR}=\widetilde{\RRR}^{-1}\\
\label{QYBE}
\widetilde{\RRR}_{12}\widetilde{\RRR}_{13}\widetilde{\RRR}_{23}=\widetilde{\RRR}_{23}\widetilde{\RRR}_{13}\widetilde{\RRR}_{12}
\end{gather}
The final equation (\ref{QYBE}) is known as the \textit{quantum Yang-Baxter Equation}.
Finally, $\widetilde{\RRR}$ factors in the following manner:
\begin{equation}
\widetilde{\RRR}=\qqq^{-\sum \epsilon_i\otimes\epsilon_i}\left(1\otimes 1 +\widetilde{\RRR}^*\right)
\label{RFactor}
\end{equation}
where, if we let $\UU^+$ and $\UU^-$ be the subspaces of $\UU$ spanned by products of $\{E_i\}$ and $\{F_i\}$ respectively, then $\widetilde{\RRR}^*$ is an element of a suitable completion of $\UU^+\otimes\UU^-$.

The action of $\widetilde{\RRR}$ is well-defined on highest weight representations of $\UU$.
We will be particularly interested in its action on the \textit{vector representation} $\mathbb{V}\cong\CC(\qqq)^n$.
Let $\{e_1,\ldots, e_n\}$ be a basis of $\mathbb{V}$.
If we let $E_{ij}$ be the matrix unit
\[
E_{ij}e_k=\delta_{jk}e_i
\]
then the specialized $R$-matrix has the form
\[
\RRR':=\widetilde{\RRR}\big|_{\mathbb{V}\otimes\mathbb{V}}
=\qqq\sum_i E_{ii}\otimes E_{ii}+\sum_{i\not=j}E_{ii}\otimes E_{jj}+(\qqq-\qqq^{-1})\sum_{i<j}E_{ji}\otimes E_{ij}
\]
We will use more often the transpose of this matrix:
\[
\RRR=\qqq\sum_i E_{ii}\otimes E_{ii}+\sum_{i\not=j}E_{ii}\otimes E_{jj}+(\qqq-\qqq^{-1})\sum_{i>j}E_{ji}\otimes E_{ij}
\]
Let $\tau:\mathbb{V}\otimes\mathbb{V}\rightarrow\mathbb{V}\otimes\mathbb{V}$ be the tensor swap.
$\RRR$ then satisfies the \textit{Hecke condition}:
\begin{equation}
\tau\RRR-\RRR^{-1}\tau=\qqq-\qqq^{-1}
\label{HeckeCond}
\end{equation}

Finally, we will also make use of the \textit{Drinfeld element} $u$ and the \textit{ribbon element} $\nu$.
To define $u$, it is helpful to introduce Einstein notation for $\widetilde{\RRR}$:
\[
\widetilde{\RRR}=\sum_{s}{}_s r\otimes r_s=:{}_s r\otimes r_s
\]
We then set $u:=\iota_\UU(r_s){}_sr$.
In \cite{DrinAlmost}, it was shown that
\[\iota_\UU^2(x)=u xu^{-1}\]
As a consequence of (\ref{AntiSquare}), we have that
\[
\nu:=(\qqq^{-2\rho}u)^{-1}=u^{-1}\qqq^{2\rho}
\]
is central.

\subsubsection{Representations}\label{Representations}
Let $\CCC$ be the category of finite-dimensional $\UU$-modules, and let $\CCC_\ZZ\subset\CCC$ be the subcategory generated by irreducible modules with highest weight in $P^+$.
For $\lambda\in P^+$, we denote the corresponding irreducible module by $V_\lambda\in\CCC_\ZZ$.
By \cite{ReshTurRibbon}, $\CCC$ and $\CCC_\ZZ$ both have the structure of a \textit{ribbon category}, and this gives us a graphical calculus for working with matrix elements of representations.
Here, we review the structures needed in this paper---more details can be found in \cite{BakKir}.
\begin{enumerate}
\item \textit{Tensor and dual}: 
A morphism $f: V\rightarrow W$ will be denoted by an upward pointing arrow from $V$ to $W$ with a coupon labeled by $f$.
We will drop the coupon when denoting the identity morphism, and the identity morphism of $V^*$ will be denoted using a \textit{downward} point arrow from $V$ to $V$.
Tensor product of objects and morphisms will be denoted using horizontal juxtaposition.
For instance, the diagram below denotes $f\otimes g\otimes \mathrm{id}: U_1\otimes V_1\otimes W^*\rightarrow U_2\otimes V_2\otimes W^*$:
\[
\includegraphics{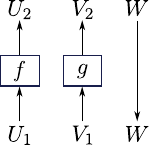}
\]
By (\ref{AntiSquare}), the action of $\qqq^{2\rho}$ identifies $V\cong V^{**}$, and since $\Delta_\UU(\qqq^{2\rho})=\qqq^{2\rho}\otimes\qqq^{2\rho}$, this identification respects tensor products.
\item \textit{Characters}:
For $\mathsf{C}\in \CC(\qqq)^\times$, $\UU$ has a character $\chi_\mathsf{C}$ where
\begin{equation*}
\begin{aligned}
\chi_\mathsf{C}(E_i)&= 0, & \chi_\mathsf{C}(F_i)&= 0, & \chi_\mathsf{C}(\qqq^{h})&= \mathsf{C}^{\langle h,\omega_n\rangle}
\end{aligned}
\end{equation*}
We denote by $\mathbb{1}_\mathsf{C}$ the corresponding representation.
Observe that $\mathbb{1}:=\mathbb{1}_1$ is the trivial representation while $\mathbb{1}_\qqq$ is the \textit{quantum determinant} representation.
\item \textit{(Quantum) Co/evaluations}:
For $V\in\CCC$, let $\left\{ v_i \right\}$ be a basis of $V$ with corresponding dual basis $\left\{ v^i \right\}$.
The canonical maps
\begin{align*}
\mathbb{1}\ni a&\mapsto av_i\otimes v^i\in V\otimes V^*\\
V^*\otimes V\ni v^i\otimes v_j&\mapsto \delta_{ij}\in\mathbb{1}
\end{align*}
are homomorphisms $\mathrm{coev}_V:\mathbb{1}\rightarrow V\otimes V^*$ and $\mathrm{ev}_V:V^*\otimes V\rightarrow\mathbb{1}$, respectively.
Graphically, we will omit depicting $\mathbb{1}$; $\mathrm{ev}_V$ and $\mathrm{coev}_V$ appear as caps and cups oriented towards the left:
\[
\includegraphics{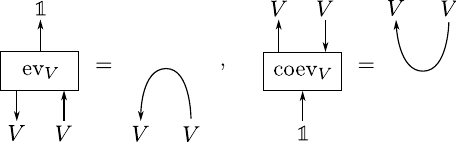}
\]
Using $\mathrm{ev}_W$ and $\mathrm{coev_V}$, we can define for $f:V\rightarrow W$ an adjoint $f^*:W^*\rightarrow V^*$.

%The ordering of tensor factors matter in $\mathrm{ev}_V$ and $\mathrm{coev}_V$.
To define maps with the opposite ordering, we will use $\qqq^{-2\rho}:V^{**}\rightarrow V$:
\begin{align*}
\mathbb{1}\ni a&\mapsto av^i\otimes \qqq^{-2\rho}v_i\in V^*\otimes V\\
V\otimes V^*\ni \qqq^{-2\rho}v_i\otimes v^j&\mapsto \delta_{ij}\in\mathbb{1}
\end{align*}
We denote these maps by $\mathrm{qcoev}_V$ and $\mathrm{qev}_V$, respectively.
Graphically, they will be depicted as cups and caps with orientations opposite from before:
\[
\includegraphics{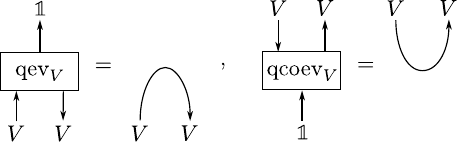}
\]
\item \textit{Braiding}: Recall the universal $R$-matrix $\widetilde{\RRR}$ from \ref{RMatrix}.
Let $\tau_{V,W}:V\otimes W\rightarrow W\otimes V$ be the tensor swap.
By (\ref{RCoprod}), $\beta_{V,W}:=\tau_{V,W}\widetilde{\RRR}: V\otimes W\rightarrow W\otimes V$ is a $\UU$-module isomorphism, and in fact it endows $\mathfrak{\CCC}$ with the structure of a \textit{braided monoidal category}.
We depict $\beta_{V,W}$ and $\beta_{W,V}$ by braid crossings:
\[
\includegraphics{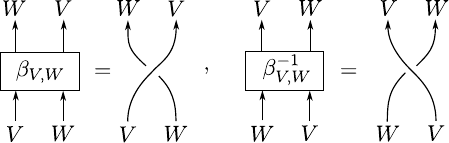}
\]
\item \textit{Ribbon structure}: The action of the ribbon element $\nu$ gives a $\UU$-endomorphism of $V$, which can be drawn in two ways:
\[
\includegraphics{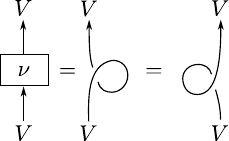}
\]
In \cite{DrinAlmost}, it was computed by Drinfeld that
\begin{equation}
\nu\big|_{V_\lambda}=\qqq^{\langle\lambda,\lambda+2\rho\rangle}
\label{RibbonAct}
\end{equation}
\end{enumerate}

Drawing these decorated strands in $\RR^3$ as emanating upwards from a fixed horizontal axis, we can enhance them into ribbons with a front and back side.
We require the front side to always face the reader at the start and end of the ribbon---the back side may appear in the middle due to the loop given by $\nu$, in which case the ribbon is twisted twice.
The graphical calculus assigns to each morphism in $\CCC$ such a \textit{$\CCC$-colored ribbon tangle}.
\begin{thm}[\cite{ReshTurRibbon}]
A morphism in $\CCC$ only depends on the isotopy type of its associated $\CCC$-colored ribbon tangle.
\end{thm}
\noindent In practice, we will work with strands instead of ribbons but keep track of loops coming from $\nu$.

\subsection{Reflection equation algebra}
Whereas $\UU$ deforms the universal enveloping algebra, we will also make use of a $\UU$-equivariant quantization $\OO_G$ of the ring of functions on $G=GL_n$.
This algebra has both a categorical description as well as one by generators and relations.

\subsubsection{Braided Tannakian reconstruction}\label{OGDef}
We review here the construction of $\OO_G$ via a braided analogue of Tannakian reconstruction, defined by Majid \cite{MajidBraided}.
As a $\CC(\qqq)$-vector space, $\OO_G$ is defined as the \textit{coend} of $\CCC_\ZZ$:
\begin{equation}
\begin{aligned}
\OO_G&:=\left(\bigoplus_{V\in\CCC_\ZZ} V^*\otimes V\right)\bigg/
\left\langle
f^*(v^*)\otimes w-v^*\otimes f(w)\,\middle|\,
\begin{array}{l}
v^*\in V^*,\, w\in W,\\
f\in\mathrm{Hom}_\UU(W,V)
\end{array}
\right\rangle\\
&\cong \bigoplus_{\lambda\in P^+}V_\lambda^*\otimes V_\lambda
\end{aligned}
\label{Coend}
\end{equation}
$\OO_G$ is evidently a $\UU$-module.
Using the categorical structures of $\CCC_\ZZ$ reviewed in \ref{Representations}, we can endow $\OO_G$ with a $\UU$-equivariant Hopf algebra structure:
\begin{itemize}
\item \textit{Algebra structure:}
For $v^*\otimes v\in V^*\otimes V$ and $w^*\otimes w\in W^*\otimes W$,
\[
m_{\OO_G}(v^*\otimes v\otimes w^*\otimes w) = r_tr_sw^*\otimes {}_trv^*\otimes{}_srv\otimes w
\]
\[
\includegraphics{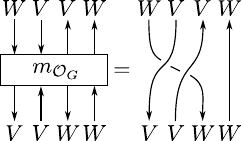}
\]
The inclusion $\mathbb{1}\rightarrow\mathbb{1}^*\otimes\mathbb{1}\in\OO_G$ provides the unit.
\item \textit{Coalgebra structure:}
For $v^*\otimes v\in V^*\otimes V$, we define the coproduct $\Delta_{\OO_G}(v^*\otimes v)$ as:
\[
\Delta_{\OO_G}(v^*\otimes v)=v^*\otimes \mathrm{coev}_V(1)\otimes v
\] 
\[
\includegraphics{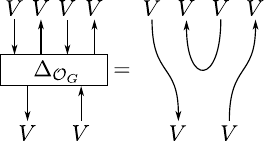}
\]
The evaluation map $\mathrm{ev}_V$ on $V^*\otimes V$ yields the counit $\epsilon_{\OO_G}$.
\item \textit{Antipode:} The antipode $\iota_{\OO_G}$ is given by:
\[
\iota_{\OO_G}(v^*\otimes v)= \nu r_s v\otimes {}_srv^*
\]
\[
\includegraphics{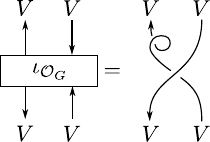}
\]
\end{itemize}

\subsubsection{Presentation}
$\OO_G$ is also a localization of the \textit{reflection equation algebra}.
To define the latter, recall the matrix units $\{E_{ij}\}$, the vector representation $\mathbb{V}$, and the specialized $R$-matrix $\RRR$ from \ref{RMatrix}.
Let $\{m_{ij}\}_{1\le i,j\le n}$ be a set of generators.
We place them in a matrix $\mathbf{M}$ as follows:
\[
\mathbf{M}=\sum_{i,j}E_{ij}\otimes m_{ij}
\]
For relations involving such matrices, we will only use subscripts to denote the tensorand wherein we place the matrix units---all remaining factors are assumed to be placed in the final tensorand.
\begin{thm}\cite{DonMudRET}
Let $\mathcal{R}\hbox{\it ef}$ be the $\CC(\qqq)$-algebra generated by the entries of $\mathbf{M}$ subject to the relations:
\[
\RRR_{21}\mathbf{M}_{1}\RRR\mathbf{M}_{2}=\mathbf{M}_{2}\RRR_{21}\mathbf{M}_{1}\RRR
\]
We have the following:
\begin{enumerate}
\item Let $\{e_i\}$ be a basis of $\mathbb{V}$ with dual basis $\{e^i\}$.
The map $\varphi:\mathcal{R}\hbox{\it ef}\rightarrow\OO_G$ induced by $m_{ij}\mapsto e^i\otimes e_j$ is an injective algebra homomorphism.
\item There exists a central element $\det_{\qqq}(\mathbf{M})\in\mathcal{R}\hbox{\it ef}$ such that 
\[\varphi(\textstyle\det_{\qqq}(\mathbf{M}))=\mathrm{qcoev}_{\mathbb{1}_\qqq}(1)\]
The map $\varphi$ becomes an isomorphism upon inverting $\det_{\qqq}(\mathbf{M})$.
\end{enumerate}
\end{thm}

$\mathcal{R}\hbox{\it ef}$ is called the \textit{reflection equation algebra}.
We will abuse notation and view $\mathbf{M}$ as a matrix of elements in $\OO_G$.
By the antipode condition, we have that $\mathbf{M}^{-1}:=(1\otimes \iota_{\OO_G})(\mathbf{M})$ does indeed satisfy $\mathbf{M}\mathbf{M}^{-1}=\mathbf{M}^{-1}\mathbf{M}=\mathbf{I}$.
The element $\det_\qqq(\mathbf{M})$ is called the \textit{quantum determinant}.
An explicit formula for it is given in \cite{JorWhite}.
We will be content with its categorical interpretation in terms of $\mathrm{qcoev}_{\mathbb{1}_\qqq}$.
For instance, from this interpretation, it is easy to see that
\[
\textstyle\iota_{\OO_G}(\det_\qqq(\mathbf{M}))=\det_\qqq(\mathbf{M})^{-1}
\]
We will denote by $\det_\qqq(\mathbf{M}^{-1}):=\iota_{\OO_G}(\det_\qqq(\mathbf{M}))$.
Note that this is not necessarily the Jordan--White formula applied to $\mathbf{M}^{-1}$.

\subsubsection{Joseph--Letzter--Rosso homomorphism}
To $v^*\otimes v\in V^*\otimes V$, we can associate a functional on $\UU$ by
\[
(v^*\otimes v)(x)=v^*(xv)\hbox{ for }x\in\UU
\]
The following map can be viewed as assigning a dual element via a quantum Killing form:
\begin{thm}\cite{JosLetztLocal,RossoKilling}
The maps $\kappa,\bar{\kappa}:\OO_G\rightarrow \mathcal{U}$ given by
\begin{align*}
\kappa(v^*\otimes v)&=((v^*\otimes v)(-)\otimes 1)\widetilde{\mathbf{R}}_{21}\widetilde{\mathbf{R}}\\
\bar{\kappa}(v^*\otimes v)&=((v^*\otimes v)(-)\otimes 1)\widetilde{\mathbf{R}}^{-1}\widetilde{\mathbf{R}}^{-1}_{21}
\end{align*}
are $\mathcal{U}$-equivariant algebra embeddings (where $\mathcal{U}$ is endowed with the adjoint action).
Moreover, the center of $\UU$ is $\kappa(\OO_G^\UU)=\bar{\kappa}(\OO_G^\UU)$.
\end{thm}

The map $\kappa$ does not respect coproducts, but the discrepancy can be pinned down precisely:
\begin{prop}[\cite{VarVassRoot}]\label{Coideal}
Let $\left\{ v_i \right\}\subset V$ and $\left\{ v^i \right\}\subset V^*$ be dual bases.
For $v^*\otimes v\in V^*\otimes V$, we have
\begin{equation}
\label{KappaCo}
(\Delta_\UU\circ\kappa)(v^*\otimes v)= \kappa(v^*\otimes v_i) r_sr_t\otimes\kappa({}_srv^i\otimes{}_trv)
\end{equation}
In particular, $\kappa(\OO_G)$ is a left coideal sublagebra of $\mathcal{U}$, i.e. 
\begin{align*}
%\Delta(\kappa(\mathfrak{R}))&\subset \mathcal{U}\otimes\kappa(\mathfrak{R})\\
(\Delta_\UU\circ\kappa)(\OO_G)&\subset \mathcal{U}\otimes\kappa(\OO_G)
\end{align*}
\end{prop}
\noindent We end by making a few observations:
\begin{enumerate}
\item Note that the map $V^*\otimes V\otimes W\rightarrow W$ given by
\begin{equation*}
v^*\otimes v\otimes w\mapsto \kappa(v^*\otimes v)(w)
\end{equation*}
is a $\UU$-morphism depicted by the following diagram:
\begin{equation}
\includegraphics{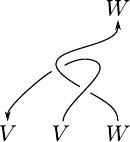}
\label{KappaPic}
\end{equation}
\item The $\Delta_\UU$ formula (\ref{KappaCo}) is the result of taking (\ref{KappaPic}) with two strands ($\UU$-modules) passing through and making bunny ears to force the appearance of $\Delta_{\OO_G}$:
\[
\includegraphics{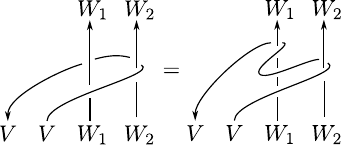}
\]
\item From (\ref{RFactor}), it is easy to deduce that
\[
\textstyle\kappa(\det_\qqq(\mathbf{M}))=\kappa(\mathrm{qcoev}_{\mathbb{1}_\qqq}(1))=\qqq^{-2\omega_n}
\]
\end{enumerate}

%We depict $\kappa|_{V^*\otimes V}$ by replacing the strand for $W$ with a dotted strand, oriented upwards.
%Isotopy of the dotted strand leads to the same map into $\UU$ because isotopic diagrams act the same way on $W$ for all $W\in\CCC$.

\section{Quantum differential operators}\label{QDO}
In this section, we review and study the ring of quantum differential operators on the matrix space corresponding to the following framed cyclic quiver:
\begin{equation}
\includegraphics{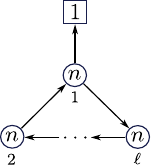}
\label{CyclicQ}
\end{equation}
Here, the numbers within the squares and circle denote the dimensions of the vector spaces we will assign to them, and the subscripts $1,\ldots,\ell$ are the labelings of the vertices, which we will view as elements of $\ZZ/\ell\ZZ$ for $\ell\ge 1$.
In the case of $\ell=1$, we will have a loop from vertex $1$ back to itself.
We will define this algebra piecemeal, addressing the round vertices first before the square one.
The latter corresponds to the \textit{quantum Weyl algebra}.
Unfortunately, we will need to split the former into the cases $\ell>1$ and $\ell=1$.

\subsection{The case $\mathbf{\ell>1}$}
Here, we will work with the algebras defined by Jordan \cite{JordanMult} for the subquiver of (\ref{CyclicQ}) consisting of round vertices.

\subsubsection{Braided coordinate ring $\OO_\ell$}\label{Oell}
The quantized function algebra $\OO_\ell$ has generators 
\[\left\{x_{ij}^{(a)}\, \middle| \, 
\begin{array}{l}
1\le i, j\le n\\
a\in\ZZ/\ell\ZZ
\end{array}
\right\}\]
We put the generators with superscript index $a$ into the matrix
\[
\XX^{(a)}:=\sum_{i,j} E_{ij}\otimes x_{ij}^{(a)}
\]
%which we view as coordinate functions on the space of matrices mapping from vertex $a+1$ to $a$.
For $\ell\ge 3$, its relations are as follows (recall that subscripts only denote where we place the matrix units):
\begin{align}
\label{Xaa}
\RRR_{21}\XX_1^{(a)}\XX_2^{(a)}&= \XX_2^{(a)}\XX_1^{(a)}\RRR;\\
%\XX_1^{(a)}\XX_2^{(a+1)}&= \XX_2^{(a+1)}\RRR^{-1}\XX_1^{(a)};\\
\label{XXad}
\XX_1^{(a)}\XX_2^{(a-1)}&=\XX_2^{(a-1)}\RRR_{21}\XX_1^{(a)};\\
\label{XXcom}
\XX_1^{(a)}\XX_2^{(b)}&= \XX_2^{(b)}\XX_1^{(a)}\hbox{ for }b\not= a, a\pm 1
\end{align}
For $\ell=2$, we replace (\ref{XXad}) with:
\begin{equation}
\XX_1^{(2)}\RRR\XX_2^{(1)}=\XX_2^{(1)}\RRR_{21}\XX_1^{(2)} 
\label{XXad2}
\end{equation}

\subsubsection{Algebra of quantum differential operators $\DD_\ell$}\label{Dell}
The algebra $\DD_\ell$ of quantum differential operators on $\OO_\ell$ is obtained by adjoining additional generators
\[
\left\{ 
\del_{ij}^{(a)}\,\middle|
\begin{array}{l}
1\le i,j\le n\\
a\in\ZZ/\ell\ZZ
\end{array}
 \right\}
\]
to $\OO_\ell$.
As before, we put them into the matrices
\[
\DDD^{(a)}:=\sum_{i,j}E_{ij}\otimes\del_{ij}^{(a)}
\]
For $\ell \ge 3$, these new generators satisfy the relations:
\begin{align}
\label{Daa}
\RRR_{21}\DDD_{1}^{(a)}\DDD_{2}^{(a)}&= \DDD_{2}^{(a)}\DDD_1^{(a)}\RRR;\\
\label{DDadj}
\DDD_1^{(a)}\DDD_2^{(a+1)}&= \DDD_2^{(a+1)}\RRR^{-1}\DDD_1^{(a)};\\
\nonumber
\DDD_1^{(a)}\DDD_2^{(b)}&= \DDD_2^{(b)}\DDD_1^{(a)}\hbox{ for }b\not= a, a\pm 1;\\
\label{DXa}
\DDD_2^{(a)}\RRR^{-1}\XX_1^{(a)}&= \XX_1^{(a)}\RRR\DDD_2^{(a)}+(\qqq-\qqq^{-1})\mathbf{\Omega};\\
\label{DXadj1}
\DDD_1^{(a)}\XX_2^{(a+1)}&=\RRR\XX_2^{(a+1)}\DDD_1^{(a)};\\
\label{DXadj2}
\DDD_1^{(a)}\XX_2^{(a-1)}&=\XX_2^{(a-1)}\DDD_1^{(a)}\RRR_{21}^{-1};\\
\nonumber
\DDD_1^{(a)}\XX_2^{(b)}&= \XX_2^{(b)}\DDD_1^{(a)}\hbox{ for }b\not=a,a\pm 1;
\end{align}
where in (\ref{DXa}),
\[
\mathbf{\Omega}=\sum_{ij} E_{ij}\otimes E_{ji}\otimes 1
\]
For $\ell=2$, we replace (\ref{DDadj}), (\ref{DXadj1}), and (\ref{DXadj2}) with
\begin{align}
\nonumber
\DDD_1^{(1)}\RRR_{21}^{-1}\DDD_2^{(2)}&= \DDD_2^{(2)}\RRR^{-1}\DDD_1^{(1)};\\
\nonumber
\DDD_1^{(a)}\XX_2^{(a+1)}&= \RRR_{21}^{-1}\XX_2^{(a+1)}\DDD_1^{(a)}\RRR_{21}^{-1};
\end{align}

\begin{rem}
In comparison with \cite{JordanMult}, we have scaled the $\del$-generators of \textit{loc. cit.} by $(\qqq-\qqq^{-1})$.
This will simplify the formulas for the moment map in \ref{MomentEll}.
\end{rem}

\subsubsection{Equivariance}\label{DellEq}
The algebras $\OO_\ell$ and $\DD_\ell$ are modules over $\UU^{\otimes \ell}$.
Here, we will describe these module structures in more detail.
For a fixed $a\in\ZZ/\ell\ZZ$, let $\UU^{(a)}$ denote the $a$th tensorand of $\UU^{\otimes\ell}$ (we index the tensorands using $\ZZ/\ell\ZZ$), and let $\mathbb{V}_{(a)}$ denote the vector representation of $\UU^{(a)}$ with basis $\{e_i^{(a)}\}$ and dual basis $\{e^i_{(a)}\}$.
One should view $\UU^{(a)}$ as the copy of $\UU$ acting on the round vertex in (\ref{CyclicQ}) labeled by $a$.
Finally, let
\begin{equation*}
\begin{aligned}
\OO_\ell^{(a)}&:=\left\langle x_{ij}^{(a)}\,\middle|\, 1\le i,j\le n\right\rangle,&
\DD_\ell^{(a)}&:=\left\langle x_{ij}^{(a)}, \del_{ij}^{(a)}\,\middle|\, 1\le i,j\le n\right\rangle
\end{aligned}
\end{equation*}

We will view $n\times n$ matrices as elements of the \textit{exterior} tensor product of vector representations $\mathbb{V}\boxtimes\mathbb{V}^*$.
%Thus, we view it as a $\UU^{\otimes 2}$-representation, one for each tensorand.
The quantized ring of functions on matrices would then be a suitable quantization of the polynomial ring generated by $\mathbb{V}^*\boxtimes\mathbb{V}$, and likewise, functions on the dual space would be generated by $\mathbb{V}\boxtimes\mathbb{V}^*$.
In this way, we will coordinates of $\XX^{(a)}$ as functions on matrices from vertex $a+1$ to vertex $a$ and coordinates of $\DDD^{(a)}$ as analogous fuctions from vertex $a$ to vertex $a+1$.
To be precise, consider the following algebra homomorphisms out of the tensor algebras:
\begin{align*}
\Pi_\OO^{(a)}:\mathcal{T}\left(\mathbb{V}^*_{(a)}\boxtimes\mathbb{V}_{(a+1)}\right)&\rightarrow \OO_\ell^{(a)}&
\Pi_\DD^{(a)}:\mathcal{T}\left(\mathbb{V}^*_{(a)}\boxtimes\mathbb{V}_{(a+1)}\oplus\mathbb{V}_{(a)}\boxtimes\mathbb{V}_{(a+1)}^*\right)&\rightarrow \DD_\ell^{(a)}\\
e^i_{(a)}\boxtimes e_i^{(a+1)}&\mapsto x_{ij}^{(a)} & e^i_{(a)}\boxtimes e_i^{(a+1)}&\mapsto x_{ij}^{(a)}\\
&& e_i^{(a)}\boxtimes e^j_{(a+1)}&\mapsto \del_{ji}^{(a)}
\end{align*}
\begin{prop}[\cite{JordanMult}]
$\OO_\ell$ and $\DD_\ell$ are $\UU^{\otimes\ell}$-equivariant algebras.
\end{prop}

\begin{proof}
First consider the subalgebras $\OO_\ell^{(a)}$ and $\DD_\ell^{(a)}$.
The kernels of the algebra homomorphisms $\Pi_\OO^{(a)}$ and $\Pi_\DD^{(a)}$ correspond to the relations (\ref{Xaa}), (\ref{Daa}), and (\ref{DXa}).
These relations, in turn, can be expressed as images of $\UU_{(a)}\otimes\UU_{(a+1)}$-module homomorphisms into the respective tensor algebras (cf. 3.3.2 of \cite{JordanMult}).
It follows that $\Pi_\OO^{(a)}$ and $\Pi_\DD^{(a)}$ are in fact $\UU_{(a)}\otimes\UU_{(a+1)}$-module homomorphisms; thus, $\OO_\ell^{(a)}$ and $\DD_\ell^{(a)}$ are $\UU_{(a)}\otimes\UU_{(a+1)}$-equivariant. 

The remaining relations involve different $(a)$-indices, and they merely impose that $\OO_\ell$ and $\DD_\ell$ are \textit{cyclically-ordered} \textit{braided} tensor products of algebras $\{\OO_\ell^{(a)}\}$ and $\{\DD_\ell^{(a)}\}$, respectively.
Namely, in the case $\ell\ge 3$, an element $x^{(a-1)}$ with index $a-1$ and an element $y^{(a)}$ with index $a$ both share an action of $\UU^{(a)}$.
If we let $\widetilde{\RRR}^{(a)}={}_sr^{(a)}\otimes r^{(a)}_s$ denote the universal $R$-matrix of $\UU^{(a)}$, then we have imposed:
\begin{align}
\label{BraidedEll}
x^{(a-1)}y^{(a)}&= \left(r_s^{(a)}\cdot y^{(a)}\right)\left({_s}r^{(a)}\cdot x^{(a-1)}\right), & y^{(a)}x^{(a-1)}&= \left(\iota_\UU({_s}r^{(a)})\cdot x^{(a-1)}\right)\left(r_s^{(a)}\cdot y^{(a)}\right)
\end{align}
(the cyclicity of the order comes from viewing $a\in\ZZ/\ell\ZZ$).
For $\ell=2$, we have instead:
\begin{equation}
\begin{aligned}
x^{(1)}y^{(2)}&= \left(\big[\iota_{\UU}({}_tr^{(1)})\otimes r_s^{(2)}\big]\cdot y^{(2)}\right)\left( \big[r_{t}^{(1)}\otimes{}_sr^{(2)}]\cdot x^{(1)} \right)\\
y^{(2)}x^{(1)}&= \left(\big[r_s^{(1)}\otimes \iota_{\UU}({}_tr^{(2)})\big]\cdot x^{(1)}\right)\left( \big[r_{t}^{(1)}\otimes{}_sr^{(2)}]\cdot y^{(2)} \right)
\end{aligned}
\label{Braided2}
\end{equation}
Because these relations can be written in terms of the braiding and inverse braiding, respectively, $\OO_\ell$ and $\DD_\ell$ are indeed $\UU^{\otimes\ell}$-equivariant.
\end{proof}

\subsubsection{FRT quantum determinants}\label{FRTDet}
Each $\OO^{(a)}_\ell$ is isomorphic to the \textit{FRT algebra} (cf. Chapter 9 of \cite{KlimSchm}).
This algebra also has a notion of quantum determinant:
\begin{equation}
\begin{aligned}
\textstyle\det_\qqq(\XX^{(a)})&:=\sum_{w\in\Sigma_n}(-\qqq)^{l(w)}x_{w(1)1}^{(a)}\cdots x_{w(n)n}^{(a)}\\
&= \sum_{w\in\Sigma_n}(-\qqq)^{l(w)}x_{1w(1)}^{(a)}\cdots x_{nw(n)}^{(a)}
\end{aligned}
\label{FRTDetForm}
\end{equation}
Here, we will review the properties of $\det_\qqq(\XX^{(a)})$.
One of the distinguishing features of $\OO_\ell^{(a)}$ is that it possesses natural maps
\begin{align*}
\varphi_L&:=\mathrm{id}_{\mathbb{V}_{(a)}^*}\otimes \mathrm{coev}_{\mathbb{V}}: \mathbb{V}_{(a)}^*\rightarrow\OO_\ell^{(a)}\otimes\mathbb{V}_{(a)}^*\\
\varphi_R&:=\mathrm{coev}_{\mathbb{V}}\otimes\mathrm{id}_{\mathbb{V}_{(a+1)}}: \mathbb{V}_{(a+1)}\rightarrow\mathbb{V}_{(a+1)}\otimes\OO_\ell^{(a)}
\end{align*}
(these are in fact comodules with respect to a natural bialgebra structure on $\OO_\ell^{(a)}$, which we will not use).
Ignoring the distinction between $\UU_{(a)}$ and $\UU_{(a+1)}$, both maps are $\UU$-module homomorphism.
However, due to the trivial action on the image of $\mathrm{coev}_{\mathbb{V}}$, we have that $\varphi_L$ is a $\UU_{(a)}$-module homomorphism when we have $\UU_{(a)}$ act trivially on the target $\mathbb{V}_{(a)}^*$ tensorand.
Likewise, $\varphi_R$ is a $\UU_{(a+1)}$-morphism when we ignore the $\UU_{(a+1)}$-action on the target $\mathbb{V}_{(a+1)}$-tensorand.

Because $\RRR$ satisfies the Hecke relation (\ref{HeckeCond}), we can define the quantum exterior algebras $\wedge_\qqq\mathbb{V}$ and $\wedge_\qqq\mathbb{V}^*$ by performing Hecke antisymmetrizations; these are $\UU$-equivariant algebras.
The characters $\mathbb{1}_{\qqq}$ and $\mathbb{1}_{\qqq^{-1}}$ are their respective top degree summands.
A key feature of $\varphi_L$ and $\varphi_R$ is that they extend to $\UU_{(a)}$- and $\UU_{(a+1)}$-equivariant algebra homomorphisms 
\begin{align*}
\varphi_L^\wedge:&\wedge_\qqq\mathbb{V}_{(a)}^*\rightarrow\OO_\ell^{(a)}\otimes\wedge_\qqq\mathbb{V}_{(a)}\\
\varphi_R^\wedge:&\wedge_\qqq\mathbb{V}_{(a+1)}\rightarrow\wedge_\qqq\mathbb{V}_{(a+1)}\otimes\OO_\ell^{(a)}
\end{align*}
respectively.
Here, as with before, when discussing equivariance, we only consider the $\UU_{(a)}$- and $\UU_{(a+1)}$-actions on the $\OO_\ell^{(a)}$ tensorand in the target.
Let $\mathbb{1}_{\qqq^{-1}}^{(a)}$ and $\mathbb{1}_{\qqq}^{(a+1)}$ denote the characters viewed as $\UU_{(a)}$- and $\UU_{(a+1)}$-modules, respectively.

\begin{prop}\label{FRTDetProp}
The FRT quantum determinant $\det_\qqq(\XX^{(a)})$ transforms under the $\UU_{(a)}\otimes\UU_{(a+1)}$-action as a nonzero element of $\mathbb{1}_{\qqq^{-1}}^{(a)}\boxtimes\mathbb{1}_{\qqq}^{(a+1)}$.
\end{prop}

\begin{proof}
By Proposition 9.7 of \cite{KlimSchm}, one can choose basis elements 
\begin{align*}
&y\in\mathbb{1}_{\qqq^{-1}}\subset\wedge_\qqq\mathbb{V}^*\\
&y'\in\mathbb{1}_\qqq\subset\wedge_\qqq\mathbb{V}
\end{align*}
such that
\begin{align*}
\varphi_L^\wedge(y)&=\textstyle\det_\qqq(\XX^{(a)})\otimes y\\
\varphi_R^\wedge(y')&= \textstyle y'\otimes\det_\qqq(\XX^{(a)})
\end{align*}
The result follows from equivariance.
\end{proof}

\begin{cor}\label{FRTOre}
The FRT quantum determinants $\{\det_\qqq(\XX^{(a)})\}_{a\in\ZZ/\ell\ZZ}$ generate an Ore set in $\OO_\ell$.
\end{cor}

\begin{proof}
By Proposition 9.9 of \cite{KlimSchm}, $\det_\qqq(\XX^{(a)})$ is central in $\OO_\ell^{(a)}$.
Its commutation with other superscript indices is given by the $R$-matrix.
However, by Proposition \ref{FRTDetProp} and (\ref{RFactor}), it is easy to see that $\det_\qqq(\XX^{(a)})$ at worst $\qqq$-commutes.
\end{proof}

\subsubsection{PBW basis}\label{DellPBW}
%Let
%\begin{align*}
%\widetilde{\del}_{ij}^{(a)}&:=(\qqq-\qqq^{-1})\del_{ij}^{(a)}, &
%\widetilde{\DDD}^{(a)}&:=(\qqq-\qqq^{-1})\DDD^{(a)}
%\end{align*}
%The analogous relations in \ref{Oell} and \ref{Dell} involving $\widetilde{\DDD}$ can be written over $\CC[\qqq^{\pm 1}]$, and we 
Let $\DD_\ell^{\ZZ}\subset\DD_\ell$ denote $\CC[\qqq^{\pm1}]$-subalgebra generated by the entries of $\{\XX^{(a)},\DDD^{(a)}\}_{a\in\ZZ/\ell\ZZ}$.
%Thus, we have chosen an integral form of $\DDD_\ell$ with $\widetilde{\DDD}^{(a)}$ instead of $\DDD^{(a)}$.
We define a \textit{standard monomial} of $\DD_\ell$ and $\DD_\ell^{\ZZ}$ to be a product
\begin{equation}
x_{i_1j_1}^{(a_1)}\cdots x_{i_Mj_M}^{(a_M)}\del_{k_1l_1}^{(b_1)}\cdots\del_{k_Nl_N}^{(b_N)}
\label{StdMon}
\end{equation}
where:
\begin{itemize}
\item the sequence of triples of indices $(a_1, i_1, j_1),\ldots, (a_M,i_M,j_M)$ is in increasing lexicographic order;
\item the sequence of triples of indices $(b_1, k_1, l_1),\ldots, (b_N,k_N,l_N)$ is in increasing lexicographic order.
\end{itemize}
Here, we view the superscript indices $\{a_1,\ldots, a_M,b_1,\ldots, b_M\}$ as integers $1,\ldots,\ell$, and so we want these indices to be \textit{ascending}.

\begin{thm}[\cite{JordanMult}]\label{DellFlat}
The standard monomials form a basis of $\DD_\ell^\ZZ$.
Thus, $\DD_\ell$ is a flat deformation of the coordinate ring of $T^*(\mathrm{Mat}_{n\times n}^{\oplus\ell})$.
\end{thm}

\subsection{The case $\mathbf{\ell =0}$}
For the quiver (\ref{CyclicQ}) with $\ell=1$, the algebra constructed by Jordan \cite{JordanMult} is related to $\SHH_n^0(\qqq,\ttt)$.
We will thus call this the $\ell=0$ case.
The resulting algebra is Varagnolo--Vasserot's algebra of quantum differential operators on $GL_n$ \cite{VarVassRoot}, which had also appeared prior in work of Alekseev--Schomerus \cite{AlekScho} in their quantizations of character varieties.

\subsubsection{Definition of $\DD_0$}
The algebra $\DD_0^+$ has generators
\[
\{a_{ij}, b_{ij}\, |\, 1\le i,j\le n\}
\]
Placing them in matrices
\begin{equation*}
\begin{aligned}
A&:=\sum_{i,j}E_{ij}\otimes a_{ij}, &
B&:=\sum_{i,j}E_{ij}\otimes b_{ij}
\end{aligned}
\end{equation*}
the relations are:
\begin{align}
\label{DAA}
\RRR_{21}\mathbf{A}_{1}\RRR\mathbf{A}_{2}&= \mathbf{A}_{2}\RRR_{21}\mathbf{A}_{1}\RRR;\\
\label{DBB}
\RRR_{21}\mathbf{B}_{1}\RRR\mathbf{B}_{2}&= \mathbf{B}_{2}\RRR_{21}\mathbf{B}_{1}\RRR;\\
\label{DAB}
\RRR_{21}\mathbf{B}_{1}\RRR\mathbf{A}_{2}&= \mathbf{A}_{2}\RRR_{21}\mathbf{B}_{1}\RRR_{21}^{-1};
\end{align}
The first two relations yield surjections from the reflection equation algebra $\mathcal{R}\hbox{\it ef}$ onto each of the subalgebras generated by the entries of $\mathbf{A}$ and $\mathbf{B}$.
Thus, we can make sense of the quantum determinants $\det_\qqq(\mathbf{A})$ and $\det_\qqq(\mathbf{B})$.
By Proposition 1.8.3(b) of \cite{VarVassRoot}, the quantum determinants generate an Ore set of $\DD_0^+$.
\begin{defn}
$\DD_0$ is the localization of $\DD_0^+$ at the multiplicative set generated by $\{\det_\qqq(\mathbf{A}),\det_\qqq(\mathbf{B})\}$.
We define the \textit{fourth quadrant subalgebra} $\DD_0^{\mathrm{IV}}\subset\DD_0$ to be the subalgebra generated by the entries of $\mathbf{A}$ and $\mathbf{B}^{-1}$.
\end{defn}

\subsubsection{PBW and function representation}
Denote the entries of $\mathbf{B}^{-1}$ by $\{b'_{ij}\}$.
Let $\DD_0^{\mathrm{IV,\ZZ}}\subset\DD_0^{\mathrm{IV}}$ be the $\CC[\qqq^{\pm 1}]$-subalgebra generated by the generators $\{a_{ij},b'_{ij}\}_{i,j=1}^n$.
Similar to \ref{DellPBW}, standard monomial for $\DD_0^{\mathrm{IV},\ZZ}$ and $\DD_0^{\mathrm{IV}}$ is a product 
\[
a_{i_1j_1}\cdots a_{i_Mj_M}b_{k_1l_1}'\cdots b_{k_Nl_N}'
\]
where the
\begin{itemize}
\item the sequence of pairs of indices $(i_1,j_1),\ldots (i_M,j_M)$ is in lexicographic order;
\item the sequence of pairs of indices $(k_1,l_1),\ldots (k_N,l_N)$ is in \textit{reverse} lexicographic order.
\end{itemize}

\begin{prop}[\cite{JordanMult, VarVassRoot}]\label{D0Iso}
We have the following:
\begin{enumerate}
\item The standard monomials form a basis of $\DD_0^{\mathrm{IV},\ZZ}$.
Thus, $\DD_0^{\mathrm{IV}}$ is a flat deformation of the coordinate ring of $T^*(\mathrm{Mat}_{n\times n})$.
\item $\DD_0$ is isomorphic to $\OO_G\otimes\OO_G$ as a $\UU$-module.
Here, the first tensorand corresponds to the subalgebra generated by the entries of $\mathbf{A}$ while the second tensorand corresponds to that of $\mathbf{B}$.
\end{enumerate}
\end{prop}

\noindent We denote by $\OO_G^{\mathbf{A}},\OO_G^{\mathbf{B}}\subset\DD_0$ the subalgebras generated by the entries of $\mathbf{A}$ and $\mathbf{B}$, respectively, so 
\[\DD_0\cong\OO_G^{\mathbf{A}}\otimes\OO_G^{\mathbf{B}}\]
Consequently, we can define a function representation on
\[
\OO_G^{\mathbf{A}}\cong\DD_0\big/\DD_0\left( \ker\epsilon_{\OO_G^\mathbf{B}}  \right)
\]
Here, $\epsilon_{\OO_G^\mathbf{B}}$ is the counit of $\OO_G^\mathbf{B}$ (cf. \ref{OGDef}).

\begin{prop}[\cite{VarVassRoot}]\label{FuncFaith}
%The function representation of $\DD_0$ is faithful.
For $v^*\otimes v\in V^*\otimes V\subset\OO^\mathbf{A}_G$, viewed as an element of the function representation, and $b\in\OO_G^\mathbf{B}$,
\[
b\cdot (v^*\otimes v)=\left(\kappa(b)\bullet v^*\right)\otimes v
\]
Here, $\bullet$ denotes the $\UU$-action on $V^*$.
\end{prop}

\subsubsection{Homomorphism into $\DD_\ell$}
For a sequence of noncommuting elements $x_1,\ldots, x_M$, let
\[
\overset{\curvearrowright}{\prod_{a=1}^M}x_a:=x_1\cdots x_M
\]
For $\ell>1$, consider the following two matrices of elements of $\DD_\ell$:
\begin{equation*}
\begin{aligned}
\XX^\circ&:= \overset{\curvearrowright}{\prod_{a=1}^\ell}\XX^{(a)}=\XX^{(1)}\XX^{(2)}\cdots\XX^{(\ell)},&
{}^L\underline{\YY}^{(1)}&:=\mathbf{I}+\XX^{(1)}\DDD^{(1)}
\end{aligned}
\end{equation*}
where $\mathbf{I}:=\sum_i E_{ii}\otimes 1$ is the identity matrix.

\begin{lem}\label{D0DellLem}
The map
\begin{equation*}
\begin{aligned}
\mathbf{A}&\mapsto\XX^\circ,&
\mathbf{B}^{-1}&\mapsto\YLu^{(1)}
\end{aligned}
\end{equation*}
defines an algebra homomorphism $\Phi^\ell:\DD_0^{\mathrm{IV}}\rightarrow\DD_\ell$.
If we endow $\DD_0^{\mathrm{IV}}$ with the $\UU^{\otimes \ell}$ action whereby $\UU_{(1)}$ acts via the usual $\UU$-action and $\UU_{(a)}$ acts trivially for $a\not=1$, then $\Phi^\ell$ is $\UU^{\otimes\ell}$-equivariant.
\end{lem}

\begin{proof}
We need to check that $\XX^\circ$ and ${}^L\underline{\YY}^{(1)}$ satisfy the relations defining $\DD_0^{\mathrm{IV}}$.
For (\ref{DAA}), let us first consider the case $\ell\ge 3$.
We begin by moving $\XX_2^{(1)}$ all the way to the left:
\begin{align}
\nonumber
\RRR_{21}\XX^\circ_1\RRR\XX^\circ_2&=
\RRR_{21}\XX_1^{(1)}\XX_1^{(2)}
\left(\overset{\curvearrowright}{\prod_{a=3}^{\ell-1}}\XX_1^{(a)}\right)
\underbrace{\XX_1^{(\ell)}\RRR\XX_2^{(1)}}_{(\ref{XXad})}
\left(\overset{\curvearrowright}{\prod_{a=2}^{\ell}}\XX_2^{(a)}\right)\\
\nonumber
&=\RRR_{21}\XX_1^{(1)}\XX_1^{(2)}
\underbrace{\left(\overset{\curvearrowright}{\prod_{a=3}^{\ell-1}}\XX_1^{(a)}\right)\XX_2^{(1)}}_{(\ref{XXcom})}
\XX_1^{(\ell)}\left(\overset{\curvearrowright}{\prod_{a=2}^{\ell}}\XX_2^{(a)}\right)\\
\nonumber
&=\RRR_{21}\XX_1^{(1)}
\underbrace{\XX_1^{(2)}\XX_2^{(1)}}_{(\ref{XXad})}
\left(\overset{\curvearrowright}{\prod_{a=3}^{\ell}}\XX_1^{(a)}\right)
\left(\overset{\curvearrowright}{\prod_{a=2}^{\ell}}\XX_2^{(a)}\right)\\
\nonumber
&= 
\underbrace{\RRR_{21}\XX_1^{(1)}\XX_2^{(1)}}_{(\ref{Xaa})}
\RRR_{21}
\left(\overset{\curvearrowright}{\prod_{a=2}^{\ell}}\XX_1^{(a)}\right)
\left(\overset{\curvearrowright}{\prod_{a=2}^{\ell}}\XX_2^{(a)}\right)\\
\label{PhiLem1}
&= \XX_2^{(1)}\XX_1^{(1)}\RRR\RRR_{21}
\left(\overset{\curvearrowright}{\prod_{a=2}^{\ell}}\XX_1^{(a)}\right)
\left(\overset{\curvearrowright}{\prod_{a=2}^{\ell}}\XX_2^{(a)}\right)
\end{align}
After this, moves are identical: iterate (\ref{XXcom}), then an application of (\ref{XXad}), (\ref{Xaa}), and (\ref{XXad}) again.
We illustrate this for the first step following (\ref{PhiLem1}):
\begin{align*}
&=\XX_2^{(1)}\XX_1^{(1)}\RRR\RRR_{21}\XX_1^{(2)}\XX_1^{(3)}
\underbrace{\left(\overset{\curvearrowright}{\prod_{a=4}^{\ell}}\XX_1^{(a)}\right)
\XX_2^{(2)}}_{(\ref{XXcom})}
\left(\overset{\curvearrowright}{\prod_{a=3}^{\ell}}\XX_2^{(a)}\right)\\
&=\XX_2^{(1)}\XX_1^{(1)}\RRR\RRR_{21}\XX_1^{(2)}
\underbrace{\XX_1^{(3)}\XX_2^{(2)}}_{(\ref{XXad})}
\left(\overset{\curvearrowright}{\prod_{a=4}^{\ell}}\XX_1^{(a)}\right)
\left(\overset{\curvearrowright}{\prod_{a=3}^{\ell}}\XX_2^{(a)}\right)\\
&= \XX_2^{(1)}\XX_1^{(1)}\RRR
\underbrace{\RRR_{21}\XX_1^{(2)}\XX_2^{(2)}}_{(\ref{Xaa})}
\RRR_{21}
\left(\overset{\curvearrowright}{\prod_{a=3}^{\ell}}\XX_1^{(a)}\right)
\left(\overset{\curvearrowright}{\prod_{a=3}^{\ell}}\XX_2^{(a)}\right)\\
&= \XX_2^{(1)}
\underbrace{\XX_1^{(1)}\RRR\XX_2^{(2)}}_{(\ref{XXad})}
\XX_1^{(2)}\RRR\RRR_{21}
\left(\overset{\curvearrowright}{\prod_{a=3}^{\ell}}\XX_1^{(a)}\right)
\left(\overset{\curvearrowright}{\prod_{a=3}^{\ell}}\XX_2^{(a)}\right)\\
&= \XX_2^{(1)}\XX_2^{(2)}\XX_1^{(1)}\XX_1^{(2)}\RRR\RRR_{21}
\left(\overset{\curvearrowright}{\prod_{a=3}^{\ell}}\XX_1^{(a)}\right)
\left(\overset{\curvearrowright}{\prod_{a=3}^{\ell}}\XX_2^{(a)}\right)
\end{align*}
The closing step entails commuting $\XX_2^{(\ell)}$ to the left:
\begin{align*}
&= \left( \overset{\curvearrowright}{\prod_{a=1}^{\ell-1}}\XX_2^{(a)} \right)\left( \overset{\curvearrowright}{\prod_{a=1}^{\ell-1}}\XX_1^{(a)} \right)\RRR
\underbrace{\RRR_{21}\XX_1^{(\ell)}\XX_2^{(\ell)}}_{(\ref{Xaa})}\\
&= \left( \overset{\curvearrowright}{\prod_{a=1}^{\ell-1}}\XX_2^{(a)} \right)\left( \overset{\curvearrowright}{\prod_{a=1}^{\ell-2}}\XX_1^{(a)} \right)
\underbrace{\XX_1^{(\ell-1)}\RRR\XX_2^{(\ell)}}_{(\ref{XXad})}
\XX_1^{(\ell)}\RRR\\
&= \left( \overset{\curvearrowright}{\prod_{a=1}^{\ell-1}}\XX_2^{(a)} \right)\XX_1^{(1)}
\underbrace{\left( \overset{\curvearrowright}{\prod_{a=2}^{\ell-2}}\XX_1^{(a)} \right)\XX_2^{(\ell)}}_{(\ref{XXcom})}
\XX_1^{(\ell-1)}\XX_1^{(\ell)}\RRR\\
&= \left( \overset{\curvearrowright}{\prod_{a=1}^{\ell-1}}\XX_2^{(a)} \right)
\underbrace{\XX_1^{(1)}\XX_2^{(\ell)}}_{(\ref{XXad})}
\left( \overset{\curvearrowright}{\prod_{a=2}^{\ell}}\XX_1^{(a)} \right)\RRR\\
&= \left( \overset{\curvearrowright}{\prod_{a=1}^{\ell}}\XX_2^{(a)} \right)\RRR_{21}\left( \overset{\curvearrowright}{\prod_{a=1}^{\ell}}\XX_1^{(a)} \right)\RRR=\XX_2^\circ\RRR_{21}\XX_1^\circ\RRR\\
\end{align*}
For $\ell=2$, we have:
\begin{align*}
\RRR_{21}\XX_{1}^\circ\RRR\XX_2^\circ&= \RRR_{21}\XX_1^{(1)}
\underbrace{\XX_1^{(2)}\RRR\XX_2^{(1)}}_{(\ref{XXad2})}
\XX_2^{(2)}\\
&= \underbrace{\RRR_{21}\XX_1^{(1)}\XX_2^{(1)}}_{(\ref{Xaa})}
\underbrace{\RRR_{21}\XX_{1}^{(2)}\XX_2^{(2)}}_{(\ref{Xaa})}\\
&= \XX_2^{(1)}
\underbrace{\XX_1^{(1)}\RRR\XX_{2}^{(2)}}_{(\ref{XXad2})}
\XX_1^{(2)}\RRR\\
&= \XX_2^{(1)}\XX_2^{(2)}\RRR_{21}\XX_1^{(1)}\XX_1^{(2)}\RRR
=\XX_2^\circ\RRR_{21}\XX_1^\circ\RRR
\end{align*}

The remaining relations rely on results presented in \ref{MomentEll}; these results do not rely on the present material.
Relation (\ref{DBB}) was proved by Jordan \cite{JordanMult} (cf. Proposition \ref{DellMuProp}(2) below). 
For (\ref{DAB}), we have:
\begin{align*}
{}^L\underline{\YY}^{(1)}_1\RRR_{21}^{-1}\XX^\circ_2\RRR_{21}
&=
\underbrace{{}^L\underline{\YY}^{(1)}_1\RRR_{21}^{-1}\XX^{(1)}_2}_{(\ref{YL1})}
\left(\overset{\curvearrowright}{\prod_{a=2}^{\ell-1}}\XX^{(a)}_2\right)\XX^{(\ell)}_2\RRR_{21}\\
&= 
\RRR\XX^{(1)}_2
{}^L\underbrace{\underline{\YY}^{(1)}_1\left(\overset{\curvearrowright}{\prod_{a=2}^{\ell-1}}\XX^{(a)}_2\right)}_{(\ref{YL2})}
\XX^{(\ell)}_2\RRR_{21}\\
&= 
\RRR\XX^{(1)}_2\left(\overset{\curvearrowright}{\prod_{a=2}^{\ell-1}}\XX^{(a)}_2\right)
\underbrace{{}^L\underline{\YY}^{(1)}_1\XX^{(\ell)}_2\RRR_{21}}_{(\ref{YL3})}\\
&= 
\RRR\XX^{(1)}_2\left(\overset{\curvearrowright}{\prod_{a=2}^{\ell-1}}\XX^{(a)}_2\right)
\XX^{(\ell)}_2\RRR_{21}{}^L\underline{\YY}^{(1)}_1\\
&=\RRR\XX^\circ_2\RRR_{21}{}^L\underline{\YY}^{(1)}_1\qedhere
\end{align*}
\end{proof}

\subsubsection{Localization in $\DD_\ell$}
By Lemma \ref{D0DellLem}, we can make sense of the element 
\[\textstyle\det_\qqq(\XX^\circ):=\Phi^\ell\left(\det_\qqq(\mathbf{A})\right)\in\DD_\ell\]

\begin{lem}\label{DellDet}
The quantum determinant $\det_\qqq(\XX^\circ)$ generates an Ore set in $\DD_\ell$.
\end{lem}

\begin{proof}
First, let us consider the commutation relations between $\XX^\circ$ and $\{\XX^{(a)}\}$.
Using (\ref{Xaa})-(\ref{XXcom}), it is easy to see that
\begin{align}
\label{X1circ}
\XX^{(1)}_1\XX_2^\circ&= \RRR_{21}^{-1}\XX_2^\circ\RRR_{21}\XX^{(1)}_1\\
\label{Xellcirc}
\XX_1^\circ\XX_2^{(\ell)}&= \XX_2^{(\ell)}\RRR_{21}\XX_1^\circ\RRR^{-1}_{21}\\
\nonumber
\XX_1^{(a)}\XX_2^\circ&=\XX^\circ_2\XX_1^{(a)}\hbox{ for }a\not=1,\ell
\end{align}
Equations (\ref{X1circ}) and (\ref{Xellcirc}) can be written as follows.
Let $\bullet$ denote the action of $\UU$ on $\mathbb{V}$ or $\mathbb{V}^*$, and let $\bowtie$ denote the action of $\UU$ on any subspace $V^*\otimes V\subset\OO_G$ via the coproduct $\Delta_\UU$.
We will recycle notation from \ref{DellEq}: for $e^i_{(1)}\boxtimes e_j^{(1)}=x_{ij}^{(1)}$, $e^i_{(\ell)}\boxtimes e_j^{(\ell)}=x_{ij}^{(\ell)}$, and $x^\circ_{kl}:=\Phi^\ell(a_{kl})$,
\begin{align*}
\left(e^i_{(1)}\boxtimes e_j^{(1)}\right)x^\circ_{kl}&=\left( \iota_\UU({}_sr)\bowtie x^\circ_{kl} \right)\left(r_s\bullet e^i_{(1)}\boxtimes e_j^{(1)} \right) \\
x^\circ_{kl}\left(e^i_{(\ell)}\boxtimes e_j^{(\ell)}\right)&=\left( e^i_{(\ell)}\boxtimes r_s\bullet e_j^{(\ell)} \right)\left( \iota_\UU(r_s)\bowtie x^\circ_{kl} \right)
\end{align*}
By (\ref{RCoprod}) and the fact that $\bowtie$ is an algebra action, we have that for any product $x^\circ$ in the entries of $\XX^\circ$,
\begin{align*}
\left(e^i_{(1)}\boxtimes e_j^{(1)}\right)x^\circ&=\left( \iota_\UU({}_sr)\bowtie x^\circ \right)\left(r_s\bullet e^i_{(1)}\boxtimes e_j^{(1)} \right) \\
x^\circ\left(e^i_{(\ell)}\boxtimes e_j^{(\ell)}\right)&=\left( e^i_{(\ell)}\boxtimes r_s\bullet e_j^{(\ell)} \right)\left( \iota_\UU(r_s)\bowtie x^\circ \right)
\end{align*}
Because $\det_\qqq(\XX^\circ)$ is $\UU$-invariant, we conclude that it commutes with $\OO_\ell$.

For $\{\DDD^{(a)}\}$, we can show similar relations.
Let us consider the easiest case of $a\not=1,\ell$:
\begin{align}
%\XX^{\circ}_1\DDD_2^{(1)}&= \DDD^{(1)}_2\RRR^{-1}\XX_2^\circ\RRR+\hbox{ terms in }\OO_\ell\\
%\DDD_1^{(\ell)}\XX_2^\circ&= \RRR_{21}^{-1}\XX_2^\circ\RRR_{21}\DDD_1^{(\ell)}+\hbox{ terms in }\OO_\ell\\
\label{DetDa}
\DDD_1^{(a)}\XX_2^\circ&=\XX_2^{(1)}\cdots\XX_2^{(a-1)}\left(\XX_2^{(a)}\right)_{*j}\RRR_{21}\RRR\left( \XX_2^{(a+1)} \right)_{j*}\XX_2^{(a+2)}\cdots\XX_2^{(\ell)}\DDD^{(a)}_1+\hbox{ terms in }\OO_\ell
\end{align}
Here, the subscripts $*j$ and $j*$ emphasize that we are considering the matrix units $E_{*j}$ and $E_{j*}$, where $*$ is whatever makes the matrix products work.
Moreover, we use an Einstein notation and thus there is an implicit summation over $j$.
Thus, all we are saying is that the $j$-indices must match.
We can rewrite (\ref{DetDa}) using $\bar{\kappa}$:
\begin{equation}
\DDD_1^{(a)}\XX_2^\circ=\bar{\kappa}(\mathrm{coev}_{\mathbb{V}}(1))\big|_{\mathbb{V}^*}\XX^\circ_2\DDD^{(a)}_1+\hbox{ terms in }\OO_\ell
\label{DetDa2}
\end{equation}
Because $\bar{\kappa}(\mathrm{coev}_{\mathbb{V}}(1))\in\UU$ is central, it acts on $\mathbb{V}^*$ by a constant (which is easy to see is nonzero).
Similarly, we have
\begin{align}
\label{DetD1}
\XX^{\circ}_1\DDD_2^{(1)}&= \kappa(\mathrm{coev}_{\mathbb{V}}(1))\big|_{\mathbb{V}}\DDD_2^{(1)}\RRR^{-1}\XX_1^{\circ}\RRR +\hbox{ terms in }\OO_\ell\\
\label{DetDell}
\DDD_1^{(\ell)}\XX^\circ_2&= \RRR\XX^\circ_2\RRR_{21}\DDD_1^{(\ell)}+\hbox{ terms in }\OO_\ell
\end{align}
 
We can endow $\DD_\ell$ with the structure of a filtered algebra by setting
\begin{equation*}
\begin{aligned}
\deg(x_{ij})&=0,&\deg(\del_{ij})=1
\end{aligned}
\end{equation*}
Through similar reasoning as in the first paragraph, we have that in the associated graded, $\det_\qqq(\XX^\circ)$ commutes up to scalar with $\del_{ij}^{(a)}$. %for $a\not=\ell$.
%For $a=\ell$, it commutes up to scaling by $\bar{\kappa}(\mathrm{qcoev}_{\mathbb{1}_\qqq}(1))\big|_\mathbb{V^*}$.
We have cancelability because $\DD_\ell$ is an integral domain, and for the Ore conditions, we can solve the problem at each step in the filtration.
At filtration degree 0, the problem is trivial because $\det_\qqq(\XX^\circ)$ is central in $\OO_\ell$.
\end{proof}

\begin{cor}\label{XDetCor}
The matrices $\XX^{(a)}$ are invertible in $\DD_\ell\left[ \det_\qqq(\XX^\circ)^{-1} \right]$.
\end{cor}

\begin{proof}
Recall the FRT quantum determinants $\{\det_{\qqq}(\XX^{(a)})\}\in\OO_\ell$ from \ref{FRTDet}. 
By Corollary \ref{FRTOre}, they generate an Ore set in $\OO_\ell$.
Proposition 9.8 of \cite{KlimSchm}) implies that $\XX^{(a)}$ has a left and right inverse in $\OO_\ell\left[ \det_\qqq(\XX^{(a)})^{-1} \right]$.
Let
\[
\textstyle\OO_\ell^\star:=\OO_\ell\left[ \det_{\qqq}(\XX^\circ)^{-1},\det_\qqq(\XX^{(1)})^{-1},\ldots,\det_\qqq(\XX^{(\ell)})^{-1} \right]
\]
A consequence of Theorem \ref{DellFlat} is that $\OO_\ell$ is an integral domain.
Thus, $\OO_\ell[\det_\qqq(\XX^\circ)^{-1}]$ includes into $\OO_\ell^\star$.
We first have that
\[
\XX^{(2)}\cdots\XX^{(\ell)}\left(\XX^\circ\right)^{-1}
\]
is a \textit{right} inverse of $\XX^{(1)}$.
Since $\XX^{(1)}$ has a left-right inverse over $\OO_\ell^\star$, it follows that the expression above is also a left inverse.
Similarly, 
\[
\left( \XX^\circ \right)^{-1}\XX^{(1)}
\]
is a left and hence left-right inverse to $\XX^{(2)}\cdots\XX^{(\ell)}$.
Proceeding by induction, if we have an inverse to $\XX^{(a)}\cdots\XX^{(\ell)}$, then 
\[
\XX^{(a+1)}\cdots\XX^{(\ell)}\left( \XX^{(a)}\cdots\XX^{(\ell)} \right)^{-1}
\]
is a right inverse of $\XX^{(a)}$.
By similar arguments, this must be a left-right inverse.
\end{proof}

\subsection{The case $\mathbf{\ell=1}$}
We will introduce a new algebra of quantum differential operators that is closer in spirit to the construction of Crawley-Boevey--Shaw \cite{CBShaw} for the quiver (\ref{CyclicQ}) for $\ell=1$.

\subsubsection{Definition of $\DD_1$}\label{D1Def}
The algebra $\DD_1$ has generators
\[
\left\{ x_{ij},\del_{ij}\, \middle|\, 1\le i,j\le n \right\}
\]
which we place into the matrices
\begin{equation*}
\begin{aligned}
\XX&:=\sum_{i,j}E_{ij}\otimes x_{ij},&
\DDD&:=\sum_{i,j}E_{ij}\otimes \del_{ij}
\end{aligned}
\end{equation*}
The relations are:
\begin{align}
\label{D1rel1}
\RRR_{21}\XX_1\RRR\XX_2&= \XX_2\RRR_{21}\XX_1\RRR;\\
\label{D1rel2}
\RRR^{-1}\DDD_1\RRR_{21}^{-1}\DDD_2&=\DDD_2\RRR^{-1}\DDD_1\RRR^{-1}_{21};\\
\label{D1rel3}
\DDD_1\RRR_{21}^{-1}\XX_2\RRR_{21}&= \RRR\XX_2\RRR_{21}\DDD_1 + (\qqq-\qqq^{-1})\RRR\mathbf{\Omega}
\end{align}
This presentation describes $\DD_1$ as a quotient of the tensor algebra of $\mathbb{V^*}\otimes\mathbb{V}\oplus\mathbb{V}\otimes\mathbb{V}^*$.
Let us denote the first summand by $\mathbb{V}^*_{\mathbf{X}}\otimes\mathbb{V}_\mathbf{X}$ and the second by $\mathbb{V}_{\DDD}\otimes\mathbb{V}_\DDD^*$ and use similar notation for the bases, e.g. $\{e^i_\XX\}\subset\mathbb{V}^*_\XX$.
The quotient map is then:
\begin{align*}
\mathcal{T}\left(\mathbb{V}^*_\XX\otimes\mathbb{V}_\XX \oplus \mathbb{V}_\DDD\otimes\mathbb{V}_\DDD^* \right)&\rightarrow\DD_1\\
e^i_\XX\otimes e_j^\XX&\mapsto x_{ij}\\
\iota_{\OO_G}(e^i_\DDD\otimes e_j^\DDD)&\mapsto \del_{ij}
\end{align*}
Note that we are identifying $\mathbb{V}_\DDD^*\otimes\mathbb{V}_\DDD\cong\mathbb{V}_\DDD\otimes\mathbb{V}^*_\DDD$ using $\iota_{\OO_G}$.

\begin{prop}\label{D1Equiv}
$\DD_1$ is a $\UU$-equivariant algebra.
\end{prop}

\begin{proof}
Relations (\ref{D1rel1})-(\ref{D1rel3}) are, respectively, the images of the maps
\begin{align*}
\beta_{\mathbb{V}_\XX,\mathbb{V}^*_\XX}^{-1}\left( \beta_{\mathbb{V}^*_\XX,\mathbb{V}^*_\XX}-\beta_{\mathbb{V}_\XX,\mathbb{V}_\XX} \right): 
\mathbb{V}^*_\XX\otimes\mathbb{V}^*_\XX\otimes\mathbb{V}_\XX\otimes\mathbb{V}_\XX
&\rightarrow \mathbb{V}^*_\XX\otimes\mathbb{V}_\XX\otimes\mathbb{V}^*_\XX\otimes\mathbb{V}_\XX\\
\beta_{\mathbb{V}^*_\DDD,\mathbb{V}_\DDD}\left(\beta_{\mathbb{V}^*_\DDD,\mathbb{V}^*_\DDD}^{-1}- \beta_{\mathbb{V}_\DDD,\mathbb{V}_\DDD}^{-1} \right): 
\mathbb{V}^*_\DDD\otimes\mathbb{V}^*_\DDD\otimes\mathbb{V}_\DDD\otimes\mathbb{V}_\DDD
&\rightarrow \mathbb{V}_\DDD^*\otimes\mathbb{V}_\DDD\otimes\mathbb{V}_\DDD^*\otimes\mathbb{V}_\DDD\\
\left(
\begin{array}{l}
\beta_{\mathbb{V}_\XX^*,\mathbb{V}_\DDD}\beta_{\mathbb{V}_\DDD,\mathbb{V}_\XX}^{-1}\vspace{0.05in}
-\beta_{\mathbb{V}_\XX,\mathbb{V}^*_\DDD}^{-1}\beta_{\mathbb{V}^*_\DDD,\mathbb{V}^*_\XX}\\
-(\qqq-\qqq^{-1})
\mathrm{ev}_{\mathbb{V}\otimes\mathbb{V}}\beta_{\mathbb{V}^*_\DDD,\mathbb{V}^*_\XX}
\end{array}\right)
:
\mathbb{V}^*_\DDD\otimes\mathbb{V}_\XX^*\otimes\mathbb{V}_\XX\otimes\otimes\mathbb{V}_\DDD
&\rightarrow
\left(
\begin{array}{l}
\left(\mathbb{V}_\DDD^*\otimes\mathbb{V}_\DDD\otimes\mathbb{V}^*_\XX\otimes\mathbb{V}_\XX\right)\vspace{.05in}\\
\oplus
\left(\mathbb{V}^*_\XX\otimes\mathbb{V}_\XX\otimes\mathbb{V}_\DDD^*\otimes\mathbb{V}_\DDD\right)\oplus\mathbb{1}
\end{array}\right)
\end{align*}
followed by $\iota_{\OO_G}$ for any $\mathbb{V}^*_\DDD\otimes\mathbb{V}_\DDD$ tensorand.
Thus, the relations describe a $\UU$-submodule of the tensor algebra and $\DD_1$ is $\UU$-equivariant.
\end{proof}

\subsubsection{PBW for $\DD_1$}
The notion of standard monomial for $\DD_1$ will be slightly different from that of $\DD_\ell$.
Namely, we define a standard monomial to be 
\[
\del_{k_1l_1}\cdots\del_{k_Nl_N}
x_{i_1j_1}\cdots x_{i_Mj_M}
\]
where:
\begin{itemize}
\item the sequence of pairs $(i_1, j_1),\ldots, (i_M,j_M)$ are in increasing lexicographic order
\item and the sequence of pairs $(k_1, l_1),\ldots, (k_N,l_N)$ are in \textit{decreasing} lexicographic order.
\end{itemize}
Let $\DD_1^\ZZ$ be the $\CC[\qqq^{\pm 1}]$-subalgebra generated by $\{x_{ij},\del_{ij}\}$.
\begin{prop}\label{D1PBW}
The standard monomials are a basis of $\DD_1^\ZZ$.
%The standard monomials are a linearly indepedent, and any monomial in the generators of $\DD_(\ZZ)$ lies in the $\CC[\qqq^{\pm 1}]$-span of standard monomials.
Consequently, $\DD_1$ is a flat deformation of the coordinate ring of $T^*(\mathrm{Mat}_{n\times n})$.
\end{prop}

\begin{proof}
Observe that the subalgebra generated by the coordinates of $\XX$ is isomorphic to $\mathcal{R}\hbox{\it ef}$ and the subalgebra generated by the coordinates of $\DDD$ is isomorphic to $\iota_{\OO_G}(\mathcal{R}\hbox{\it ef}\,)\subset\OO_G$.
It is known that the analogous standard monomials are linearly independent in $\mathcal{R}\hbox{\it ef}$ \cite{KlimSchm}, and the relations show that any monomial in $\mathcal{R}\hbox{\it ef}$ can be straightened into a $\CC[\qqq^{\pm 1}]$-linear combination of the standard ones (cf. \cite{JordanMult} Example 4.10).
Thus, we get the desired statement for monomials in $\{x_{ij}\}$, and the analogous statement for monomials in the $\{\del_{ij}\}$ holds after applying the algebra anti-automorphism $\iota_{\OO_G}$.
To upgrade this to $\DD_1$, it suffices to show that:
\begin{enumerate}
\item $x_{kl}\del_{ij}$ lies in the $\CC[\qqq^{\pm 1}]$-span of standard monomials;
\item $\DD_1\cong\iota_{\OO_G}(\mathcal{R}\hbox{\it ef}\,)\otimes\mathcal{R}\hbox{\it ef}\,$.
\end{enumerate}

For the first statement, let us unpack (\ref{D1rel3}):
\begin{align}
\nonumber
&\quad \qqq^{\delta_{i,k}-\delta_{i,l}}x_{kl}\del_{ij}\\
\nonumber
&= \qqq^{-\delta_{j,k}+\delta_{j,l}}\del_{ij}x_{kl}
 - (\qqq-\qqq^{-1})\bigg(\qqq^{\delta_{i,j}}\delta_{i,l}\delta_{j,k}+(\qqq-\qqq^{-1})\theta(j-l)\delta_{i,j}\delta_{k,l}\bigg)\\
\label{D1DX}
&\quad-(\qqq-\qqq^{-1})\bigg(
\underbrace{\theta(k-i)x_{il}\del_{kj}}_{(\mathrm{a})}
+ 
\underbrace{\delta_{i,l}\sum_{m>i}x_{km}\del_{mj}}_{(\mathrm{b})}
\bigg)
-(\qqq-\qqq^{-1})^2
\underbrace{\theta(k-i)\delta_{k,l}\sum_{m>k}x_{im}\del_{mj}}_{(\mathrm{c})}
\\
\nonumber
&\quad+(\qqq-\qqq^{-1})\bigg(
\theta(l-j)\del_{il}x_{kj}
-
\delta_{j,k}\sum_{m<j}\del_{im}x_{ml}
\bigg)
-(\qqq-\qqq^{-1})^2
\theta(l-j)\delta_{k,l}\sum_{m<l}\del_{im}x_{mj}
\end{align}
Here, $\theta(x):\RR\rightarrow\RR$ is the Heaviside theta function:
\[
\theta(x)=
\begin{cases}
1 & x>0\\
0 & x\le 0
\end{cases}
\]
Only the terms in the second line (\ref{D1DX}) are not standard monomials, and our goal is to show that these terms disappear after multiple iterations of the relation.
First consider a term from (b).
Clearly, it cannot simultaneously satisfy the conditions for (c).
If it simultaneously satisfies the conditions for (a), then the result of (a) will be of the form $x_{ii}\del_{kj}$ with $k>i$, which does not satisfy any of the conditions for (a), (b), or (c).
This shows that iterating (a) will terminate once $i=n$.
Next, note that a term from (a) can yield at most a term in (b) after applying the relation, which terminates upon iteration.
Finally, a term from (c) will yield a term in (b) and thus will terminate as well.

For the second statement, the proof is similar to that of Theorem 5.3 from \cite{JordanMult}.
Let us recycle notation from \ref{D1Def}.
Namely, we view:
\begin{itemize}
\item $\mathcal{R}\hbox{\it ef}$ as a quotient of the tensor algebra $\mathcal{T}_\XX:=\mathcal{T}(\mathbb{V}^*_\XX\otimes\mathbb{V}_\XX)$ by the ideal $I_\XX$ generated by the relation (\ref{D1rel1});
\item $\iota_{\OO_G}(\mathcal{R}\hbox{\it ef}\,)$ as a quotient of the tensor algebra $\mathcal{T}_\DDD:=\mathcal{T}(\mathbb{V}^*_\DDD\otimes\mathbb{V}_\DDD)$ by the ideal $I_\DDD$ generated by the relation $(\ref{D1rel2})$;
\item and $\DD_1$ as a quotient of the tensor algebra $\mathcal{T}_{\DDD\XX}:=\mathcal{T}(\mathbb{V}^*_\DDD\otimes\mathbb{V}_\DDD\oplus\mathbb{V}^*_\XX\otimes\mathbb{V}_\XX)$ by the ideal generated by $I_\XX$, $I_\DDD$, and the ideal $I_{\DDD\XX}$ generated by (\ref{D1rel3}).
\end{itemize}
As in \textit{loc. cit.}, it suffices to show that:
\begin{equation*}
\begin{aligned}
\mathcal{T}_\XX I_\DDD&\subset I_\DDD\mathcal{T}_\XX  + I_{\DDD\XX}, &
I_\XX\mathcal{T}_\DDD&\subset \mathcal{T}_\DDD I_\XX + I_{\DDD\XX}
\end{aligned}
\end{equation*}
We will only show the second containment.
Just like in \textit{loc. cit.}, this is equivalent to showing that the entries of
\[
\left(\RRR_{21}\XX_1\RRR_{12}\XX_2-\XX_2\RRR_{21}\XX_1\RRR_{12}\right)\RRR_{13}\RRR_{23}\DDD_3
\]
lie in $I_\XX\mathcal{T}_\DDD + I_{\DDD\XX}$ without using (\ref{D1rel1}) or (\ref{D1rel2}).
To that end:
\begin{align}
\nonumber
&\quad
	\RRR_{21}\XX_1\RRR_{12}
	\underbrace{\XX_{2}\RRR_{13}}_{\mathrm{commute}}
	\RRR_{23}\DDD_3
	-\XX_2\RRR_{21}\XX_1
	\underbrace{\RRR_{12}\RRR_{13}\RRR_{23}}_{(\ref{QYBE})}
	\DDD_3
	\\
\nonumber
&=
	\RRR_{21}\XX_1\RRR_{12}\RRR_{13}
	\underbrace{\XX_{2}\RRR_{23}\DDD_3}_{(\ref{D1rel3})}
	-\XX_2\RRR_{21}
	\underbrace{\XX_1\RRR_{23}\RRR_{13}\RRR_{12}\DDD_3}_{\mathrm{commute}}
	\\
\nonumber
&=
	\RRR_{21}\XX_1
	\underbrace{\RRR_{12}\RRR_{13}\RRR_{32}^{-1}}_{(\ref{QYBE})}
	\DDD_3\RRR_{23}^{-1}\XX_2\RRR_{23}
	-\XX_2\RRR_{21}\RRR_{23}
	\underbrace{\XX_1\RRR_{13}\DDD_3}_{(\ref{D1rel3})}
	\RRR_{12}
	\\
	\nonumber
	&\quad
	-(\qqq-\qqq^{-1})\RRR_{21}\XX_1\RRR_{12}\RRR_{13}\mathbf{\Omega}_{23}\\
\nonumber
&=
	\RRR_{21}
	\underbrace{\XX_1\RRR_{32}^{-1}\RRR_{13}\RRR_{12}\DDD_3}_{\mathrm{commute}}
	\RRR_{23}^{-1}\XX_2\RRR_{23}
	-\XX_2
	\underbrace{\RRR_{21}\RRR_{23}\RRR_{31}^{-1}}_{(\ref{QYBE})}
	\DDD_3\RRR_{13}^{-1}\XX_1\RRR_{13}
	\RRR_{12}
	\\
	\nonumber
	&\quad
	+(\qqq-\qqq^{-1})\bigg(
	\XX_2\RRR_{21}\RRR_{23}\mathbf{\Omega}_{13}
	\underbrace{\RRR_{12}}_{(\ref{HeckeCond})}
	-
	\underbrace{\RRR_{21}}_{(\ref{HeckeCond})}
	\XX_1\RRR_{12}\RRR_{13}\mathbf{\Omega}_{23}
	\bigg)
	\\
\nonumber
&=
	\RRR_{21}\RRR_{32}^{-1}
	\underbrace{\XX_1\RRR_{13}\DDD_3}_{(\ref{D1rel3})}
	\RRR_{12}\RRR_{23}^{-1}\XX_2\RRR_{23}
	-
	\underbrace{\XX_2\RRR_{31}^{-1}\RRR_{23}\RRR_{21}\DDD_3}_{\mathrm{commute}}
	\RRR_{13}^{-1}\XX_1\RRR_{13}\RRR_{12}
	\\
	\nonumber
	&\quad
	+(\qqq-\qqq^{-1})\bigg(
	\XX_2\RRR_{21}\RRR_{23}
	\underbrace{\mathbf{\Omega}_{13}\RRR_{21}^{-1}}_{\mathrm{commute}}
	-
	\RRR_{12}^{-1}\XX_1\RRR_{12}\RRR_{13}\mathbf{\Omega}_{23}
	\bigg)
	\\
	\nonumber
	&\qquad +
	(\qqq-\qqq^{-1})^2\bigg(
	\underbrace{\XX_2\RRR_{21}\RRR_{23}\mathbf{\Omega}_{13}\mathbf{\Omega}_{12}}_{\mathrm{commute}}
	-
	\mathbf{\Omega}_{12}\XX_1\RRR_{12}\RRR_{13}\mathbf{\Omega}_{23}
	\bigg)	
	\\
\nonumber
&=
	\underbrace{\RRR_{21}\RRR_{32}^{-1}\RRR_{31}^{-1}}_{(\ref{QYBE})}
	\DDD_3\RRR_{13}^{-1}\XX_1\RRR_{13}
	\RRR_{12}\RRR_{23}^{-1}\XX_2\RRR_{23}
	-\RRR_{31}^{-1}
	\underbrace{\XX_2\RRR_{23}\DDD_3}_{(\ref{D1rel3})}
	\RRR_{21}\RRR_{13}^{-1}\XX_1\RRR_{13}\RRR_{12}
	\\
	\nonumber
	&\quad
	+(\qqq-\qqq^{-1})\bigg(
	\XX_2\RRR_{21}\mathbf{\Omega}_{13}
	-
	\RRR_{12}^{-1}\XX_1\RRR_{12}\RRR_{13}\mathbf{\Omega}_{23}
	-\RRR_{21}\RRR_{32}^{-1}
	\underbrace{\mathbf{\Omega}_{13}\RRR_{12}\RRR_{23}^{-1}\XX_2\RRR_{23}}_{\mathrm{commute}}
	\bigg)
	\\
\nonumber
&=
	\RRR_{31}^{-1}\RRR_{32}^{-1}
	\underbrace{\RRR_{21}\DDD_3}_{\mathrm{commute}}
	\RRR_{13}^{-1}\XX_1\RRR_{13}
	\RRR_{12}\RRR_{23}^{-1}\XX_2\RRR_{23}
	-\RRR_{31}^{-1}
	\RRR_{32}^{-1}\DDD_3\RRR_{23}^{-1}\XX_2\RRR_{23}
	\RRR_{21}\RRR_{13}^{-1}\XX_1\RRR_{13}\RRR_{12}
	\\
	\nonumber
	&\quad
	+(\qqq-\qqq^{-1})\bigg(	
	\RRR_{31}^{-1}
	\underbrace{\mathbf{\Omega}_{23}\RRR_{21}\RRR_{13}^{-1}\XX_1\RRR_{13}\RRR_{12}}_{\mathrm{commute}}
	-
	\RRR_{12}^{-1}\XX_1\RRR_{12}\RRR_{13}\mathbf{\Omega}_{23}
	\bigg)
	\\
\label{D1PBWLast}
&=
	\RRR_{31}^{-1}\RRR_{32}^{-1}
	\DDD_3
	\bigg(
	\RRR_{21}\RRR_{13}^{-1}\XX_1\RRR_{13}
	\RRR_{12}\RRR_{23}^{-1}\XX_2\RRR_{23}
	-\RRR_{23}^{-1}\XX_2\RRR_{23}
	\RRR_{21}\RRR_{13}^{-1}\XX_1\RRR_{13}\RRR_{12}\bigg)
\end{align}
The expression in the parentheses in the final line (\ref{D1PBWLast}) is the original relation (\ref{D1rel1}) with the generating matrices $\{\XX_1,\XX_2\}$ conjugated by the $R$-matrix.
Its entries are thus in $I_\XX$, and therefore the entries of the entire line (\ref{D1PBWLast}) lie in $I_\XX\mathcal{T}_\DDD$.
\end{proof}

\subsubsection{Interpretation of the cross-relation}\label{D1Cross}
We can rewrite (\ref{D1rel3}) in two different ways:
\begin{align}
\label{D1Cross1}
\DDD_1\RRR_{21}^{-1}\XX_2&= \RRR\XX_2\RRR_{21}\DDD_1\RRR_{21}^{-1}+(\qqq-\qqq^{-1})\mathbf{\Omega}\\
\label{D1Cross2}
\XX_1\RRR\DDD_2&= \RRR^{-1}_{21}\DDD_2\RRR^{-1}\XX_1\RRR-(\qqq-\qqq^{-1})\mathbf{\Omega}
\end{align}
Let us view the matrix $\XX$ as a map $\mathbb{V}^*\otimes\mathbb{V}\rightarrow\DD_1$, and likewise, let us view $\DDD$ as a map $\iota_{\OO_G}(\mathbb{V}^*\otimes\mathbb{V})\rightarrow\DD_1$.
If we let $m_{\DD_1}$ denote the multiplication in $\DD_1$, then upon moving the $R$-matrices from the left to the right, equations (\ref{D1Cross1}) and (\ref{D1Cross2}) can be drawn as follows:
\begin{align*}
\includegraphics{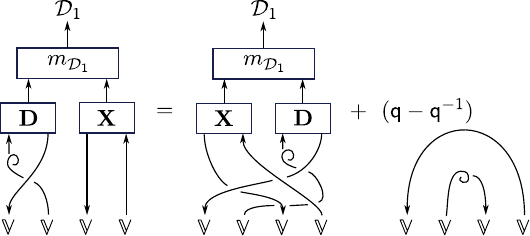}\\
\includegraphics{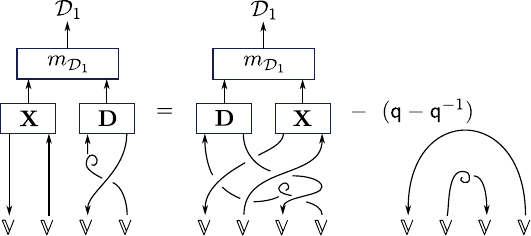}
\end{align*}

We will translate this back to algebra. 
Let $\bullet$ denote the action of $\UU$ on $\mathbb{V}$ or $\mathbb{V}^*$ and let $\bowtie$ denote the action of $\UU$ on $\OO_G\supset\mathcal{R}\hbox{\it ef}\,,\iota_{\OO_G}(\mathcal{R}\hbox{\it ef}\,)$.
For the rest of this paragraph, we will conflate $\OO_G$ with $\kappa(\OO_G)$ and Sweedler notation will be for the coproduct $\Delta_\UU$.
%Denote by $\XX:\mathcal{R}\hbox{\it ef}\,\rightarrow \DD_1(\ZZZ)$ and $\DDD:\iota_{\OO_G}(\mathcal{R}\hbox{\it ef}\,)\rightarrow\DD_1(\ZZZ)$ the induced inclusions.
For $v^*_\XX\otimes v_\XX\in\mathbb{V}^*\otimes\mathbb{V}=\mathrm{im}(\XX)$, $v_\DDD\otimes v^*_\DDD\in\mathbb{V}\otimes\mathbb{V}^*=\mathrm{im}(\DDD)$, $m=m_{kl}\in\mathcal{R}\hbox{\it ef}$, and $\underline{m}=\iota_{\OO_G}(m_{kl})\in\iota_{\OO_G}(\mathcal{R}\hbox{\it ef}\,)$.
\begin{align*}
\DDD(\underline{m})(v^*_\XX\otimes v_\XX)&= \left[\underline{m}_{(1)}\bullet v^*_\XX\otimes \iota_{\UU}({}_sr)\bullet v_\XX\right] \DDD(r_{s}\bowtie\underline{m}_{(2)})+\mathrm{constant}\\
\XX(m)(v_\DDD\otimes v^*_\DDD)&= \left[ r_s\bullet v_\DDD\otimes ({}_sr\bowtie m )_{(1)}\bullet v^*_\DDD\right]\XX\left(({}_sr\bowtie m)_{(2)}\right)+\mathrm{constant}
\end{align*}
Now let us extend the notation $\XX:\mathcal{R}\hbox{\it ef}\,\rightarrow\DD_1$ and $\DDD:\iota_{\OO_G}(\mathcal{R}\hbox{\it ef}\,)\rightarrow\DD_1$.
Because of (\ref{RMatrixCo1}) and (\ref{RMatrixCo2}), the formulas above interact well with the product in the respective domains of $\XX$ and $\DDD$: for $m'=m_{k'l'}$ and $\underline{m}'=\iota_{\OO_G}(m_{k'l'})$, we have
\begin{align*}
\DDD(\underline{m}'\underline{m})(v^*_\XX\otimes v_\XX)&= \left[\underline{m}'_{(1)}\underline{m}_{(1)}\bullet v^*_\XX\otimes \iota_{\UU}({}_tr)\iota_{\UU}({}_sr)\bullet v_\XX\right]\DDD\left( (r_t\bowtie \underline{m}'_{(2)})( r_{s}\bowtie \underline{m}_{(2)})\right)\\
&\qquad+\hbox{ terms in }\{\del_{ij}\}\\
&= \left[(\underline{m}'\underline{m})_{(1)}\bullet v^*_\XX\otimes \iota_{\UU}({}_sr)\bullet v_\XX\right]\DDD\left(r_s\bowtie (\underline{m}'\underline{m})_{(2)}\right)+\hbox{ terms in }\{\del_{ij}\}\\
\XX(m'm)(v_\DDD\otimes v^*_\DDD)&= \left[ r_tr_s\bullet v_\DDD\otimes ({}_tr\bowtie m')_{(1)}({}_sr\bowtie m)_{(1)}\bullet v^*_\DDD\right]\XX\left(({}_tr\bowtie m')_{(2)}({}_sr\bowtie m)_{(2)}\right)\\
&\qquad+\hbox{ terms in }\{x_{ij}\}\\
&= \left[ r_s\bullet v_\DDD\otimes  ({}_sr\bowtie m'm)_{(1)}\bullet v^*_\DDD\right]\XX\left( ({}_sr\bowtie m'm)_{(2)}\right)
+\hbox{ terms in }\{x_{ij}\}
\end{align*}
and likewise for iterated products.
Because $\kappa(\det_\qqq(\mathbf{M}))=\qqq^{-2\omega_n}$ is grouplike, we can conclude:
\begin{align}
\nonumber
\textstyle\det_\qqq(\DDD)x_{ij}&= \qqq^{2}x_{ij}+\hbox{ terms in }\{\del_{kl}\}\\
\label{D1Det}
\textstyle\det_\qqq(\XX)\del_{ij}&= \qqq^{-2}\del_{ij}+\hbox{ terms in }\{x_{kl}\}
\end{align}

\begin{prop}\label{DetOre}
The elements $\{\det_\qqq(\DDD),\det_\qqq(\XX)\}$ generate an Ore set in $\DD_1$.
\end{prop}

\begin{proof}
We will only prove the case of $\det_\qqq(\XX)$.
A consequence of Proposition \ref{D1PBW} is that $\DD_1$ is an integral domain, so cancellability is trivial.
We can endow $\DD_1$ with a filtration such that
\begin{equation*}
\begin{aligned}
\deg(x_{ij})&=0,
&\deg(\del_{ij})&=1
\end{aligned}
\end{equation*}
Equation (\ref{D1Det}) then states that $\det_\qqq(\XX)$ $\qqq^2$-commutes with elements of $\{\del_{ij}\}$ in the associated graded.
The proof proceeds as in that of Lemma \ref{DellDet}.
\end{proof}

\subsubsection{Homomorphism into $\DD_0$}
We discovered the relations (\ref{D1rel1})-(\ref{D1rel3}) by emulating the moves from 3.1 of \cite{BEF} and proving the following:
\begin{lem}\label{D1D0Hom}
For $\ZZZ\in\CC^\times$, the map
\begin{equation*}
\begin{aligned}
\XX&\mapsto \mathbf{A},&
\DDD&\mapsto \mathbf{A}^{-1}\left(\ZZZ^{-1}\mathbf{B}^{-1}-\mathbf{I} \right)
\end{aligned}
\end{equation*}
defines a $\UU$-equivariant algebra homomorphism $\Psi_\ZZZ^1:\DD_1\rightarrow\DD_0$.
\end{lem}

\begin{proof}
The $\UU$-equivariance statement follows from the homomorphism statement because the matrix maps are $\UU$-equivariant and both $\DD_1$ and $\DD_0$ are $\UU$-module algebras.
To show that $\Psi_\ZZZ^1$ is an algebra homomorphism, we only need to check (\ref{D1rel2}) and (\ref{D1rel3}).
For (\ref{D1rel2}), let us first unwind both sides of the relation:
\begin{align*}
\RRR^{-1}\Psi_\ZZZ^1(\DDD_1)\RRR_{21}^{-1}\Psi_\ZZZ^1(\DDD_2)
&=
\ZZZ^{-2}\underbrace{\RRR^{-1}\mathbf{A}_1^{-1}\mathbf{B}_1^{-1}\RRR_{21}^{-1}\mathbf{A}^{-1}_2\mathbf{B}_2^{-1}}_{(1)}
+
\underbrace{\RRR^{-1}\mathbf{A}_1^{-1}\RRR_{21}^{-1}\mathbf{A}_2^{-1}}_{(2)}\\
&\quad-
\ZZZ^{-1}\bigg(\underbrace{\RRR^{-1}\mathbf{A}^{-1}_1\RRR^{-1}_{21}\mathbf{A}^{-1}_2\mathbf{B}_2^{-1}
+
\RRR^{-1}\mathbf{A}_1^{-1}\mathbf{B}_1^{-1}\RRR^{-1}_{21}\mathbf{A}_2^{-1}}_{(3)}\bigg)\\
\Psi^1(\DDD_2)\RRR^{-1}\Psi^1(\DDD_1)\RRR_{21}^{-1}
&= 
\ZZZ^{-2}\underbrace{\mathbf{A}_2^{-1}\mathbf{B}_2^{-1}\RRR^{-1}\mathbf{A}_1^{-1}\mathbf{B}_1^{-1}\RRR_{21}^{-1}}_{(1)}
+
\underbrace{\mathbf{A}_2^{-1}\RRR^{-1}\mathbf{A}_1^{-1}\RRR^{-1}_{21}}_{(2)}\\
&\quad-
\ZZZ^{-1}\bigg(\underbrace{\mathbf{A}_2^{-1}\mathbf{B}_2^{-1}\RRR^{-1}\mathbf{A}_1^{-1}\RRR^{-1}_{21}
+
\mathbf{A}_2^{-1}\RRR^{-1}\mathbf{A}_1^{-1}\mathbf{B}_1^{-1}\RRR^{-1}_{21}}_{(3)}\bigg)
\end{align*}
We will show that parts labeled (1)-(4) in both expressions are equal.
\begin{enumerate}
\item This one is quite direct:
\begin{align*}
\RRR^{-1}\mathbf{A}_1^{-1}
\underbrace{\mathbf{B}_1^{-1}\RRR_{21}^{-1}\mathbf{A}^{-1}_2}_{(\ref{DAB})}
\mathbf{B}_2^{-1}
&= 
\underbrace{\RRR^{-1}\mathbf{A}_1^{-1}\RRR_{21}^{-1}\mathbf{A}_2^{-1}}_{(\ref{DAA})}
\underbrace{\RRR^{-1}\mathbf{B}_1^{-1}\RRR_{21}^{-1}\mathbf{B}_2^{-1}}_{(\ref{DBB})}\\
&= \mathbf{A}_2^{-1}
\underbrace{\RRR^{-1}\mathbf{A}_1^{-1}\RRR_{21}^{-1}\mathbf{B}_2^{-1}\RRR^{-1}}_{(\ref{DAB})}
\mathbf{B}_1^{-1}\RRR_{21}^{-1}\\
&= \mathbf{A}_2^{-1}\mathbf{B}_2^{-1}\RRR^{-1}\mathbf{A}_1^{-1}\mathbf{B}_1^{-1}\RRR^{-1}_{21}
\end{align*}
\item This follows from (\ref{DAA}).
\item We use the Hecke relations (\ref{HeckeCond}) to adjust some of the $R$-matrices at the ends:
\begin{align*}
\RRR^{-1}\mathbf{A}^{-1}_1\RRR^{-1}_{21}\mathbf{A}^{-1}_2\mathbf{B}_2^{-1}+
\underbrace{\RRR^{-1}}_{(\ref{HeckeCond})}
\mathbf{A}_1^{-1}\mathbf{B}_1^{-1}\RRR^{-1}_{21}\mathbf{A}_2^{-1}
&\overset{?}{=}
\mathbf{A}_2^{-1}\mathbf{B}_2^{-1}\RRR^{-1}\mathbf{A}_1^{-1}
\underbrace{\RRR^{-1}_{21}}_{(\ref{HeckeCond})}
+
\mathbf{A}_2^{-1}\RRR^{-1}\mathbf{A}_1^{-1}\mathbf{B}_1^{-1}\RRR_{21}^{-1}
\\
\RRR^{-1}\mathbf{A}^{-1}_1\RRR^{-1}_{21}\mathbf{A}^{-1}_2\mathbf{B}_2^{-1}
+\RRR_{21}\mathbf{A}_1^{-1}\mathbf{B}_1^{-1}\RRR^{-1}_{21}\mathbf{A}_2^{-1}
&\overset{?}{=}
\mathbf{A}_2^{-1}\mathbf{B}_2^{-1}\RRR^{-1}\mathbf{A}_1^{-1}\RRR
+
\mathbf{A}_2^{-1}\RRR^{-1}\mathbf{A}_1^{-1}\mathbf{B}_1^{-1}\RRR_{21}^{-1}\\
-(\qqq-\qqq^{-1})
\underbrace{\mathbf{\Omega}\mathbf{A}_1^{-1}\mathbf{B}_1^{-1}\RRR_{21}^{-1}\mathbf{A}_2^{-1}}_{\mathrm{commute}}&
\quad
-(\qqq-\qqq^{-1})\mathbf{A}_2^{-1}\mathbf{B}_2^{-1}\RRR^{-1}\mathbf{A}_1^{-1}\mathbf{\Omega}
\\
\RRR^{-1}\mathbf{A}^{-1}_1\RRR^{-1}_{21}\mathbf{A}^{-1}_2\mathbf{B}_2^{-1}
+\RRR_{21}\mathbf{A}_1^{-1}\mathbf{B}_1^{-1}\RRR^{-1}_{21}\mathbf{A}_2^{-1}
&\overset{?}{=}
\mathbf{A}_2^{-1}\mathbf{B}_2^{-1}\RRR^{-1}\mathbf{A}_1^{-1}\RRR
+
\mathbf{A}_2^{-1}\RRR^{-1}\mathbf{A}_1^{-1}\mathbf{B}_1^{-1}\RRR_{21}^{-1}
\end{align*}
We then have:
\begin{align*}
&\quad\underbrace{\RRR^{-1}\mathbf{A}^{-1}_1\RRR^{-1}_{21}\mathbf{A}^{-1}_2}_{(\ref{DAA})}
\mathbf{B}_2^{-1}
+\RRR_{21}\mathbf{A}_1^{-1}
\underbrace{\mathbf{B}_1^{-1}\RRR^{-1}_{21}\mathbf{A}_2^{-1}}_{(\ref{DAB})}\\
&=\mathbf{A}_2^{-1}
\underbrace{\RRR^{-1}\mathbf{A}_1^{-1}\RRR^{-1}_{21}\mathbf{B}_2^{-1}}_{(\ref{DAB})}
+
\underbrace{\RRR_{21}\mathbf{A}_1^{-1}\RRR_{21}^{-1}\mathbf{A}_2^{-1}\RRR^{-1}}_{(\ref{DAA})}
\mathbf{B}_1^{-1}\RRR_{21}^{-1}
\\
&= \mathbf{A}_2^{-1}\mathbf{B}_2^{-1}\RRR^{-1}\mathbf{A}_1^{-1}\RRR
+\mathbf{A}_2^{-1}\RRR^{-1}\mathbf{A}_1^{-1}\mathbf{B}_1^{-1}\RRR^{-1}_{21}
\end{align*}
\end{enumerate}

Finally, for (\ref{D1rel3}), we have:
\begin{align*}
\Psi_\ZZZ^1(\DDD_1)\RRR_{21}^{-1}\Psi_\ZZZ^1(\XX_2)\RRR_{21}
&=\ZZZ^{-1}\mathbf{A}_1^{-1}
\underbrace{\mathbf{B}_1^{-1}\RRR_{21}^{-1}\mathbf{A}_2\RRR_{21}}_{(\ref{DAB})}
-
\underbrace{\mathbf{A}_1^{-1}\RRR_{21}^{-1}\mathbf{A}_2\RRR_{21}}_{(\ref{DAA})}
\\
&=
\ZZZ^{-1}\underbrace{\mathbf{A}_1^{-1}\RRR\mathbf{A}_2\RRR_{21}}_{(\ref{DAA})}
\mathbf{B}_1^{-1}
-\RRR\mathbf{A}_2
\underbrace{\RRR^{-1}}_{(\ref{HeckeCond})}
\mathbf{A}_1^{-1}\\
&= \ZZZ^{-1}\RRR\mathbf{A}_2\RRR_{21}\mathbf{A}_1^{-1}\mathbf{B}_1^{-1}
-\left( \RRR\mathbf{A}_2\RRR_{21}\mathbf{A}_1^{-1} - (\qqq-\qqq^{-1})\RRR\mathbf{A}_2\mathbf{\Omega}\mathbf{A}_1^{-1}\right)\\
&= \RRR\Psi_\ZZZ^1(\mathbf{X}_2)\RRR_{21}\Psi_\ZZZ^1(\DDD_1)+(\qqq-\qqq^{-1})\RRR\mathbf{\Omega}\qedhere
\end{align*}
\end{proof}

\subsection{Quantum Weyl algebra}
Finally, we address the arrow connecting to the square vertex in (\ref{CyclicQ}).
The algebra assigned to this arrow is the \textit{quantum Weyl algebra}, first defined in \cite{GiaZhaQW}.

\subsubsection{Definition of $\WW$ and PBW}
The rank $n$ quantum Weyl algebra $\WW$ has generators
\[
\{x_i,\del_i\, |\, 1\le i\le n\}
\] 
and relations
\begin{align}
\begin{split}
x_ix_j&= \qqq x_jx_i\hbox{ for }i>j\\
\del_i\del_j&=\qqq^{-1}\del_j\del_i\hbox{ for }i>j\\
\del_ix_j&=\qqq x_j\del_i\hbox{ for }i\not=j\\
\del_ix_i&=1+\qqq^2x_i\del_i+(\qqq^2-1)\sum_{j<i}x_j\del_j
\end{split}\label{WeylPres}
\end{align}
We can put these generators into vectors:
\begin{equation*}
\begin{aligned}
\mathbf{x}&:= 
\begin{pmatrix}
x_1 & \cdots & x_n
\end{pmatrix}
= \sum_{i=1}^n e^i\otimes x_i,
&
\mathbf{d}&:= 
\begin{pmatrix}
\del_1\\
\vdots\\
\del_n
\end{pmatrix}
= \sum_{i=1}^n e_i\otimes \del_i
\\
\end{aligned}
\end{equation*}
With this notation, the relations (\ref{WeylPres}) can be expressed using the $R$-matrix:
\begin{equation}
\begin{aligned}
\qqq \mathbf{x}_1\mathbf{x}_2&= \mathbf{x}_2\mathbf{x}_1\RRR\\
\qqq \mathbf{d}_1\mathbf{d}_2&= \RRR\mathbf{d}_2\mathbf{d}_1\\
\qqq^{-1}\mathbf{d}_2\mathbf{x}_1&= \mathbf{x}_1\RRR\mathbf{d}_2+\qqq^{-1}\sum_{i=1}^ne^i\otimes e_i\otimes 1
\end{aligned}
\label{WeylR}
\end{equation}

Let
\[
\widetilde{\del}_i:=(\qqq-\qqq^{-1})\del_i
\]
and let $\WW^\ZZ$ to be the $\CC[\qqq^{\pm 1}]$-subalgebra generated by $\{x_i,\widetilde{\del}_i\}_{i=1}^n$.
Define a standard monomial for $\WW^\ZZ$ and $\WW$ to be
\[
x_{i_1}\cdots x_{i_M}\widetilde{\del}_{j_1}\cdots\widetilde{\del}_{j_N}
\]
such that the sequences $(i_1,\ldots, i_M)$ and $(j_1,\ldots, j_N)$ are in nondecreasing order.
\begin{prop}[\cite{GiaZhaQW,JordanMult}]
The standard monomials form a $\CC[\qqq^{\pm 1}]$-basis of $\WW^\ZZ$.
Thus, $\WW$ is a flat deformation of $\OO(\CC^{n}\times\CC^n)$.
\end{prop}

\subsubsection{Quantum symmetric algebra}\label{QSym}
Because $\RRR$ satisfies the Hecke relation (\ref{HeckeCond}), we can construct a $\qqq$-symmetrizer and apply it to $\mathbb{V}^{\otimes m}$ nd $(\mathbb{V}^*)^{\otimes m}$ to obtain the \textit{$m$th quantum symmetric powers} $S_\qqq^m\mathbb{V}$ and $S_\qqq^m\mathbb{V}^*$, respectively.
Moreover, we obtain \textit{quantum symmetric algebras}:
\begin{equation*}
\begin{aligned}
S_\qqq&\mathbb{V}:=\bigoplus_{m=0}^\infty S_\qqq^m\mathbb{V},&
S_\qqq&\mathbb{V}^*:=\bigoplus_{m=0}^\infty S_\qqq^m\mathbb{V}^*
\end{aligned}
\end{equation*}
These algebras are $\UU$-equivariant.

In what follows, let $\WW^{\mathbf{x}},\WW^{\mathbf{d}}\subset\WW$ be the subalgebras generated by $\{x_i\}$ and $\{\del_i\}$, respectively.
\begin{thm}[\cite{GiaZhaQW}]
We have the following:
\begin{enumerate}
\item The map $e_i\mapsto x_i$ induces an isomorphism of algebras $S_\qqq\mathbb{V}\cong \WW^{\mathbf{x}}$.
\item The map $e^i\mapsto \del_i$ induces an isomorphism of algebras $S_\qqq\mathbb{V}^*\cong \WW^{\mathbf{d}}$.
\item $\WW$ has a structure of a $\UU$-equivariant algebra that resticts to those of $\WW^{\mathbf{x}}$ and $\WW^{\mathbf{d}}$ obtained from (1) and (2).
As a $\UU$-module, $\WW\cong\WW^{\mathbf{x}}\otimes\WW^{\mathbf{d}}\cong S_\qqq\mathbb{V}\otimes S_\qqq\mathbb{V}^*$.
\end{enumerate}
\end{thm}
\noindent Part (3) allows us to define a $\UU$-equivariant action of $\WW$ on $S_\qqq\mathbb{V}$.
Namely, let $I^{\mathbf{d}}\subset\WW$ be the left ideal
\[
I^{\mathbf{d}}:=\sum_{i=1}^n\WW\del_i
\]
We then have that the $\UU$-equivariant $\WW$-module $\WW/I^{\mathbf{d}}$ is isomorphic to $S_\qqq\mathbb{V}$ as a $\UU$-module.

Finally, let us go into more detail on the $\WW$-action on $S_\qqq\mathbb{V}$.
The ordered monomials
\[
x_1^{k_1}\cdots x_n^{k_n}
\]
give a basis of $S_\qqq\mathbb{V}$, and $\WW^{\mathbf{x}}$ acts by multiplication followed by reordering.
For $\WW^{\mathbf{d}}$, we have (cf. \cite{KlimSchm}):
\[
\del_i\cdot x_1^{k_1}\cdots x_n^{k_n}= 
(\qqq x)^{k_1}\cdots (\qqq x_{i-1})^{k_{i-1}}
[k_i]_{\qqq^2} x_i^{k_i-1}x_{i+1}^{k_{i+1}}\cdots x_n^{k_n}
\]

\subsubsection{Difference operators}\label{WeylDiff}
We will reinterpret the $\UU$- and $\WW$-module structures on $S_\qqq\mathbb{V}$ in terms of $\qqq$-difference operators acting on polynomials.
Namely, we translate over these structures over to the polynomial ring $\CC(\qqq)[\mathbf{z}_n]:=\CC(\qqq)[z_1,\ldots, z_n]$ via the map
\[
x_1^{k_1}\cdots x_n^{k_n}\mapsto z_1^{k_1}\cdots z_n^{k_n}
\]
Let $\tau_{z_i,\qqq}$ be the $\qqq$-shift operator
\[
\tau_{z_i,\qqq}z_j=\qqq^{\delta_{i,j}}z_j
\]
and for $\lambda=(\lambda_1,\ldots, \lambda_n)\in\ZZ^n$, we use monomial notation:
\begin{equation*}
\begin{aligned}
z^\lambda&:= z_1^{\lambda_1}\cdots z_n^{\lambda_n},&
\tau_{\qqq}^\lambda&:=\tau_{z_1,\qqq}^{\lambda_1}\cdots\tau_{z_n,\qqq}^{\lambda_n}
\end{aligned}
\end{equation*}

Consider the following rings of $\qqq$-difference operators:
\begin{align*}
\DD \hbox{\it iff}\,_\qqq(\mathbf{z}_n)&:=
\left\{ \sum_{\lambda,\mu\in\ZZ^n}a_{\lambda,\mu}z^\lambda\tau_{\qqq}^\mu\,\middle|\,
\begin{array}{l}
a_{\lambda,\mu}\in\CC(\qqq)\\
\hbox{and only finitely many }a_{\lambda,\mu}\not=0
\end{array}
\right\}\\
\DD \hbox{\it iff}\,_\qqq^+(\mathbf{z}_n)&:=
\left\{ D\in\DD\hbox{\it iff}\,_\qqq(\mathbf{z}_n)\,\middle|\,
\begin{array}{l}
Df\in\CC(\qqq)[\mathbf{z}_n]\\
\hbox{for all }f\in\CC(\qqq)[\mathbf{z}_n]
\end{array}
\right\}
\end{align*}
The actions of $\UU$ and $\WW$ on $\CC(\qqq)[\mathbf{z}_n]$ come from a homomorphism $\mathfrak{qdiff}:\WW\rtimes\UU\rightarrow\DD\hbox{\it iff}\,_\qqq^+(\mathbf{z}_n)$ given by
\begin{align*}
\mathfrak{qdiff}(\qqq^{\epsilon_i})&= \tau_{z_i,\qqq} & \mathfrak{qdiff}(x_i)&= z_i\tau_{z_1,\qqq}\cdots \tau_{z_{i-1},\qqq}\\
\mathfrak{qdiff}(E_i)&= \frac{z_i}{z_{i+1}}\left( \frac{\tau_{z_{i+1},\qqq}-\tau_{z_{i+1},\qqq}^{-1}}{\qqq-\qqq^{-1}} \right) & \mathfrak{qdiff}(\del_i)&= z_i^{-1}\tau_{z_1,\qqq}\cdots\tau_{z_{i-1},\qqq}\left( \frac{\tau_{z_i,\qqq}^2-1}{\qqq^2-1} \right)\\
\mathfrak{qdiff}(F_i)&= \frac{z_{i+1}}{z_i}\left( \frac{\tau_{z_i,\qqq}-\tau_{z_i,\qqq}^{-1}}{\qqq-\qqq^{-1}} \right)
\end{align*}

\section{Harish-Chandra isomorphism}

\subsection{Moment maps}\label{Moment}
Let $H$ be a Hopf algebra, let $A$ be an $H$-module algebra, and let $\bullet$ denote the $H$ action on $A$.
A \textit{(quantum) moment map} (in the sense of \cite{VarVassRoot}) is a ring homomorphism $\mu: H\rightarrow A$ such that for $h\in H$ and $a\in A$:
\begin{equation}
\mu(h)a=\left(h_{(1)}\bullet a\right)\mu(h_{(2)})
\label{MomentDef}
\end{equation}
More generally, this notion makes sense if $H'\subset H$ is a left coideal subalgebra and $\mu$ is only defined on $H'$.
For what follows, we will be interested in the case $H=\UU$, $H'=\kappa(\OO_G)$, and $A$ is one of the algebras of quantum differential operators from Section \ref{QDO} above.
When discussing the domain of a moment map, we will abuse notation and conflate $\OO_G$ with its image in $\UU$ via $\kappa$.

For what will appear below, $A$ will always be a locally-finite $\UU$-module algebra and $\mu:\OO_G\rightarrow A$ will always be $\UU$-equivariant.
In this case, the moment map equation (\ref{MomentDef}) on $V^*\otimes V\subset\OO_G$ looks like (cf. \cite{GanJorSaf}):
\begin{equation}
\includegraphics{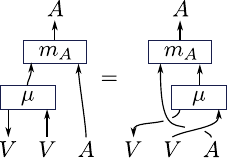}
\label{MomentPic}
\end{equation}
where $m_A$ is the multiplication in $A$.
%In terms of the generating matrix $\mathbf{M}$ for $\OO_G$ and a generating matrix $\XX$ for $A$ (as it often appears in Section \ref{QDO}), directly writing down (\ref{MomentPic}) as a relation is a bit messy due to

\subsubsection{Moment maps for $\DD_\ell$ for $\ell>1$}\label{MomentEll}
For $\ell>1$ and $a\in\ZZ/\ell\ZZ$, consider the matrices of elements:
\begin{equation*}
\begin{aligned}
{}^L\underline{\YY}^{(a)}&:= \mathbf{I}+\XX^{(a)}\DDD^{(a)},&
{}^R\YY^{(a)}&:= \mathbf{I}+\DDD^{(a)}\XX^{(a)}
\end{aligned}
%\label{DellMu}
\end{equation*}
Let $\DD_\ell^{(a)}$ be the subalgebra generated by the coordinates of $\XX^{(a)}$ and $\DDD^{(a)}$; note that the coordinates of ${}^L\underline{\YY}^{(a)}$ and ${}^R\YY^{(a)}$ lie in $\DD_\ell^{(a)}$.
Next, recall the notation we introduced in \ref{DellEq}.
Using this notation, $\DD_\ell^{(a)}$ is a $\UU_{(a)}\otimes\UU_{(a+1)}$-module, where
\begin{equation*}
\begin{aligned}
x_{ij}&= e^i_{(a)}\boxtimes e_j^{(a+1)}\in\mathbb{V}^*_{(a)}\boxtimes\mathbb{V}_{(a+1)},&
\del_{ij}&= e_j^{(a)}\boxtimes e^j_{(a+1)}\in\mathbb{V}_{(a)}\boxtimes\mathbb{V}^*_{(a+1)}
\end{aligned}
\end{equation*}

\begin{prop}[\cite{JordanMult}]\label{DellMuProp}
${}^R\YY^{(a)}$ and ${}^L\underline{\YY}^{(a)}$ satisfy the following:
\begin{enumerate}
\item The map $\mathbf{M}\mapsto{}^R\YY^{(a)}$ defines a $\UU$-equivariant ring homomorphism $\mu_R^{(a)}:\mathcal{R}\hbox{\it ef}\,\rightarrow \DD_\ell^{(a)}$.
\item The map $\mathbf{M}^{-1}\mapsto{}^L\underline{\YY}^{(a)}$ defines a $\UU$-equivariant ring homomorphism $\mu_L^{(a)}:\iota_G(\mathcal{R}\hbox{\it ef}\,)\rightarrow \DD_\ell^{(a)}$.
\item The elements $\mu_R^{(a)}(\det_\qqq(\mathbf{M}))$ and $\mu_L^{(a)}(\det_\qqq(\mathbf{M})^{-1})$ generate Ore sets of $\DD_\ell$.
Let $\DD_\ell^{(a),\star}$ denote the localization of $\DD_\ell^{(a)}$.
The extension $\mu_R^{(a)}:\OO_G\rightarrow \DD_\ell^{(a),\star}$ is a moment map for the $\UU_{(a+1)}$-action, and the extension $\mu_L^{(a)}:\OO_G\rightarrow \DD_\ell^{(a),\star}$ is a moment map for the $\UU_{(a)}$-action.
\end{enumerate}
\end{prop}

Let ${}^L\YY^{(a)}:=\big({}^L\underline{\YY}^{(a)}\big)^{-1}$ and let $\DD_\ell^\star$ denote the localization of $\DD_\ell$ at the determinants from part (3) above.
Translating (\ref{MomentPic}) shows that part (3) is equivalent to the identities:
\begin{align}
\label{YL1}
{}^L\YY^{(a)}_1\RRR\XX_2^{(a)}&= \RRR_{21}^{-1}\XX^{(a)}_2{}^L\YY^{(a)}_1\\
\label{YL2}
{}^L\YY^{(a)}_1\DDD^{(a)}_2&= \DDD^{(a)}_2\RRR_{21}{}^L\YY^{(a)}_1\RRR\\
\label{YR1}
{}^R\YY^{(a)}_1\XX^{(a)}_2&= \XX^{(a)}_2\RRR_{21}{}^R\YY^{(a)}_1\RRR\\
\label{YR2}
{}^R\YY^{(a)}_1\RRR\DDD^{(a)}_2&= \RRR_{21}^{-1}\DDD^{(a)}_2{}^R\YY^{(a)}_1
\end{align}
%We note that it is awkward to write (\ref{YL1}) and (\ref{YR2}) without the $R$-matrix on the left.
Using the relations for $\DD_\ell$, we can deduce that ${}^L\YY^{(a)}$ satisfies:
\begin{align}
\label{YL3}
{}^L\YY^{(a)}_1\XX_2^{(a-1)}&= \XX_2^{(a-1)}\RRR_{21}{}^L\YY^{(a)}_1\RRR_{21}^{-1}\\
\label{YL4}
{}^L\YY^{(a)}_1\XX_2^{(b)}&= \XX_2^{(b)}{}^L\YY^{(a)}_1\hbox{ for }b\not=a,a-1\\
\label{YL5}
\RRR_{21}{}^L\YY^{(a)}_1\RRR_{21}^{-1}\DDD_2^{(a-1)}&= \DDD_2^{(a-1)}{}^L\YY^{(a)}_1\\
\label{YL6}
{}^L\YY^{(a)}_1\DDD_2^{(b)}&= \DDD_2^{(b)}{}^L\YY^{(a)}_1\hbox{ for }b\not=a,a-1
\end{align}
On the other hand, ${}^R\YY^{(a)}$ satisfies:
\begin{align}
\label{YR3}
\RRR^{-1}{}^R\YY^{(a)}_1\RRR\XX_2^{(a+1)}&= \XX_2^{(a+1)}{}^R\YY^{(a)}_1\\
\label{YR4}
{}^R\YY^{(a)}_1\XX_2^{(b)}&= \XX_2^{(b)}{}^R\YY^{(a)}_1\hbox{ for }b\not=a,a+1\\
\label{YR5}
{}^R\YY^{(a)}_1\DDD_2^{(a+1)}&= \DDD_2^{(a+1)}\RRR^{-1}{}^R\YY^{(a)}_1\RRR\\
\label{YR6}
{}^R\YY^{(a)}_1\DDD_2^{(b)}&= \DDD_2^{(b)}{}^R\YY^{(a)}_1\hbox{ for }b\not=a,a+1
\end{align}

\begin{prop}[\cite{JordanMult,GanJorSaf}]
The map $\mu_{\DD_\ell}^{(a)}:\OO_G\rightarrow\DD_\ell^\star$ given by
\begin{align*}
\mu_{\DD_\ell}^{(a)}(\mathbf{M})&={}^L\YY^{(a)}{}^R\YY^{(a-1)}\\
&=\left( \mathbf{I}+\XX^{(a)}\DDD^{(a)} \right)^{-1}\left( \mathbf{I}+\DDD^{(a-1)}\XX^{(a-1)} \right)
\end{align*}
is a moment map for the $\UU_{(a)}$-action.
Consequently,
\[
\mu_{\DD_\ell}(\mathbf{M}_1\otimes\cdots\otimes\mathbf{M}_\ell)=\mu_{\DD_\ell}^{(1)}(\mathbf{M}_1)\cdots\mu_{\DD_\ell}^{(\ell)}(\mathbf{M}_\ell)
\]
is a moment map for the $\UU^{\otimes\ell}$-action.
\end{prop}

%\begin{proof}
%This is shown in \textit{loc. cit.} through general principles, but one can check directly that $\mu_{\DD_\ell}^{(a)}(\mathbf{M})$ satisfies the appropriate modifications of (\ref{YL1})-(\ref{YR2}):
%\begin{align*}
%{}^L\YY^{(a)}_1{}^R\YY^{(a-1)}_1\RRR\XX_2^{(a)}&= \RRR_{21}^{-1}\XX^{(a)}_2{}^L\YY^{(a)}_1{}^R\YY^{(a-1)}_1\\
%{}^L\YY^{(a)}_1{}^R\YY^{(a-1)}_1\DDD^{(a)}_2&= \DDD^{(a)}_2\RRR_{21}{}^L\YY^{(a)}_2{}^R\YY^{(a-1)}_1\RRR\\
%{}^L\YY^{(a)}_1{}^R\YY^{(a-1)}_1\XX^{(a-1)}_2&= \XX^{(a-1)}_2\RRR_{21}{}^L\YY^{(a)}_1{}^R\YY^{(a-1)}_1\RRR\\
%{}^L\YY^{(a)}_1{}^R\YY^{(a-1)}_1\RRR\DDD^{(a-1)}_2&= \RRR_{21}^{-1}\DDD^{(a-1)}_2{}^L\YY^{(a)}_1{}^R\YY^{(a-1)}_1\\
%{}^L\YY^{(a)}_1{}^R\YY^{(a-1)}_1\XX_2^{(b)}&= \XX_2^{(b)}{}^L\YY^{(a)}_1{}^R\YY^{(a-1)}_1\hbox{ for }b\not=a, a-1\\
%{}^L\YY^{(a)}_1{}^R\YY^{(a-1)}_1\DDD_2^{(b)}&= \DDD_2^{(b)}{}^L\YY^{(a)}_1{}^R\YY^{(a-1)}_1\hbox{ for }b\not=a, a-1\qedhere
%\end{align*}
%\end{proof}
%To extend Theorem \ref{DellMuThm} to a statment about $\DD_\ell$, let us consider how ${}^R\YY^{(a)}$ and ${}^L\YY^{(a)}$ interacts with ${}^\star\DD_\ell^{(b)}$ for $b\not=a$.
%These interactions are facilitated by (\ref{BraidedEll}) and (\ref{Braided2}).

\subsubsection{Moment maps for $\DD_0$ and $\WW$}
Recall that in the $\UU$-module structure for $\DD_0$, we view the entries of the generating matrices $\mathbf{A}$ and $\mathbf{B}$ as elements of $\mathbb{V}^*\otimes\mathbb{V}$.

\begin{prop}[\cite{JordanMult,VarVassRoot}]
The assignment
\[
\mathbf{M}\mapsto\mathbf{B}\mathbf{A}^{-1}\mathbf{B}^{-1}\mathbf{A}
\]
defines a moment map $\mu_{\DD_0}:\OO_G\rightarrow \DD_0$.
\end{prop}

\begin{prop}[\cite{JordanMult}]
The assignment
\[
\mathbf{M}\mapsto \qqq^{-2}\left(\mathbf{I}+(\qqq^2-1)\mathbf{d}\mathbf{x}\right)
\]
defines a ring homomorphism $\mu_\WW:\mathcal{R}\hbox{\it ef}\,\rightarrow\WW$.
The element $\mu_\WW(\det_\qqq(\mathbf{M}))$ generates an Ore set in $\WW$.
Let $\WW^\star$ be the localization.
The extension $\mu_\WW:\OO_G\rightarrow\WW^\star$ is a moment map.
\end{prop}

\subsubsection{Moment map for $\DD_1$}\label{D1Moment}
As elements of a $\UU$-module, the entries of $\XX$ are elements of $\mathbb{V}^*\otimes\mathbb{V}$ and the entries of $\DDD$ are elements of $\iota_{\OO_G}(\mathbb{V}^*\otimes\mathbb{V})\cong\mathbb{V}\otimes\mathbb{V}^*$.
Like in the case $\ell >1$, we will start with approximations to a moment map for the $\UU$-action on each of the tensorands.
Let:
\begin{equation*}
\begin{aligned}
{}^L\underline{\YY}&:=\mathbf{I}+\XX\DDD,&
{}^R\YY&:=\mathbf{I}+\DDD\XX 
\end{aligned}
\end{equation*}
\begin{lem}
The matrices ${}^L\underline{\YY}$ and ${}^R\YY$ satisfy:
\begin{align}
\label{D1YL1}
\RRR^{-1}\YLu_1\RRR_{21}^{-1}\YLu_2&= \YLu_2\RRR^{-1}\YLu_1\RRR^{-1}_{21}\\
\label{D1YR1}
\RRR_{21}\YR_1\RRR\YR_2&= \YR_2\RRR_{21}\YR_1\RRR\\
\label{D1YL2}
\YLu_1\RRR_{21}^{-1}\XX_2&= \RRR\XX_2\RRR_{21}\YLu_1\RRR_{21}^{-1}\\
\label{D1YL3}
\YLu_1\RRR_{21}^{-1}\DDD_2&= \RRR^{-1}_{21}\DDD_2\RRR^{-1}\YLu_1\RRR_{21}^{-1}\\
\label{D1YR2}
\YR_1\RRR\XX_2&= \RRR\XX_2\RRR_{21}\YR_1\RRR\\
\label{D1YR3}
\YR_1\RRR\DDD_2&= \RRR^{-1}_{21}\DDD_2\RRR^{-1}\YR_1\RRR
\end{align}
\end{lem}

\begin{proof}
We will only prove (\ref{D1YL1}) and (\ref{D1YL2}).
For (\ref{D1YL1}), we have:
\begin{align*}
\RRR^{-1}\YLu_1\RRR_{21}^{-1}\YLu_2
&=\RRR^{-1}\RRR_{21}^{-1}+\RRR^{-1}\XX_1\DDD_1\RRR_{21}^{-1}+\RRR^{-1}\RRR_{21}^{-1}\XX_2\DDD_2 +\RRR^{-1}\XX_1\DDD_1\RRR_{21}^{-1}\XX_2\DDD_2\\
\YLu_2\RRR^{-1}\YLu_1\RRR^{-1}_{21}
&=\RRR^{-1}\RRR_{21}^{-1}+\XX_2\DDD_2\RRR^{-1}\RRR_{21}^{-1}+\RRR^{-1}\XX_1\DDD_1\RRR^{-1}_{21}+\XX_2\DDD_2\RRR^{-1}\XX_1\DDD_1\RRR_{21}^{-1}
\end{align*}
All we need to show is:
\begin{align*}
\RRR^{-1}\RRR_{21}^{-1}\XX_2\DDD_2 
+
\underbrace{\RRR^{-1}}_{(\ref{HeckeCond})}
\XX_1\DDD_1\RRR_{21}^{-1}\XX_2\DDD_2
&\overset{?}{=}
\XX_2\DDD_2\RRR^{-1}\RRR_{21}^{-1}
+\XX_2\DDD_2\RRR^{-1}\XX_1\DDD_1
\underbrace{\RRR_{21}^{-1}}_{(\ref{HeckeCond})}\\
\RRR^{-1}\RRR_{21}^{-1}\XX_2\DDD_2 
+
\RRR_{21}
\XX_1\DDD_1\RRR_{21}^{-1}\XX_2\DDD_2
&\overset{?}{=}
\XX_2\DDD_2\RRR^{-1}\RRR_{21}^{-1}
+\XX_2\DDD_2\RRR^{-1}\XX_1\DDD_1
\RRR
\end{align*}
We start from the left:
\begin{align*}
&\quad
	\RRR^{-1}
	\underbrace{\RRR_{21}^{-1}}_{(\ref{HeckeCond})}
	\XX_2\DDD_2 
	+\RRR_{21}\XX_1
	\underbrace{\DDD_1\RRR_{21}^{-1}\XX_2}_{(\ref{D1rel3})}
	\DDD_2\\
&= 		
	\underbrace{\RRR_{21}\XX_1\RRR\XX_2}_{(\ref{D1rel1})}
	\underbrace{\RRR_{21}\DDD_1\RRR_{21}^{-1}\DDD_2}_{(\ref{D1rel2})}
	+\XX_2\DDD_2
	+(\qqq-\qqq^{-1})\bigg(
	\RRR_{21}
	\underbrace{\XX_1\RRR\mathbf{\Omega}}_{\mathrm{commute}}
	\RRR_{21}^{-1}\DDD_2
	-\RRR^{-1}\mathbf{\Omega}\XX_2\DDD_2 
	\bigg)\\
&= 
	\XX_2
	\underbrace{\RRR_{21}\XX_1\RRR\DDD_2\RRR^{-1}}_{(\ref{D1rel3})}
	\DDD_1\RRR
	+\XX_2\DDD_2
	+(\qqq-\qqq^{-1})\bigg(
	\underbrace{\RRR_{21}-\RRR^{-1}}_{(\ref{HeckeCond})}
	\bigg)\mathbf{\Omega}\XX_2\DDD_2
	\\
&=
	\XX_2
	\DDD_2\RRR^{-1}\XX_1
	\DDD_1\RRR 
	+\XX_2\DDD_2
	-(\qqq-\qqq^{-1})
	\XX_2\RRR_{21}
	\underbrace{\mathbf{\Omega}\RRR^{-1}\DDD_1}_{\mathrm{commute}}
	\RRR
	+(\qqq-\qqq^{-1})^2\XX_2\DDD_2\\
&= 	
	\XX_2
	\DDD_2\RRR^{-1}\XX_1
	\DDD_1\RRR 
	+\XX_2\DDD_2
	-(\qqq-\qqq^{-1})
	\XX_2\DDD_2\mathbf{\Omega}
	\underbrace{\RRR}_{(\ref{HeckeCond})}
	+(\qqq-\qqq^{-1})^2\XX_2\DDD_2\\
&= 	
	\XX_2
	\DDD_2\RRR^{-1}\XX_1
	\DDD_1\RRR 
	+\XX_2\DDD_2
	-(\qqq-\qqq^{-1})
	\XX_2\DDD_2
	\underbrace{\mathbf{\Omega}}_{(\ref{HeckeCond})}
	\RRR_{21}^{-1}\\
&= 	\XX_2
	\DDD_2\RRR^{-1}\XX_1
	\DDD_1\RRR 
	+\XX_2\DDD_2\RRR^{-1}\RRR^{-1}_{21}
\end{align*}
For (\ref{D1YL2}), we start from the right:
\begin{align*}
\RRR\XX_2\RRR_{21}\YLu_1\RRR_{21}^{-1}
&= 	
	\underbrace{\RRR}_{(\ref{HeckeCond})}
	\XX_2
	+
	\underbrace{\RRR\XX_2\RRR_{21}\XX_1}_{(\ref{D1rel1})}
	\DDD_1\RRR_{21}^{-1}\\
&= 
	\RRR_{21}^{-1}\XX_2
	+
	\XX_1
	\underbrace{\RRR\XX_2\RRR_{21}\DDD_1\RRR_{21}^{-1}}_{(\ref{D1rel3})}
	+(\qqq-\qqq^{-1})\mathbf{\Omega}\XX_2
	\\
&= 
	\RRR_{21}^{-1}\XX_2
	+
	\XX_1
	\DDD_1\RRR_{21}^{-1}\XX_2\\
&= \YLu_1\RRR_{21}^{-1}\XX_2\qedhere
\end{align*}
\end{proof}

\begin{cor}
The homomorphism $\Psi^1:\DD_1\rightarrow\DD_0$ is injective.
\end{cor}

\begin{proof}
By Proposition \ref{DetOre}, $\det_\qqq(\mathbf{X})$ generates an Ore set in $\DD_1$ and thus there is an extension:
\[\textstyle\Psi_{\mathrm{loc}}^1:\DD_1[\det_\qqq(\XX)^{-1}]\rightarrow\DD_0\]
The image of $\Psi_{\mathrm{loc}}^{1}$ lies in $\DD_0^{\mathrm{IV}}[\det_\qqq(\mathbf{A})^{-1}]$.
By (\ref{D1YL1}) and (\ref{D1YL2}), the assignment 
\begin{equation*}
\begin{aligned}
\mathbf{A}&\mapsto \XX,&
\mathbf{B}^{-1}&\mapsto\YLu
\end{aligned}
\end{equation*}
induces a map 
\[\textstyle\Pi: \DD_0^{\mathrm{IV}}[\det_\qqq(\mathbf{A})^{-1}]\rightarrow \DD_1[\det_\qqq(\mathbf{\XX})^{-1}]\]
Observe that $\Psi_{\mathrm{loc}}^1$ is a section of $\Pi$ and is hence injective.
Since $\DD_1$ is an intgral domain, it includes into its localization.
\end{proof}

Equation (\ref{D1YL1}) implies that the entries of $\YLu$ generate a copy of $\iota_{\OO_G}(\mathcal{R}\hbox{\it ef}\,)$.
Hence, we can make sense of $\det_\qqq(\YLu)$.

\begin{cor}
The element $\det_\qqq(\YLu)$ generates an Ore set of $\DD_1$.
Let $\DD_1^\star$ denote the localization at $\det_\qqq(\YLu)$ and let $\YL:=\YLu^{-1}$.
The assigment
\[
\mathbf{M}\mapsto\YL\YR=\left( \mathbf{I}+\XX\DDD \right)^{-1}\left( \mathbf{I}+\DDD\XX \right)
\]
defines a moment map $\mu_{\DD_1}:\OO_G\rightarrow\DD_1^\star$.
\end{cor}

\begin{proof}
We have $\Psi^1(\YLu)=\mathbf{B}^{-1}$, which generates an Ore set of $\DD_0^{\mathrm{IV}}[\det_\qqq(\mathbf{A})^{-1}]$.
Since $\Psi^1$ is injective, we have that $\det_\qqq(\YLu)$ generates an Ore set in $\DD_1$.
We will abuse notation and denote the extension $\Psi^1:\DD_1^\star\rightarrow\DD_0$
Next, observe that
\[
\Psi^1\left( \YL\YR \right)=\mathbf{B}\mathbf{A}^{-1}\mathbf{B}^{-1}\mathbf{A}
\]
Again, by the injectivity of $\Psi^1$, we can conclude that $\mu_{\DD_1}:\OO_G\rightarrow\DD_1^\star$ is a ring homomorphism and a moment map.
\end{proof}

%Finally, let ${}^\star\DD_1$ denote the localization at $\mu_L(\det_\qqq(\mathbf{M})^{-1})$ and $\mu_R(\det_\qqq(\mathbf{M}))$, and let $\YL:=\YLu^{-1}$.
%\begin{prop}
%The assignment 
%\[
%\mu_{\DD_1}(\mathbf{M})=\YL\YR=(\mathbf{I}+\XX\DDD)^{-1}(\mathbf{I}+\DDD\XX)
%\]
%defines a moment map $\mu_{\DD_1}:\OO_G\rightarrow{}^\star\DD_1$.
%\end{prop}
%
%\begin{proof}
%The desired identities are
%\begin{align*}
%\YL_1\YR_1\RRR\XX_2&= \RRR_{21}^{-1}\XX_2\RRR_{21}\YL_1\YR_2\RRR\\
%\YL_1\YR_1\RRR\DDD_2&= \RRR_{21}^{-1}\DDD_2\RRR_{21}\YL_1\YR_2\RRR
%\end{align*}
%which are easily deduced from (\ref{D1YL2})-(\ref{D1YR3}).
%\end{proof}

\subsection{Quantum Hamiltonian reduction}
Fix a ground ring $R$ and recall the setup at the start of \ref{Moment}: we have a Hopf algebra $H$, a left coideal subalgebra $H'\subset H$, an $H$-module algebra $A$, and a moment map $\mu:H'\rightarrow A$.
If $H'$ is closed under the adjoint action of $H$, then we can define a character of $H'$ to be an $R$-algebra homomorphism $\chi:H'\rightarrow R$ that is invariant with respect to this adjoint action.
For such a character $\chi$, let 
\[I_\chi:=A\big(\mu(\ker \chi)\big)\]
We can perform \textit{quantum Hamiltonian reduction}:
\[
A\sslash_\chi H':= \left[A\big/I_\chi\right]^H
\]

\begin{prop}[\cite{GanJorSaf,VarVassRoot}]
The quantum Hamiltonian reduction $A\sslash_\chi H'$ is an $R$-algebra.
\end{prop}

In our setup of $H=\UU$ and $H'=\kappa(\OO_G)\cong\OO_G$, we do indeed have that $\OO_G$ is closed under the adjoint action of $\UU$.
Recall from \ref{Representations} that for $\mathsf{C}\in\CC(\qqq)^\times$, we have a character $\chi_\mathsf{C}$.
From (\ref{RFactor}), it is easy to see that the restriction of $\chi_\mathsf{C}$ to $\OO_G$ is given by
\[
\chi_\mathsf{C}(\mathbf{M})=\mathsf{C}^{-2}\mathbf{I}
\]
Later on, we will base change to $\LL_\qqq:=\CC(\qqq)[\ttt^{\pm 1}]$, and the same holds for $\mathsf{C}\in\LL_\qqq^\times$.

\subsubsection{Quantized multiplicative quiver varieties}
For two locally finite $\UU$-module algebras $A$ and $B$, we define their \textit{braided tensor product} $A\,\widetilde{\otimes}\,B$ to be the $\UU$-module $A\otimes B$ with the product
\[
(a_1\otimes b_1)(a_2\otimes b_2)=a_1(r_s\cdot a_2)\otimes ({}_sr\cdot b_1)b_2
\]
The extra braiding makes $A\,\widetilde{\otimes}\,B$ into a $\UU$-module algebra.
The following was proved in \cite{GanJorSaf}:
\begin{prop}\label{MomentProd}
Let $A$ and $B$ be two locally finite $\UU$-module algebras with moment maps $\mu_A:\OO_G\rightarrow A$ and $\mu_B:\OO_G\rightarrow B$.
The matrix product
\begin{align*}
\mu_{A\,\widetilde{\otimes}\, B}&:=(\mu_A\otimes\mu_B)\circ\Delta_{\OO_G}\\
\mu_{A\,\widetilde{\otimes}\,B}(m_{ij})&=\sum_{k}\mu_A(m_{ik})\mu_B(m_{kj})
\end{align*}
defines a moment map $\mu_{A\,\widetilde{\otimes}\,B}:\OO_G\rightarrow A\,\widetilde{\otimes}\, B$.
\end{prop}

We now set
\begin{equation*}
\begin{aligned}
\MM_\ell&:=\DD_\ell\,\widetilde{\otimes}\,\WW,&
\MM_\ell^\star&:=\DD_\ell^\star\,\widetilde{\otimes}\, \WW^\star\hbox{ for }\ell\ge 1\\
\MM_0&:=\DD_0\,\widetilde{\otimes}\, \WW,&
\MM_0^\star&:=\DD_0\,\widetilde{\otimes}\, \WW^\star
\end{aligned}
\end{equation*}
We will view $\WW$ as a $\UU^{\otimes\ell}$-module where only $\UU_{(1)}$ acts nontrivially.
For $\ell\ge 1$, we have by Proposition \ref{MomentProd} that the maps
\begin{align}
\label{MuEll}
\mu_{\MM_\ell}^{(a)}(\mathbf{M})&:=
\begin{cases}
\left(\mathbf{I}+\XX^{(a)}\DDD^{(a)} \right)^{-1}\left(\mathbf{I}+\DDD^{(a-1)}\XX^{(a-1)}\right) & a\not=1\\
\qqq^{-2}\left(\mathbf{I}+\XX^{(1)}\DDD^{(1)} \right)^{-1}\left(\mathbf{I}+\DDD^{(\ell)}\XX^{(\ell)}\right)\left(\mathbf{I}+(\qqq^2-1)\mathbf{d}\mathbf{x}\right) & a=1
\end{cases}
\end{align}
are moment maps for the $\UU_{(a)}$-action on $\MM_\ell^\star$.
These can be combined into a moment map for the $\UU^{\otimes\ell}$-action:
\[
\mu_{\MM_\ell}(\mathbf{M}_1\otimes\cdots\otimes\mathbf{M}_\ell):=\mu_{\MM_\ell}^{(1)}(\mathbf{M}_1)\cdots\mu_{\MM_\ell}^{(\ell)}(\mathbf{M}_\ell)
\]
For $\mathsf{C}=(\mathsf{C}_1,\ldots, \mathsf{C}_\ell)\in (\CC(\qqq)^\times)^\ell$, we have the following character of $\UU^{\otimes\ell}$:
\[
\chi_\mathsf{C}:=\chi_{\mathsf{C}_1}\otimes\cdots\otimes\chi_{\mathsf{C}_\ell}
\]
Finally, for $\ell=0$, the moment map for $\MM_0^\star$ is:
\[
\mu_{\MM_0}(\mathbf{M}):=\qqq^{-2}\mathbf{B}\mathbf{A}^{-1}\mathbf{B}^{-1}\mathbf{A}\left( \mathbf{I}+(\qqq^2-1)\mathbf{d}\mathbf{x} \right)
\]

Now, we will incorporate the parameters $\ZZZ=(\ZZZ_1,\ldots,\ZZZ_\ell)\in(\CC^\times)^{\ZZ/\ell\ZZ}$ from the cyclotomic DAHA.
Namely, after picking a choice of square root for each $\ZZZ_a$, we consider character parameters of the form
\begin{equation}
\mathsf{C}_a=\left(\frac{\ZZZ_{a-1}}{\ZZZ_a}\right)^{\frac{1}{2}}\mathsf{K}_a
\label{ZX}
\end{equation}
for some $\mathsf{K}_a\in\CC(\qqq)^\times$ or later $\CC(\qqq)[\ttt^{\pm1}]^\times$.
Let $\mathsf{K}=(\mathsf{K}_1,\ldots, \mathsf{K}_\ell)$ and denote by $\ZZZ\mathsf{K}$ the parameter vector $\mathsf{C}$ given by (\ref{ZX}).

\begin{defn}
For $\ell\ge 1$, $\ZZZ\in(\CC^\times)^\ell$, and $\mathsf{K}\in(\CC(\qqq)^\times)^\ell$, let us abbreviate $I_{\mathsf{\ZZZ\mathsf{K}}}:=I_{\chi_{\mathsf{\ZZZ\mathsf{K}}}}\subset\MM_\ell^\star$. 
The \textit{quantized multiplicative quiver variety} is the quantum Hamiltonian reduction
\[
\AAA_\ell(\ZZZ,\mathsf{K}):=\left[ \MM_\ell^\star\big/I_{\ZZZ\mathsf{K}} \right]^{\UU^{\otimes\ell}}
\]
Similarly, for $\ell=0$ and $\mathsf{K}\in\CC(\qqq)^\times$, we abbreviate $I_\mathsf{K}:=I_{\chi_{\mathsf{K}}}\subset\MM_0^\star$ and define
\[
\AAA_0(\mathsf{K}):=\left[ \MM_0^\star\big/I_{\mathsf{K}} \right]^\UU
\]
\end{defn}

\begin{prop}\label{DGen}
$\AAA_\ell(\ZZZ,\mathsf{K})$ is generated by the image of the natural map $\pi:(\DD_\ell^\star)^{\UU^{\otimes\ell}}\rightarrow\AAA_\ell(\ZZZ,\mathsf{K})$.
%Likewise, $\AAA_0(\mathsf{C})$ is generated by the image of $\pi:\DD_0^\UU\rightarrow\AAA_0(\mathsf{C})$.
\end{prop}

\begin{proof}
The proof is similar to that of Lemma 5.7 in \cite{QHarish}.
\end{proof}

\subsubsection{Classical degeneration and filtrations}\label{Class}
Recall that for $\ell\ge 1$, we have defined notations of standard monomials for $\DD_\ell$, and we also have such a notion for $\WW$.
Let us define a standard monomial for $\MM_\ell$ to be a product
\[
M_\DD M_\WW
\]
where $M_\DD$ is a standard monomial for $\DD_\ell$ and $M_\WW$ is a standard monomial for $\WW$.
We let $\MM_\ell^\ZZ\subset\MM_\ell$ to be the $\CC[\qqq^{\pm1}]$-subalgebra generated by
\[
\left\{ x_{ij},\del_{ij} \right\}_{i,j=1}^n\cup\left\{ x_i,\widetilde{\del_i} \right\}_{i=1}^n
\]

\begin{prop}[\cite{JordanMult}]
The standard monomials form a basis of $\MM_\ell^\ZZ$.
Thus, $\MM_\ell$ is a flat deformation of the coordinate ring of $T^*(\mathrm{Mat}_{n\times n}^{\oplus\ell}\oplus\CC^n)$.
\end{prop}

For any subspace $V\subset\MM_\ell$, we set
\begin{equation*}
\begin{aligned}
V_\ZZ&:= V\cap \MM_\ell^\ZZ,&
V\big|_{\qqq=1}&:=V_\ZZ\bigg/(\qqq-1)
\end{aligned}
\end{equation*}
Because monomials span $\MM_\ell$, we evidently have
\[
V=\CC(\qqq)\otimes V_\ZZ
\]
Moreover, for finite-dimensional $V$, because $\MM_\ell^\ZZ$ is a free $\CC[\qqq^{\pm 1}]$-module and $\CC[\qqq^{\pm1}]$ is a PID, we have
\begin{equation*}
\dim_{\CC(\qqq)} V= \mathrm{rank}_{\CC[\qqq^{\pm 1}]}V_\ZZ=\dim_\CC V\big|_{\qqq=1}
\end{equation*}
For a subquotient $W/V$ of $\MM_\ell$ with $W,V\subset\MM_\ell$, we define 
\[W/V\big|_{\qqq=1}:=W\big|_{\qqq=1}\bigg/V\big|_{\qqq=1}\]
Note that for finite dimensional $W$ and $V$,
\begin{equation}
\mathrm{dim}_{\CC(\qqq)}W/V= \dim_{\CC}W/V\big|_{\qqq=1}
\label{DimEq}
\end{equation}

Next, let us consider invariants.
We will use $\gl_n^{(a)}$ to denote the $a$th summand of $\gl_n^{\oplus \ell}$.
For 
\[(\mathbf{Z}^{(1)},\ldots, \mathbf{Z}^{(\ell)}, \mathbf{z}, \mathbf{W}^{(1)},\ldots,\mathbf{W}^{(\ell)},\mathbf{w})\in \left(\mathrm{Mat}_{n\times n}^{\oplus\ell}\oplus\CC^n\right)^{\oplus 2}\cong T^*\left(\mathrm{Mat}_{n\times n}^{\oplus\ell}\oplus\CC^n\right)\]
and $g\in\gl_n^{(a)}$, we set
\begin{equation}
g\cdot(\mathbf{Z}^{(1)},\ldots, \mathbf{Z}^{(\ell)}, \mathbf{z})=
\begin{cases}
(\mathbf{Z}^{(1)},\ldots,\mathbf{Z}^{(a-2)},-\mathbf{Z}^{(a-1)}g, g\mathbf{Z}^{(a)},\mathbf{Z}^{(a+1)},\ldots,\mathbf{Z}^{(\ell)},\mathbf{z}) & a\not=1\\
(g\mathbf{Z}^{(1)},\mathbf{Z}^{(2)}\ldots,\mathbf{Z}^{(\ell-1)},-\mathbf{Z}^{(\ell)}g, -g^T\mathbf{z}) & a=1
\end{cases}
\label{NakaAc}
\end{equation}
and take the induced action on $T^*(\mathrm{Mat}_{n\times n}^{\oplus \ell}\oplus \CC^n)$.
It is evident that $\MM_\ell^\ZZ$ is closed under the $\UU^{\otimes\ell}$-action.
If we identify
\begin{align*}
\XX^{(a)}\big|_{\qqq=1}&= \hbox{ coordinate to }\mathbf{Z}^{(a)}\\
\DDD^{(a)}\big|_{\qqq=1}&= \hbox{ coordinate to }\mathbf{W}^{(a)}\\
\mathbf{x}\big|_{\qqq=1}&=\hbox{ coordinate to }\mathbf{z}\\ 
\mathbf{d}\big|_{\qqq=1}&=\hbox{ coordinate to }\mathbf{w}
\end{align*}
then the actions of $E_i,F_i\in\UU_{(a)}$ become those of the matrix units $E_{i, i+1},E_{i+1,i}\in\gl_n^{(a)}$ on the coordinate ring upon taking $\qqq=1$.
Thus, for a $\UU^{\otimes\ell}$-submodule $V\subset\MM_\ell$, $V\big|_{\qqq=1}$ is a $\gl_n^{\oplus\ell}$-module and we have
\[
V^{\UU^{\otimes\ell}}\big|_{\qqq=1}\subset \left(V\big|_{\qqq=1}\right)^{\mathfrak{gl}_n^{\oplus\ell}}
\]
We will prove the opposite containment in \ref{QTrace} below.
For $\ell=0$, we can define similar notions using the integral form $\DD_0^\ZZ$.

Finally, we can define a filtration $F_k^\DD$ of $\MM_\ell$ wherein
\begin{equation*}
\begin{aligned}
\deg(x_{ij})=\deg(\del_{ij})&=1,& \deg(x_i)=\deg(\del_i)&= 0
\end{aligned}
\end{equation*}
%Each $F_k$ is a finite-dimensional span of standard monomials.
%Moreover, each $F_k$ is a $\UU$-submodule.
We endow all subspaces and quotients of $\MM_\ell$ with the induced filtration.
Let
\[I^+_{\mathsf{C}}:=I_{\mathsf{C}}\cap\MM_\ell\]
Observe that in $\mathrm{gr}_{F^\DD}\left(\MM_\ell\big/ I^+_{\mathsf{C}}\right)$, we have the relations
\begin{equation}
\ZZZ_a\XX^{(a)}\DDD^{(a)}=
\begin{cases}
\mathsf{K}_a^{2}\ZZZ_{a-1}\DDD^{(a-1)}\XX^{(a-1)} & a\not=1\\
\mathsf{K}_1^{2}\ZZZ_{\ell}\left(\DDD^{(\ell)}\XX^{(\ell)}+\DDD^{(\ell)}\XX^{(\ell)}\mathbf{\widetilde{d}}\mathbf{x}\right)& a=1
\end{cases}
\label{GrRel}
\end{equation}
where $\mathbf{\widetilde{d}}=(\qqq-\qqq^{-1})\mathbf{d}$.
We can also define a filtration $F_l^{\WW}$ of $\MM_\ell$ wherein
\begin{equation*}
\begin{aligned}
\deg(x_{ij})=\deg(\del_{ij})&=0,&
\deg(x_i)=\deg(\del_i)=1
\end{aligned}
\end{equation*}
Let us set $F_{k,l}:=F_k^\DD\cap F_l^\WW$.
Then each $F_{k,l}$ is a finite-dimensional $\UU$-submodule of $\MM_{\ell}$.

\subsubsection{Quantum trace elements}\label{QTrace}
Recall that if we view $\mathbb{V}_{(a)}$ as a copy of $\mathbb{V}$ placed at the $a$th round vertex in the quiver (\ref{CyclicQ}), then we view $\XX^{(a)}$ as the coordinates of a matrix $\mathbf{Z}^{(a)}:\mathbb{V}_{(a+1)}\rightarrow\mathbb{V}_{(a)}$. 
Likewise, we view $\DDD^{(a)}$ as the coordinates of a matrix $\mathbf{W}^{(a)}:\mathbb{V}_{(a)}\rightarrow\mathbb{V}_{(a+1)}$, $\mathbf{x}$ as the coordinates of a covector $\mathbf{z}:\mathbb{V}_{(1)}\rightarrow \CC(\qqq)$, and $\mathbf{d}$ as the coordinates of a vector $\mathbf{w}:\CC(\qqq)\rightarrow\mathbb{V}_{(1)}$.
We call a matrix product $\mathbf{\Pi}$ of $\{\XX^{(a)},\DDD^{(a)},\mathbf{x},\mathbf{d}\}_{a\in\ZZ/\ell\ZZ}$ a \textit{cyclic word} it describes a composition of maps that starts and ends at the same round vertex in (\ref{CyclicQ}).
For example, the following elements are cyclic words:
\begin{equation}
\begin{aligned}
\XX^\circ&:=\XX^{(1)}\cdots\XX^{(\ell)},&
\DDD^\circ&:=\DDD^{(\ell)}\cdots\DDD^{(1)}
\end{aligned}
\label{XDCircDef}
\end{equation}

For a cyclic word $\mathbf{\Pi}$, we define its \textit{quantum trace} to be
\[
\tr_\qqq(\mathbf{\Pi}):=\sum_{i=1}^n \qqq^{-n+2i}\mathbf{\Pi}_{ii}
\]

\begin{prop}
Quantum traces over cyclic words are $\UU^{\otimes\ell}$-invariant elements in $\MM_\ell^\ZZ$ that generate $\MM_\ell^{\UU^{\otimes\ell}}$.
\end{prop}

\begin{proof}
The proof of invariance is similar to that of Proposition 5.4 in \cite{QHarish}.
Clearly, $\tr_\qqq(\mathbf{\Pi})\in\MM_\ell^\ZZ$ and becomes the usual trace at $\qqq=1$.
To prove generation, it suffices to consider each finite-dimensional piece $F_{k,l}\MM_\ell$.
By a theorem of Le Bruyn--Procesi \cite{LBProc}, $F_{k,l}\MM_\ell\big|_{\qqq=1}$ is generated by classical traces of cyclic words.
We then get the desired result for $F_{k,l}\MM_\ell$ by applying Nakayama's lemma.
\end{proof}

\begin{cor}\label{DegInv}
For any $\UU$-submodule $V\subset\MM_\ell$,
\[
V^{\UU^{\otimes\ell}}\big|_{\qqq=1}=\left( V\big|_{\qqq=1} \right)^{\gl_n^{\oplus\ell}}
\]
\end{cor}

\begin{proof}
Proved similarly to Lemma 5.5 of \cite{QHarish}.
\end{proof}

\subsubsection{Nakajima quiver varieties}
For special values of the parameter $\mathsf{C}$, we can relate the $\qqq=1$ degeneration to the Nakajima quiver variety for (\ref{CyclicQ}) \cite{NakaKM}.
We begin by recalling its definition.
Let $(\mathbf{G}^{(1)},\ldots,\mathbf{G}^{(\ell)},\mathbf{j},\mathbf{H}^{(1)},\ldots,\mathbf{H}^{(\ell)},\mathbf{i})$ be the coordinates of another copy of $\mathbb{M}_\ell:=T^*\left( \mathrm{Mat}^{\oplus\ell}_{n\times n}\oplus\CC^n \right)$, equipped with the same $\gl_n^{\oplus\ell}$-action as (\ref{NakaAc}).
Let $I_\ell\subset\CC[\mathbb{M}_\ell]$ be the ideal of relations
\begin{equation}
\mathbf{G}^{(a)}\mathbf{H}^{(a)}=
\begin{cases}
\mathbf{H}^{(a-1)}\mathbf{G}^{(a-1)} & a\not=1\\
\mathbf{H}^{(\ell)}\mathbf{G}^{(\ell)} + \mathbf{i}\mathbf{j}^T& a=1
\end{cases}
\label{NakRel}
\end{equation}
The \textit{Nakajima quiver variety} is then
\[
\mathfrak{M}_\ell:=\mathrm{Spec}\left( \CC[\mathbb{M}_\ell]\big/I_\ell \right)^{\gl_n^{\oplus\ell}}
\]
We can also define a version without the square vertex: let $\mathbb{M}'_\ell:=T^*\left( \mathrm{Mat}_{n\times n}^{\oplus\ell} \right)$ and let $I'_\ell\subset\CC[\mathbb{M}_\ell']$ be the ideal of relations
\[
\mathbf{G}^{(a)}\mathbf{H}^{(a)}=\mathbf{H}^{(a-1)}\mathbf{G}^{(a-1)}\hbox{ for all }a
\]
Define
\[
\mathfrak{M}_\ell':=\mathrm{Spec}\left( \CC[\mathbb{M}']\big/ I_\ell' \right)^{\gl_n^{\oplus\ell}}
\]
There is an obvious map $p:\mathfrak{M}_\ell\rightarrow\mathfrak{M}_\ell'$ given by projecting away the $(\mathbf{j},\mathbf{i})$ coordinates.
Finally, recall the ring $\mathcal{R}_\ell$ from (\ref{RellDef}):
\begin{align*}
\mathcal{R}_\ell^{\Sigma_n}&=\CC[z_1,\ldots, z_n,w_1,\ldots, w_n]^{\Sigma_n\wr\ZZ/\ell\ZZ}
\end{align*}
\begin{prop}\label{GanProp}
We have
\[
\CC[\mathfrak{M}_\ell]\cong\CC[\mathfrak{M}_\ell']\cong\mathcal{R}_\ell^{\Sigma_n}
\]
The first isomorphism is given by $p^*$ and the second is given by restricting to the locus
\begin{equation*}
\begin{aligned}
\mathbf{G}^{(a)}&=\mathbf{S},
&\mathbf{H}^{(a)}&= \mathbf{T}\hbox{ {\rm for all} }a
\end{aligned}
\end{equation*}
where $\mathbf{T}$ and $\mathbf{S}$ are diagonal matrices with coordinates $(z_1,\ldots, z_n)$ and $(w_1,\ldots, w_n)$, respectively. 
\end{prop}

\begin{proof}
The isomorphism $\CC[\mathfrak{M}_\ell]\cong\CC[\mathfrak{M}'_\ell]$ is Lemma 2.3 of \cite{CBDec}. 
The second isomorphism $\CC[\mathfrak{M}'_\ell]\cong\mathcal{R}_\ell^{\Sigma_n}$ is the main result of \cite{GanChev}.
\end{proof}

$\mathcal{R}_\ell^{\Sigma_n}$ inherits a grading where
\[
\deg(z_i)=\deg(w_i)=1
\]
For any graded vector space $V$, we denote by $\dim^k V$ the dimension of its $k$th graded piece.
Finally, recall the natural map $\pi:(\DD_\ell^\star)^{\UU^{\otimes\ell}}\rightarrow\AAA_\ell(\ZZZ,\mathsf{K})$ from Proposition \ref{DGen}.

\begin{lem}\label{QTraceLem}
Let $\mathsf{K}$ be such that $\mathsf{K}_a\big|_{\qqq=1}=1$ for all $a$.
\begin{enumerate}
\item We have an isomorphism of rings:
\[
\mathrm{gr}_{F^\DD}\left( \left(\MM_\ell\big/I_{\ZZZ\mathsf{K}}^+\right)^{\UU^{\otimes\ell}}\bigg|_{\qqq=1}\right)\cong\CC[\mathfrak{M}_\ell]
\]
\item For any $k\in\ZZ_{\ge 0}$, we have 
\begin{equation*}
\dim^k_{\CC(\qqq)} \mathrm{gr}_{F^\DD}\left( \left(\MM_\ell\big/I_{\ZZZ\mathsf{K}}^+\right)^{\UU^{\otimes\ell}}\right)\le \dim^k_{\CC}\mathcal{R}_\ell^{\Sigma_n}
%\label{UpBound}
\end{equation*}
\item $\AAA_\ell(\ZZZ,\mathsf{K})$ is generated by the quantum traces
\begin{equation}
\tr_\qqq\left( (\XX^\circ)^{k_1}(\YLu^{(1)})^{k_2}(\DDD^\circ)^{k_3} \right)
\label{TrXYD}
\end{equation}
where $\XX^\circ$ and $\DDD^\circ$ are defined in (\ref{XDCircDef}).
\end{enumerate}
\end{lem}

\begin{proof}
Comparing (\ref{GrRel}) and (\ref{NakRel}), we have that
\begin{equation}
\begin{aligned}
\mathbf{G}^{(a)}&\mapsto\XX^{(a)},&
\mathbf{H}^{(a)}&\mapsto\ZZZ_a\DDD^{(a)},&
\mathbf{j}&\mapsto \mathbf{x},&
\mathbf{i}&\mapsto\DDD^{(\ell)}\XX^{(\ell)}\tilde{\mathbf{d}}
\end{aligned}
\label{NakQuant}
\end{equation}
defines a map
\begin{equation*}
f:\CC[\mathfrak{M}_\ell]\rightarrow \mathrm{gr}_{F^\DD}\left(\MM_\ell^{\UU^{\otimes\ell}}\big/(I_{\mathsf{C}}^+)^{\UU^{\otimes\ell}}\big|_{\qqq=1}\right)
\end{equation*}
(here, we are implicitly using Corollary \ref{DegInv}).
Its image contains $\mathrm{gr}_{F^\DD}\left( \pi\left( \DD_\ell^{\UU^{\otimes\ell}} \right)\big|_{\qqq=1} \right)$, and thus by Proposition \ref{DGen}, it is surjective.
Part (2) follows from
%applying (\ref{DimEq}) and 
Proposition \ref{GanProp}.

For injectivity in part (1), we note that Proposition \ref{GanProp} implies that $\mathfrak{M}_\ell$ is reduced and irreducible.
Thus, its coordinate ring is an integral domain and includes into its localization at $\{\det(\mathbf{G}^{(\ell)}),\det(\mathbf{H}^{(\ell)})\}$.
We can then define a map
\[
g:\mathrm{gr}_{F^\DD}\left(\MM_\ell^{\UU^{\otimes\ell}}\big/(I_{\mathsf{C}}^+)^{\UU^{\otimes\ell}}\big|_{\qqq=1}\right)
\rightarrow \CC[\mathfrak{M}_\ell][\det(\mathbf{G}^{(\ell)})^{-1},\det(\mathbf{H}^{(\ell)})^{-1}]
\]
by
\begin{equation*}
\begin{aligned}
\XX^{(a)}&\mapsto\mathbf{G}^{(a)},&
\DDD^{(a)}&\mapsto \ZZZ_a^{-1}\mathbf{H}^{(a)},&
\mathbf{x}&\mapsto \mathbf{j},&
\widetilde{\mathbf{d}}&\mapsto \ZZZ_a\left(\mathbf{H}^{(\ell)}\mathbf{G}^{(\ell)}\right)^{-1}\mathbf{i}
\end{aligned}
\end{equation*}
Clearly, $f$ is a section of $g$ and thus $f$ is injective.

For part (3), recall the generator $p_{a,k,b}$ (\ref{CycPow}) of $\mathcal{R}_\ell^{\Sigma_n}$.
It is easy to see that composition of the isomorphisms from Proposition \ref{GanProp} together with (\ref{NakQuant}) sends $p_{k_1,k_2,k_3}$ to the associated graded of the quantum trace (\ref{TrXYD}) at $\qqq=1$.
By applying Nakayama's lemma to each $F_{k}^{\DD}\pi\left( \DD_\ell^{\UU^{\otimes\ell}} \right)$, which is finite-dimensional, we conclude that the given quantum traces generate $\pi\left( \DD_\ell^{\UU^{\otimes\ell}} \right)$.
Finally, we apply Proposition \ref{DGen} to carry the generation statement over to $\AAA_\ell(\ZZZ,\mathsf{K})$.
\end{proof}

Now consider $\ell=0$.
Let
\begin{align*}
\MM_0^{\mathrm{IV}}&:=\DD_0^{\mathrm{IV}}\,\widetilde{\otimes}\, \WW\\
I_{\mathsf{K}}^{\mathrm{IV}}&:=I_{\mathsf{K}}\cap\MM_0^{\mathrm{IV}}
\end{align*}
The following is proved similarly to Lemma \ref{QTraceLem}(1):

\begin{prop}\label{QTraceLem0}
For $\mathsf{K}$ with $\mathsf{K}\big|_{\qqq=1}=1$, the map
\begin{equation*}
\begin{aligned}
\mathbf{G}^{(1)}&\mapsto\mathbf{A},&
\mathbf{H}^{(1)}&\mapsto \mathbf{B}^{-1},&
\mathbf{j}&\mapsto \mathbf{x},&
\mathbf{i}&\mapsto \mathbf{B}^{-1}\mathbf{A}\tilde{\mathbf{d}}
\end{aligned}
\end{equation*}
gives an isomorphism
\[
\left( \MM_0^{\mathrm{IV}}\big/I_{\mathsf{K}}^{\mathrm{IV}} \right)^\UU\bigg|_{\qqq=1}\cong\CC[\mathfrak{M}_1]
\]
If we endow $\MM_0^{\mathrm{IV}}$ with the grading where
\begin{equation*}
\begin{aligned}
\deg(a_{ij})&= \deg(b_{ij})=1,&
\deg(x_i)&= \deg(\del_i)=0
\end{aligned}
\end{equation*}
then composing this isomorphism with that of Proposition \ref{GanProp} gives a graded isomorphism with $\mathcal{R}_1^{\Sigma_n}$.
\end{prop}

\subsection{From $\ell$ to 0}
Recall that $I_{\ZZZ\mathsf{K}}^+:=I_{\ZZZ\mathsf{K}}\cap\MM_\ell$.
We begin by proving some identities in $\MM_\ell/I_{\ZZZ\mathsf{K}}^+$ for $\ell>1$.
The ideas presented already appeared in the proof of Lemma 5.2 of \cite{BEF}.
In the quantum setting, we have to reckon with the fact that $I_{\ZZZ\mathsf{K}}$ is only a \textit{left} ideal.
These identities will allow us to map $\AAA_\ell(\ZZZ,\mathsf{K})$ into a quotient of $\AAA_0(\mathsf{K}_1)$ for special values of $\mathsf{K}$.
For $\ell=1$, we already have the homomorphism $\Psi^1:\DD_1\rightarrow\DD_0$.

\subsubsection{Small corn-on-the-cob}
We start with the following:
\begin{lem}\label{SmallCob}
For $\ell>1$ and $2\le a\le\ell$, we have the following identity in $\MM_\ell/I_{\ZZZ\mathsf{K}}^+$:
\begin{equation}
\YLu^{(1)}\XX^{(1)}\cdots\XX^{(a-1)}\XX^{(a)}=\qqq^{-2n(a-1)}\left( \ZZZ_a\big/\ZZZ_1 \right)\mathsf{K}_2^{-2}\cdots\mathsf{K}_{a}^{-2}\XX^{(1)}\cdots\XX^{(a)}\YR^{(a)}
\label{Cob0}
\end{equation}
\end{lem}

\begin{proof}
First, we have
\begin{align}
\nonumber
\YLu^{(1)}\XX^{(1)}\cdots\XX^{(a)}&= \left( \mathbf{I}+\XX^{(1)}\DDD^{(1)} \right)\XX^{(1)}\cdots\XX^{(a)}\\
\label{Cob1}
&= \XX^{(1)}
\underbrace{\YR^{(1)}\XX^{(2)}\cdots\XX^{(a)}}
\end{align}
Let us look closer into the underbraced part of (\ref{Cob1}) without the matrix product between $\YR^{(1)}$ and $\XX^{(2)}$:
\begin{align*}
\underbrace{\YR^{(1)}_1\RRR\XX^{(2)}_2}_{(\ref{YR3})}
\cdots\XX^{(a)}_2
&= \RRR\XX_2^{(2)}
\underbrace{\YR^{(1)}_1\XX^{(3)}_2\cdots\XX^{(a)}_2}_{(\ref{YR4})}\\
&= \RRR\XX_2^{(2)}\cdots\XX_2^{(a)}\underbrace{\YR^{(1)}_1}_{(\ref{MuEll})}\\
&= \left(\ZZZ_2\big/\ZZZ_1\right)\mathsf{K}_2^{-2}\RRR\XX_2^{(2)}\underbrace{\cdots\XX_2^{(a)}\YLu^{(2)}_1}_{(\ref{YL4})}\\
&= \left(\ZZZ_2\big/\ZZZ_1\right)\mathsf{K}_2^{-2}\underbrace{\RRR\XX_2^{(2)}\YLu^{(2)}_1}_{(\ref{YL1})}
	\XX^{(3)}_2\cdots\XX_2^{(a)}\\
&= \left(\ZZZ_2\big/\ZZZ_1\right)\mathsf{K}_2^{-2}\YLu^{(2)}_1\RRR_{21}^{-1}\XX_2^{(2)}\cdots\XX_2^{(a)}
\end{align*}
Incorporating the matrix product, we have
\begin{equation}
\left(\YR^{(1)}_1\right)_{*j}\RRR\left( \XX_2^{(2)} \right)_{j*}\XX_2^{(3)}\cdots\XX_2^{(a)}
=
\left(\ZZZ_2\big/\ZZZ_1\right)\mathsf{K}_2^{-2}\left(\YLu_1^{(2)}\right)_{*j}\RRR_{21}^{-1}\left( \XX_2^{(2)} \right)_{j*}\XX_2^{(3)}\cdots\XX_2^{(a)}
\label{Cob2}
\end{equation}
Here, as in the proof of Lemma \ref{DellDet}, the subscripts $*j$ and $j*$ are the emphasize the matching indices for the matrix product.

Equation (\ref{Cob2}) can be depicted graphically as follows.
Let us view $\YR^{(1)}$, $\XX^{(b)}$, and $\YLu^{(2)}$ as maps $\mathbb{V}^*\otimes\mathbb{V}\rightarrow\DD_\ell$, and denote by $m_{\DD_\ell}$ the product in $\DD_\ell$.
We then have:
\begin{equation}
\includegraphics{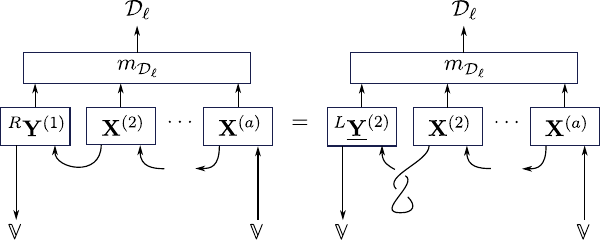}
\label{CornCob1}
\end{equation}
Using (\ref{RibbonAct}) to compute the effect of the ribbon twists, (\ref{Cob1}) becomes:
\[
\XX^{(1)}\YR^{(1)}\XX^{(2)}\cdots\XX^{(a)}=\qqq^{-2n}\left(\ZZZ_2\big/\ZZZ_1\right)\mathsf{K}_2^{-2}\XX^{(1)}\YLu^{(2)}\XX^{(2)}\cdots\XX^{(a)}
\]
We then continue as before but now starting with superscript index 2.
\end{proof}

\subsubsection{Small corn-on-the-cob and determinants}
We emphasize that the identity (\ref{Cob0}) is a consequence of the \textit{left} ideal $I_{\ZZZ\mathsf{K}}^+$.
Recall by Lemma \ref{DellDet} that $\det_\qqq(\XX^\circ)$ generates an Ore set in $\DD_\ell$.
It also generates an Ore set in $\MM_\ell$ because its commutation with $\WW$ is via the $R$-matrix.
Let $\DD_\ell^\circ$ and $\MM^\circ_\ell/I_{\ZZZ\mathsf{K}}^+$ denote the respective localizations at $\det_\qqq(\XX^\circ)$.
By Corollary \ref{XDetCor}, we can make sense of 
\[\XXu^{(a)}:=\left(\XX^{(a)}\right)^{-1}\]
in $\DD_\ell^\circ$, and thus by \textit{left} multiplication, the identity (\ref{Cob0}) yields
\begin{equation}
\XXu^{(a)}\cdots\XXu^{(1)}\YLu^{(1)}\XX^{(1)}\cdots\XX^{(a)}=\qqq^{-2n(a-1)}\left( \ZZZ_a\big/\ZZZ_1 \right)\mathsf{K}_2^{-2}\cdots\mathsf{K}_{a}^{-2}\YR^{(a)}
\label{Cob3}
\end{equation}
Again, because $I_{\ZZZ\mathsf{K}}^+$ is a left ideal, (\ref{Cob3}) does not \textit{a priori} allow us to identify corresponding products of entries of the two matrices.
Let $\YLu^{(1,a)}$ denote the matrix on the left hand side of (\ref{Cob3}).

\begin{lem}
In $\DD_\ell^\circ$, we have the identities:
\begin{align}
\label{Cob4}
\RRR_{21}\YLu^{(1,a)}_1\RRR\YLu^{(1,a)}_2&= \YLu^{(1,a)}_2\RRR_{21}\YLu^{(1,a)}_1\RRR\\
\label{Cob5}
\RRR_{21}\YR^{(a)}_1\RRR\YLu^{(1,a)}_1&= \YLu^{(1,a)}_2\RRR_{21}\YR^{(a)}_1\RRR
\end{align}
\end{lem}

\begin{proof}
These are straightforward applications of (\ref{Xaa})-(\ref{XXcom}), (\ref{YL3})-(\ref{YL4}), and (\ref{YR3})-(\ref{YR4}).
For (\ref{Cob4}), we have
\begin{align*}
&\quad\RRR_{21}\YLu^{(1,a)}_1\RRR\YLu^{(1,a)}_2\\
&= 
\RRR_{21}\XXu^{(a)}_1\cdots\XXu^{(1)}_1\YLu^{(1)}_1\XX^{(1)}_1\cdots
\underbrace{\XX^{(a)}_1\RRR\XXu^{(a)}_2}_{(\ref{Xaa})}
\cdots\XXu^{(1)}_2\YLu^{(1)}_2\XX^{(1)}_2\cdots\XX^{(a)}_2\\
&= 
\RRR_{21}\XXu^{(a)}_1\cdots\XXu^{(1)}_1\YLu^{(1)}_1\XX^{(1)}_1\cdots
\underbrace{\XX^{(a-1)}_1\XXu^{(a)}_2}_{(\ref{XXad})}
\underbrace{\RRR_{21}\XX^{(a)}_1\XXu^{(a-1)}_2}_{(\ref{XXad})}
\cdots\XXu^{(1)}_2\YLu^{(1)}_2\XX^{(1)}_2\cdots\XX^{(a)}_2\\
&= 
\RRR_{21}\XXu^{(a)}_1
\underbrace{\cdots\XXu^{(1)}_1\YLu^{(1)}_1\XX^{(1)}_1\cdots \XX^{(a-2)}_1\XXu^{(a)}_2}_{(\ref{XXcom})+(\ref{YL4})}
\XX^{(a-1)}_1\RRR
\XXu^{(a-1)}_2
\underbrace{\XX^{(a)}_1\XXu^{(a-2)}_2\cdots\XXu^{(1)}_2\YLu^{(1)}_2\XX^{(1)}_2\cdots}_{(\ref{XXcom})+(\ref{YL4})}
\XX^{(a)}_2\\
&= 
\RRR_{21}\XXu^{(a)}_1
\underbrace{\XXu^{(a-1)}_1\XXu^{(a)}_2}_{(\ref{XXad})}
\XXu^{(a-2)}\cdots\XXu^{(1)}_1\YLu^{(1)}_1\XX^{(1)}_1\cdots \XX^{(a-1)}_1\\
&\quad\times\RRR
\XXu^{(a-1)}_2\cdots\XXu^{(1)}_2\YLu^{(1)}_2\XX^{(1)}_2\cdots\XX^{(a-2)}_2
\underbrace{\XX^{(a)}_1\XX^{(a-1)}_2}_{(\ref{XXad})}
\XX^{(a)}_2\\
&=\underbrace{\RRR_{21}\XXu^{(a)}_1\XXu^{(a)}_2\RRR^{-1}}_{(\ref{Xaa})}
\XXu^{(a-1)}_1\cdots\XXu^{(1)}_1\YLu^{(1)}_1\XX^{(1)}_1\cdots \XX^{(a-1)}_1\\
&\quad\times\RRR
\XXu^{(a-1)}_2\cdots\XXu^{(1)}_2\YLu^{(1)}_2\XX^{(1)}_2\cdots\XX^{(a-1)}_2
\underbrace{\RRR_{21}\XX^{(a)}_1\XX^{(a)}_2}_{(\ref{Xaa})}\\
&=\XXu^{(a)}_2\XXu^{(a)}_1\cdots\XXu^{(1)}_1\YLu^{(1)}_1\XX^{(1)}_1\cdots \XX^{(a-1)}_1
\RRR
\XXu^{(a-1)}_2\cdots\XXu^{(1)}_2\YLu^{(1)}_2\XX^{(1)}_2\cdots\XX^{(a)}_2\XX^{(a)}_1\RRR
\end{align*}
The commutation of the remaining matrices works \textit{mutatis mutandis}.
Equation (\ref{Cob5}) is much easier to prove, and we leave it as an exercise.
\end{proof}

\begin{cor}\label{CobDetCor1}
In $\MM^\circ_\ell/I_{\ZZZ\mathsf{K}}^+$, for any $k\ge 0$, we have
\[
\YLu_1^{(1,a)}\cdots \YLu_k^{(1,a)}=\qqq^{-2kn(a-1)}\left( \ZZZ_a\big/\ZZZ_1 \right)^{k}\left( \mathsf{K}_2^{-2}\cdots\mathsf{K}_a^{-2} \right)^k\YR^{(a)}_1\cdots\YR^{(a)}_k
\]
\end{cor}

\begin{proof}
This is a consequence of (\ref{Cob4}) and (\ref{Cob5}) being of the same form.
Graphically depicting (\ref{Cob5}), we can move $\YR^{(a)}$ to the right, apply relation (\ref{Cob3}) $\YLu^{(1)}$, and then by (\ref{Cob4}) apply the reverse graphical move to place $\YLu^{(1)}$ in the original location.
For $k=2$, this looks like:
\[
\includegraphics{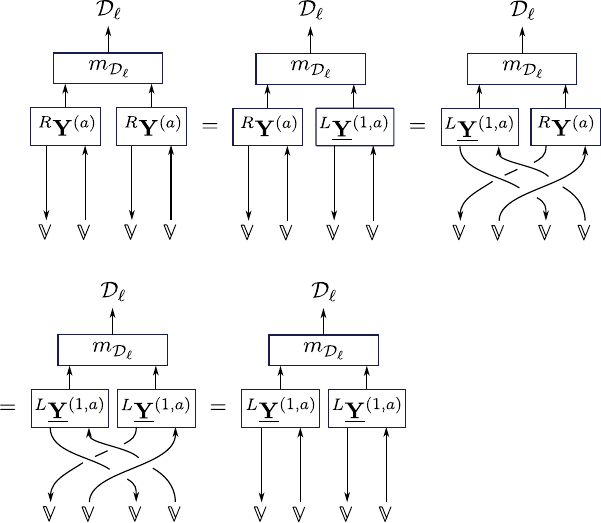}
\]
Here, we have omitted the extra scalars from (\ref{Cob1}).
\end{proof}

\begin{cor}\label{CobDetCor2}
For $\ell>1$ and any $a$, $\det_\qqq(\YR^{(a)})^{-1}$ and $\det_\qqq(\YLu^{(a)})^{-1}$ lie in the localization
\[
\textstyle\MM_\ell^\circ\big/I_{\ZZZ\mathsf{K}}^+\left[ \det_\qqq(\YLu^{(1)})^{-1} \right]=\MM_\ell/I_{\ZZZ\mathsf{K}}^+\left[ \det_\qqq(\XX^\circ)^{-1},\det_\qqq(\YLu^{(1)})^{-1} \right]
\]
\end{cor}

\begin{proof}
Because 
\[\Delta_{\OO_G}(\mathrm{qcoev}_{\mathbb{1}_\qqq}(1))=\mathrm{qcoev}_{\mathbb{1}_\qqq}(1)\otimes\mathrm{qcoev}_{\mathbb{1}_\qqq}(1)\]
it follows that in $\MM_\ell\big/I_{\ZZZ\mathsf{K}}^+$, we have
\[
\textstyle \det_\qqq(\YLu^{(a)})=\det_\qqq(\YR^{(a-1)})
\]
Thus, it suffices to only consider $\det_\qqq(\YR^{(a)})^{-1}$.
By (\ref{Cob4}), the entries of $\YLu^{(1,a)}$ generate a copy of $\mathcal{R}\hbox{\it ef}\,$, and thus it makes sense of $\det_\qqq(\YLu^{(1,a)})$.
Moreover, by Corollary \ref{CobDetCor1}, we have
\[
\textstyle\det_\qqq(\YLu^{(1,a)})=C\det_\qqq(\YR^{(a)})
\]
in $\MM_\ell^\circ/I_{\ZZZ\mathsf{K}}^+$ for some nonzero scalar $C\in\CC(\qqq)$.
The result follows from the observation that we have $\det_\qqq(\YLu^{(1,a)})^{-1}$ upon localizing at $\det_\qqq(\YLu^{(1)})$.
\end{proof}

\subsubsection{Big corn-on-the-cob}
Next, we consider a quantum analogue of Lemma 5.2 of \cite{BEF}.

\begin{lem}\label{BigCobLem1}
For $\ell>1$, the following equality holds in $\MM_\ell\big/I_\mathsf{C}^+$:
\begin{equation*}
\XX^{\circ}\DDD^{\circ}=\left(\YLu^{(1)}-\mathbf{I} \right)\left( \qqq^{2n}\left( \ZZZ_1\big/\ZZZ_2 \right)\mathsf{K}_2^{2}\YLu^{(1)}-\mathbf{I} \right)\cdots\left( \qqq^{2n(\ell-1)}\left( \ZZZ_1\big/\ZZZ_\ell \right)\mathsf{K}_2^{2}\cdots\mathsf{K}_\ell^{2}\YLu^{(1)}-\mathbf{I} \right)
\label{BCob1}
\end{equation*}
\end{lem}

\begin{proof}
%We will use induction on $a$; the base case of $a=1$ is trivial.
First, we have
\[
\XX^{\circ}\DDD^{\circ}=\XX^{(1)}\cdots\XX^{(\ell-1)}\left( \YLu^{(\ell)}-\mathbf{I} \right)\DDD^{(\ell-1)}\cdots\DDD^{(1)}
\]
Let us consider the $\YLu^{(\ell)}$ summand on its own.
Moreover, we will ignore the $\XX$-matrices and the matrix product between $\YLu^{(\ell)}$ and $\DDD^{(\ell-1)}$.
We make moves similar to the proof of Lemma \ref{SmallCob}:
\begin{align*}
\underbrace{\YLu^{(\ell)}_1\RRR_{21}^{-1}\DDD^{(\ell-1)}_2}_{(\ref{YL5})}
\cdots\DDD^{(1)}
&=
 \RRR_{21}^{-1}\DDD^{(\ell-1)}_2
	\underbrace{\YLu^{(\ell)}_1\DDD^{(\ell-2)}\cdots\DDD^{(1)}}_{(\ref{YL6})}\\
&=
 \RRR_{21}^{-1}\DDD^{(\ell-1)}_2\cdots\DDD^{(1)}
	\underbrace{\YLu^{(\ell)}_1}_{(\ref{MuEll})}\\
&=
 \left( \ZZZ_{\ell-1}\big/\ZZZ_\ell \right)\mathsf{K}_\ell^{2}\RRR_{21}^{-1}\DDD^{(\ell-1)}_2
	\underbrace{\cdots\DDD^{(1)}\YR^{(\ell-1)}_1}_{(\ref{YR6})}\\
&=
 \left( \ZZZ_{\ell-1}\big/\ZZZ_\ell \right)\mathsf{K}_\ell^{2}\RRR_{21}^{-1}\DDD^{(\ell-1)}_2\YR^{(\ell-1)}_1\DDD^{(\ell-1)}\cdots\DDD^{(1)}\\
&= \left( \ZZZ_{\ell-1}\big/\ZZZ_\ell \right)\mathsf{K}_\ell^{2}\YR^{(\ell-1)}_1\RRR\DDD^{(\ell-1)}_2\DDD^{(\ell-1)}\cdots\DDD^{(1)}
\end{align*}
As in the proof of Lemma \ref{SmallCob}, when we incorporate the matrix product between $\YLu^{(\ell)}$ and $\DDD^{(\ell-1)}$, we obtain an identity that is diagrammatically depicted like (\ref{CornCob1}) except that the crossings on the right-hand-side are reversed.
Altogether, we get:
\begin{align*}
&\quad\XX^{(1)}\cdots\XX^{(\ell-1)}\YLu^{(\ell)}\DDD^{(\ell-1)}\cdots\DDD^{(1)}\\
&= \qqq^{2n}\left( \ZZZ_{\ell-1}\big/\ZZZ_\ell \right)\mathsf{K}^{2}_\ell\XX^{(1)}\cdots\XX^{(\ell-1)}\YR^{(\ell-1)}\DDD^{(\ell-1)}\cdots\DDD^{(1)}\\
&= \qqq^{2n}\left( \ZZZ_{\ell-1}\big/\ZZZ_\ell \right)\mathsf{K}^{2}_\ell\XX^{(1)}\cdots\XX^{(\ell-1)}\DDD^{(\ell-1)}\YLu^{(\ell-1)}\DDD^{(\ell-2)}\cdots\DDD^{(1)}
\end{align*}
Proceeding in this manner, we obtain
\begin{align*}
&\quad\XX^{(1)}\cdots\XX^{(\ell)}\DDD^{(\ell)}\cdots\DDD^{(1)}\\
&=\XX^{(1)}\cdots\XX^{(\ell-1)}\left( \YLu^{(\ell)}-\mathbf{I} \right)\DDD^{(\ell-1)}\cdots\DDD^{(1)}\\
&= \XX^{(1)}\cdots\XX^{(\ell-1)}\DDD^{(\ell-1)}\cdots\DDD^{(1)}\left( \qqq^{2n(\ell-1)}\left( \ZZZ_1\big/\ZZZ_\ell \right)\mathsf{K}_2^{2}\cdots\mathsf{K}_\ell^{2}\YLu^{(1)}-\mathbf{I} \right)\\
&= \XX^{(1)}\cdots\XX^{(\ell-1)}\DDD^{(\ell-1)}\cdots\DDD^{(1)}\left( \qqq^{2n(\ell-1)}\left( \ZZZ_1\big/\ZZZ_\ell \right)\mathsf{K}_2^{2}\cdots\mathsf{K}_\ell^{2}\YLu^{(1)}-\mathbf{I} \right) 
\end{align*}

To proceed as above starting with $\XX^{(\ell-1)}\DDD^{(\ell-1)}$, we need to commute $\YLu^{(a)}$ for $a\not=1$ over past the $\YLu^{(1)}$ on the right, change it into $\YR^{(a-1)}$, and then commute it back.
%As in the proof of Corollary \ref{CobDetCor1}, it is helpful to establish some commutation relations.
The following is easy to prove using (\ref{YL1})-(\ref{YL2}) and (\ref{YL3})-(\ref{YL6}):
\begin{align*}
\YLu^{(a)}_1\YLu^{(1)}_2&= \YLu^{(1)}_2\YLu^{(a)}_1\hbox{ for }a\not=1\\
\YR^{(a)}_1\YLu^{(1)}_2&= \YLu^{(1)}_2\YR^{(a)}_1\hbox{ for all }a
\end{align*}
Thus, there is no issue.
Continuing this process of commuting right, exchanging, and then commuting back, we obtain the desired identity.
\end{proof}

\subsubsection{Big corn-on-the-cob of arbitrary length}
Let
\[
\left( \mathsf{K}^{2}\ZZZ\YLu^{(1)}-\mathbf{I} \right):=
\left( \YLu^{(1)}-\mathbf{I} \right)\left( \qqq^{2n}\left( \ZZZ_1\big/\ZZZ_2 \right)\mathsf{K}_2^{2}\YLu^{(1)}-\mathbf{I} \right)\cdots\left( \qqq^{2n(\ell-1)}\left( \ZZZ_1\big/\ZZZ_\ell \right)\mathsf{K}_2^{2}\cdots\mathsf{K}_\ell^{2}\YLu^{(1)}-\mathbf{I} \right)
\]
i.e. the right-hand-side of (\ref{BCob1}).
Recall that we denote by $\MM_\ell^\circ:=\MM_\ell[\det_\qqq(\XX^\circ)^{-1}]$.
Lemma \ref{BigCobLem1} implies that in $\MM_\ell^\circ\big/I_\mathsf{C}^+$, we have the identity
\[
\DDD^{\circ}=\left( \XX^\circ \right)^{-1}\left( \mathsf{K}^{2}\ZZZ\YLu^{(1)}-\mathbf{I} \right)
\]
We will need to derive a similar identity for products of the entries of $\DDD^\circ$.

\begin{lem}\label{BCobLem2}
For $\ell>1$ and $k\ge 1$, we have the following identity in $\MM_\ell^\circ\big/I_\mathsf{C}^+$:
\[
\DDD_1^\circ\cdots \DDD_k^\circ=\left( \XX_1^\circ \right)^{-1}\left( \mathsf{K}^{2}\ZZZ\YLu_1^{(1)}-\mathbf{I}_1 \right)\cdots
\left( \XX_k^\circ \right)^{-1}\left( \mathsf{K}^{2}\ZZZ\YLu_k^{(1)}-\mathbf{I}_k \right)
\]
\end{lem}

\begin{proof}
We prove this by induction; the case $k=1$ is Lemma \ref{BigCobLem1}.
Consider
\begin{align*}
\XX_1^\circ\DDD_1^\circ\cdots\DDD_k^\circ
&=\XX_1^{(1)}\cdots\XX^{(\ell-1)}_1\left( \YLu^{(\ell)}_1-\mathbf{I}_1 \right)\DDD^{(\ell-1)}_1\cdots\DDD^{(1)}_1\DDD_2^\circ\cdots\DDD_k^\circ
\end{align*}
We can try to make the same moves as in the proof of Lemma \ref{BigCobLem1}: for $a\not =1$, we commute $\YLu^{(a)}$ all the way to the right, exchange it for $\YR^{(a-1)}$, and then commute back.
However, there is a novelty once we encounter $\DDD^{(a)}_j\DDD^{(a-1)}_j$ for $j> 1$.
By (\ref{YL2}), (\ref{YR2}), (\ref{YL3}), and (\ref{YR3}), we have:
\begin{align*}
\YLu^{(a)}_1\DDD_2^{(a)}\DDD_2^{(a-1)}&= \DDD_2^{(a)}\RRR^{-1}\RRR^{-1}_{21}\DDD_2^{(a-1)}\YLu^{(a)}_1\\
\DDD_2^{(a)}\RRR^{-1}\RRR^{-1}_{21}\DDD_2^{(a-1)}\YR^{(a)}_1&= \YR^{(a)}_1\DDD_2^{(a)}\DDD_2^{(a-1)}
\end{align*}
Thus, when performing our back and forth moves, $\DDD^\circ_2\cdots\DDD^\circ_k$ remains unchanged.
We then obtain
\begin{align*}
\XX^\circ_1\DDD_1^\circ\DDD_2^\circ\cdots\DDD_k^\circ&= \left( \mathsf{K}^{2}\ZZZ\YLu_1^{(1)}-\mathbf{I}_1 \right)\DDD_2^\circ\cdots\DDD_k^\circ
\end{align*}
The result follows by \textit{left} multiplying by $\left( \XX^\circ_1 \right)^{-1}$ and applying induction to $\DDD^\circ_2\cdots\DDD_k^\circ$.
\end{proof}

\subsubsection{The map $\Psi^\ell_\ZZZ$}\label{Cycle}
Recall that for $\ell=1$, by Lemma \ref{D1D0Hom}, we have a $\UU$-equivariant algebra homomorphism $\Psi^1_\ZZZ:\DD_1\rightarrow\DD_0$.
By combining this with the identity map on $\WW$, it is easy to see that $\Psi^1_\ZZZ$ induces a map $\AAA_1(\ZZZ,\mathsf{K})\rightarrow\AAA_0(\mathsf{K})$.
Abusing notation, we will also denote this induced homomorphism by $\Psi^1_\ZZZ$.

On the other hand, for $\ell>1$, we only have a $\UU$-equivariant algebra homomorphism going the other direction $\Phi^\ell:\DD_0^{\mathrm{IV}}\rightarrow\DD_\ell$.
We can modify it to incorporate $\ZZZ$: define $\Phi^\ell_\ZZZ:\DD_0^{\mathrm{IV}}\rightarrow\DD_\ell$ by
\begin{equation}
\begin{aligned}
\Phi^\ell_\ZZZ(\mathbf{A})&= \XX^\circ,&
\Phi^\ell_\ZZZ(\mathbf{B}^{-1})&= \ZZZ_1\YLu^{(1)}
\end{aligned}
\label{PhiZ}
\end{equation}
This rescaling will only play a cosmetic effect in \ref{tqk}.
By Lemma \ref{DellDet}, we can localize at $\det_\qqq(\XX^\circ)$ and $\det_\qqq(\YLu^{(1)})$ to obtain a homomorphism
\[
\textstyle\Phi_\ZZZ^\ell:\DD_0\rightarrow\DD_\ell^\star\left[ \det_\qqq(\XX^\circ)^{-1} \right]
\]
Our aim here is to leverage this into a map going the opposite direction between the quantized multiplicative quiver varieties.
\begin{lem}\label{PsiLem}
For $\ell>1$, let
\[
\mathsf{K}_2=\cdots =\mathsf{K}_\ell=\qqq^{-n}
\]
We have the following:
\begin{enumerate}
\item For these parameters, $\Phi^\ell_\ZZZ$ induces a map $\Phi^\ell_\ZZZ:\AAA_0(\mathsf{K}_1)\rightarrow \AAA_\ell(\ZZZ,\mathsf{K})[\det_\qqq(\XX^\circ)^{-1}]$.
\item Moreover, if $\mathsf{K}_1=\qqq^k$ for $k\in\ZZ$, then the map from part (1) is an isomorphism.
We use $\Psi^\ell_\ZZZ$ to denote the natural map $\Psi^\ell_\ZZZ:\AAA_\ell(\ZZZ,\mathsf{K})\rightarrow \AAA_0(\qqq^k)$.
\end{enumerate}
\end{lem}

\begin{proof}
First, we check that $\Phi^\ell_\ZZZ$ respects the moment map relations once we take the quotient $\MM_\ell^\star[\det_\qqq(\XX^\circ)^{-1}]$ by $I_{\ZZZ\mathsf{K}}$:
\[
\Phi^\ell_\ZZZ\left(\qqq^{-2}\mathbf{B}\mathbf{A}^{-1}\mathbf{B}^{-1}\mathbf{A}\left( \mathbf{I}+(\qqq^{2}-1)\mathbf{d}\mathbf{x} \right) \right)
=
\qqq^{-2}\YL^{(1)}\left( \XX^\circ \right)^{-1}\YLu^{(1)}\XX^\circ\left( \mathbf{I}+(\qqq^2-1)\mathbf{d}\mathbf{x} \right)
\] 
Because the commutation relations between elements of $\DD_\ell$ and $\WW$ are the same regardless of the specific elements of the two algebras, we can proceed as in the proof of Lemma \ref{SmallCob}.
Namely, we can commute $\YLu^{(1)}\XX^\circ$ past $\mu_\WW(\mathbf{M})$, apply (\ref{Cob0}), and then reverse the commutation.
We then obtain:
\begin{align*}
&\quad\qqq^{-2}\YL^{(1)}\left( \XX^\circ \right)^{-1}
\underbrace{\YLu^{(1)}\XX^\circ}_{(\ref{Cob0})}
\left( \mathbf{I}+(\qqq^2-1)\mathbf{d}\mathbf{x} \right)\\
&=\qqq^{-2} \left( \ZZZ_\ell\big/\ZZZ_1 \right)\YL^{(1)}\left( \XX^\circ \right)^{-1}\XX^\circ
\underbrace{\YR^{(\ell)}\left( \mathbf{I}+(\qqq^2-1)\mathbf{d}\mathbf{x} \right)}_{(\ref{MuEll})}\\
&= \mathsf{K}_1^{-2}\YL^{(1)}\YLu^{(1)}= \mathsf{K}_1^{-2}\mathbf{I}
\end{align*}
Thus, we do indeed obtain a map $\Phi_\ZZZ^\ell:\AAA_0(\mathsf{C}_1)\rightarrow\AAA_\ell(\ZZZ,\mathsf{C})[\det_\qqq(\XX^\circ)^{-1}]$.

For surjectivity at $\mathsf{K}_1=\qqq^k$, we first note that in this case, $\Phi_\ZZZ^\ell$ covers $\left[ \MM_\ell\big/I_{\ZZZ\mathsf{K}}^+ \right]^{\UU^{\otimes\ell}}$ due to Lemmas \ref{QTraceLem} and \ref{BCobLem2}.
Corollary \ref{CobDetCor2} then implies that upon localizing at $\{\det_\qqq(\XX^\circ),\det_\qqq(\YLu^{(1)})\}$, the additional localizations at $\{\det_\qqq(\YLu^{(a)}),\det_\qqq(\YR^{(a)})\}$ in the definition of $\AAA_\ell(\ZZZ,\mathsf{K})$ are superfluous.
For injectivity, it suffices to consider the restriction of $\Phi^\ell_\ZZZ$ to $\MM_0^{\mathrm{IV}}$ because localization is exact:
\[
\Phi^\ell_\ZZZ:\AAA_0^{\mathrm{IV}}:=\left( \MM_0^{\mathrm{IV}}\big/I_{\mathsf{K}_1}^{\mathrm{IV}} \right)^{\UU}\rightarrow\left( \MM_\ell\big/I_{\ZZZ\mathsf{K}}^+ \right)^{\UU^{\otimes\ell}}=:\AAA_\ell^+
\]
It is obvious that $\Phi^\ell_\ZZZ$ respects the integral forms.
Recall the grading on $\AAA_0^{\mathrm{IV}}$ from Proposition \ref{QTraceLem0}; by that proposition, each graded piece is finite-dimensional.
By (\ref{DimEq}), it then suffices to show that each graded piece of the kernel has dimension zero once $\qqq\mapsto 1$.
Combining the isomorphisms from Lemma \ref{QTraceLem} and Proposition \ref{QTraceLem0}, we have the following commutative diagram:
\[
\begin{tikzpicture}
\draw (0,0) node {$\AAA_0^{\mathrm{IV}}\big|_{\qqq=1}$};;
\draw (3,0) node {$\mathrm{gr}_{F^\DD}\left(\AAA_\ell^{+}\big|_{\qqq=1}\right)$};;
\draw (0,-2.5) node {$\mathcal{R}_1^{\Sigma_n}$};;
\draw (2.5,-2.5) node {$\mathcal{R}_\ell^{\Sigma_n}$};;
\draw[->] (.75,0)--(1.75,0);;
\draw[dashed,->] (.75,-2.5)--(1.75,-2.5);;
\draw[->] (0,-2)--(0,-.5);;
\draw[->] (2.5,-2)--(2.5,-.5);;
\draw (0.375,-1.25) node {$\cong$};;
\draw (2.125,-1.25) node {$\cong$};;
\draw (1.25, -0.375) node {$\Phi^\ell_\ZZZ$};;
\draw (1.25, -2.125) node {$f$};;
\end{tikzpicture}
\]
The induced map $f$ can be described as follows.
Letting
\begin{align*}
\mathcal{R}_1&=\CC[x_1,\ldots, x_n, y_1,\ldots, y_n]\\
\mathcal{R}_\ell&= \CC[z_1,\ldots,z_n,w_1,\ldots, w_n]^{(\ZZ/\ell\ZZ)^n}
\end{align*}
then $f$ is the restriction to $\Sigma_n$-invariants of the homomorphism $\mathcal{R}_1\rightarrow\mathcal{R}_\ell$ where 
\begin{equation*}
\begin{aligned}
x_i&\mapsto z_i^\ell,& y_i&\mapsto z_iw_i
\end{aligned}
\end{equation*}
This latter homomorphism is injective; thus, $f$ and $\Phi^\ell_\ZZZ$ are injective as well.
\end{proof}

\subsection{Radial parts}
For the rest of this section, we set the $\mathsf{K}$-parameter as
\begin{equation}
\begin{aligned}
\mathsf{Q}^k&:=(\qqq^k,\qqq^{-n},\ldots,\qqq^{-n})\\
\mathsf{T}&:=(\qqq^{-1}\ttt,\qqq^{-n},\ldots,\qqq^{-n})
\end{aligned}
\label{QTParam}
\end{equation}
The relationship between $\AAA_\ell(\ZZZ,\mathsf{T})$ to $\SHH_n^\ell(\ZZZ,\qqq,\ttt)$ is facilitated by Etingof--Kirillov's realization of Macdonald polynomials via \textit{intertwiners}.
Namely, we relate the latter to a function representation for $\AAA_\ell(\ZZZ,\mathsf{T})$.
To establish an isomorphism between the two algebras, we can use Lemma \ref{QTraceLem}(1) to reduce the problem to showing the image of $\AAA_\ell(\ZZZ,\mathsf{Q}^k)$ under the action on intertwiners contains the image of $\SHH_n^\ell(\ZZZ,\qqq,\qqq^k)$ under its action on symmetric polynomials.

\subsubsection{Etingof-Kirillov theory at $\ttt=\qqq^k$}\label{EtKirqk}
Recall the quantum symmetric power $S_\qqq^m\mathbb{V}$ from \ref{QSym}.
For $k\ge 1$, let
\[
U_k:= S_\qqq^{n(k-1)}\mathbb{V}\otimes \mathbb{1}_{\qqq^{-(k-1)}}
\]
We note that the zero weight space $U_k[0]$ is one-dimensional; let us make a choice of spanning vector $u\in U_k[0]$.
For a representation $V\in\mathfrak{C}_\ZZ$, we denote by $I(V,U_k)$ the space of $U_k$-\textit{intertwiners}:
\[
I(V, U_k):=\Hom_\UU(V, V\otimes U_k)
\]

For $\mu\in P$ and $\Phi\in I(V, U_k)$, we can take the $\qqq^\mu$-weighted trace:
\[
\tr_{V}\left( \Phi(\qqq^\mu\cdot -) \right)\in U_k[0]
\]
Using our fixed spanning vector $u\in U_k[0]$, we can identify this trace with an element of $\CC(\qqq)$.
Letting $\mu$ vary, we obtain a \textit{weighted trace function}.
We can identify a weighted trace function with an element of the ring of Laurent polynomials $\CC(\qqq)[x_1^{\pm 1},\ldots,x_n^{\pm 1}]$ via the map
\[
x_i\mapsto \qqq^{\langle \epsilon_i, -\rangle}
\]

Finally, recall now the staircase partition $\delta=(n-1,n-2,\ldots, 1, 0)$.
We set 
\[
V_\lambda^k:=V_{\lambda+(k-1)\delta}
\]

\begin{thm}[\cite{EtKirQuant}]\label{EtKirThm}
Let $k\ge 1$. 
We have the following:
\begin{enumerate}
\item For $\lambda\in P^+$, the space $I(V^k_\lambda, U_k)$ is one-dimensional.
There exists a unique intertwiner $\Phi_\lambda^k\in I(V^k_\lambda, U_k)$ such that for any choice of highest weight vector $v_{\lambda+k\delta}\in V^k_\lambda$,
\[
\Phi_\lambda^k(v_{\lambda+k\delta})=v_{\lambda+k\delta}\otimes u +\hbox{ {\rm terms of lower weight in} }V^k_\lambda
\]
\item Let $\varphi^k_\lambda$ denote the weighted trace function of $\Phi_\lambda^k$.
Then we have
\[
P_\lambda(\qqq,\qqq^k)=\frac{\varphi_\lambda^k}{\varphi_0^k}
\]
\end{enumerate}
\end{thm}

It follows from Theorem \ref{EtKirThm}(2) that for fixed $k$, a $U_k$-intertwiner is determined by its weighted trace function.
This gives a realization of Macdonald polynomials (and thus symmetric Laurent polynomials) that directly interacts with our rings of quantum differential operators.
Specifically, to a $U_k$-intertwiner $\Phi\in I(V,U_k)$, we associate an invariant vector ${}_{\cup}\Phi\in (V^*\otimes V\otimes U_k)^\UU$ using $\mathrm{qcoev}_V$:
\[
{}_\cup\Phi:=(1\otimes\Phi)\circ \mathrm{qcoev}_{V}
\]
\[
\includegraphics{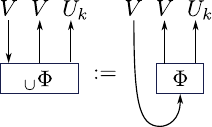}
\]
This is clearly a vector space isomorphism.

Recall that $\DD_0$ and $\WW$ have their function representations on $\OO_G$ and $S_\qqq\mathbb{V}$, respectively.
Thus, $\MM_0$ acts equivariantly on the \textit{braided} tensor product of modules $\OO_G\,\widetilde{\otimes}\,S_\qqq\mathbb{V}$: here, elements of $\WW$ commute past elements of $\OO_G$ via the $R$-matrix.
Let us denote the representation map by $\mathfrak{R}$.
It follows that $\MM_0^\UU$ preserves the $\mathbb{1}_{\qqq^{k-1}}$-isotypic component of $\OO_G\,\widetilde{\otimes}\,S_\qqq\mathbb{V}$, which is precisely $(\OO_G\otimes U_k)^\UU$.
\begin{thm}[\cite{EtKirQuant, VarVassRoot}]\label{QHC0}
For $k\ge 1$, we have the following:
\begin{enumerate}
\item The action of $\MM_0^\UU$ on $(\OO_G\otimes U_k)^\UU$ descends to an action of $\AAA_0(\qqq^{k-1})$.
\item For $\mathrm{qcoev}_{V_\mu}^{\mathbf{A}}(1):=\mathrm{qcoev}_{V_\mu}(1)\in(\OO_G^{\mathbf{A}})^\UU$, the weighted trace function of $\mathrm{qcoev}_{V_\mu}^{\mathbf{A}}(1)\cdot{}_{\cup}\Phi_\lambda^k$ is $s_\mu P_\lambda(\qqq,\qqq^k)$, where $s_\mu$ is the Schur function.
\item For $\mathrm{qcoev}_{V_\mu}^{\mathbf{B}}(1):=\mathrm{qcoev}_{V_\mu}(1)\in(\OO_G^{\mathbf{B}})^\UU$, we have
\[
\mathrm{qcoev}_{V_\mu}^{\mathbf{B}}(1)\cdot{}_\cup\Phi_\lambda^k=\qqq^{-|\mu|(n-1)}s_\mu(\qqq^{2(\lambda_1+k(n-1))}, \qqq^{2(\lambda_2+k(n-2))},\ldots,\qqq^{2\lambda_n}){}_\cup\Phi_\lambda^k
\]
\end{enumerate}
\end{thm}
\noindent Setting $\mu=\omega_i$ in part (3), we obtain scaled versions of the Macdonald operators.
Theorem \ref{QHC0}(1) gives an algebra homomorphism
\[
\mathfrak{rad}^0_{\qqq^k}: \AAA_0(\qqq^{k-1})\rightarrow \mathrm{End}\left(\Lambda_n^\pm(\qqq)\right)
\] 
called the \textit{quantum radial parts map}, which was first defined in \cite{VarVassRoot}.
he following is one of the main results of \cite{QHarish}:
\begin{thm}
For $k>2n$, $\mathfrak{rad}^0_{\qqq^k}$ is an isomorphism onto the image of $\SHH_n^0(\qqq,\qqq^k)$ under the polynomial representation map $\mathfrak{r}$.
\end{thm}

\subsubsection{Harish-Chandra isomorphism at $\ttt=\qqq^k$}\label{tqk}
For $k\ge 1$, Etingof--Kirillov theory allows us to define an action of $\AAA_0(\qqq^{k-1})$ on Macdonald polynomials at $\ttt=\qqq^k$ via its function representation.
By precomposing the action map with $\Psi_\ZZZ^\ell:\AAA_\ell(\ZZZ,\mathsf{Q}^{k-1})\rightarrow\AAA_0(\qqq^{k-1})$, we thus have an action of $\AAA_\ell(\ZZZ,\mathsf{Q}^{k-1})$ on Macdonald polynomials as well.
Thus, we can define a radial parts map at $\ell>0$:
\[
\mathfrak{rad}^\ell_{\qqq^k}: \AAA_\ell(\ZZZ,\mathsf{Q}^{k-1})\rightarrow \mathrm{End}\left(\Lambda_n^\pm(\qqq)\right)
\]
Our main goal is to show that $\mathrm{im}(\mathfrak{rad}_{\qqq^k}^\ell)$ contains $\mathfrak{r}\left( \SHH_n^\ell(\ZZZ,\qqq,\qqq^k) \right)$.

For $k>2n\ell$, recall that by Lemma \ref{CycGen}, $\SHH_n^\ell(\ZZZ,\qqq,\qqq^k)$ is generated by
\begin{equation}
\ee\CC(\qqq)[X_1,\ldots, X_n]\ee\cup\ee\CC(\qqq)[Y_1^{\pm 1},\ldots, Y^{\pm 1}_n]\ee\cup\ee\CC(\qqq)[D_1,\ldots D_n]\ee
\label{XYD}
\end{equation}
Let $\mathcal{R}(\mathbf{X^\circ})\subset\MM_\ell$ be the subalgebra generated by the entries of $\XX^\circ$ and let $\OO_G(\YLu^{(1)})\subset\MM_\ell$ be the sublagebra generated by the entries of $\YLu^{(1)}$ and $\YL^{(1)}$.
Additionally, let $\mathcal{R}(\mathbf{A}^{-1})\subset\MM_0$ be the subalgebra generated by the entries of $\mathbf{A}^{-1}$.
From Theorem \ref{QHC0}, one can deduce 
\begin{align*}
\mathfrak{rad}_{\qqq^k}^\ell\left( \mathcal{R}(\XX^\circ)^{\UU^{\otimes\ell}}\right) &= \ee \CC(\qqq)[X_1,\ldots, X_n]\ee\\
\mathfrak{rad}_{\qqq^k}^\ell\left(\OO_G(\YLu^{(1)})^{\UU^{\otimes\ell}}\right)&= \ee  \CC(\qqq)[Y_1,\ldots, Y_n]\ee\\
\mathfrak{rad}_{\qqq^k}^0\left( \mathcal{R}(\mathbf{A}^{-1})^\UU \right)&= \ee\CC(\qqq)[X_1^{-1},\ldots, X_n^{-1}]\ee
\end{align*}
Thus, the first two subalgebras in (\ref{XYD}) lie in $\mathrm{im}( \mathfrak{rad}_{\qqq^k}^\ell )$.

To consider $\ee \CC(\qqq)[D_1,\ldots, D_n]\ee$, recall the element $\gamma_\ZZZ$ (\ref{GammaZ}).
By Proposition \ref{GammaZProp}(1), we have
\begin{equation}
\ee\gamma_\ZZZ \CC(\qqq)[X_1^{-1},\ldots, X_n^{-1}]\gamma_\ZZZ^{-1}\ee= \ee \CC(\qqq)[D_1,\ldots, D_n]\ee
\label{GammaXD}
\end{equation}
By Proposition \ref{GammaZProp}(2), we can extend $\mathfrak{r}(\gamma_\ZZZ)$ to an operator $\mathfrak{R}(\gamma_\ZZZ)$ on $\OO_G\,\widetilde{\otimes}\, S_\qqq\mathbb{V}$ by
\[
\mathfrak{R}(\gamma_\ZZZ)\big|_{V_\lambda^*\otimes V_\lambda\otimes S_\qqq\mathbb{V}}=\left( \prod_{a=1}^\ell\prod_{i=1}^n\left( \qqq^{-2(n-i)-1}\ZZZ_a^{-1};\qqq^{2} \right)_{-\lambda_i} \right)
\]
Notice that we have set $\ttt=\qqq^k$ in (\ref{GammaEigen}) and incorporated the shift in highest weight $V^k_{\lambda}:=V_{\lambda+(k-1)\delta}$.
This extended operator is invertible.
The following is an equality analogous to (\ref{GammaXD}) but for $\mathrm{im}( \mathfrak{rad}_{\qqq^k}^\ell )$:
\begin{lem}\label{XDqk}
Let $\mathcal{R}(\DDD^\circ)\subset\MM_\ell$ be the subalgebra generated by the entries of $\DDD^\circ$.
We have
\[
\mathfrak{R}(\gamma_\ZZZ)\mathfrak{rad}_{\qqq^k}^0\left( \mathcal{R}(\mathbf{A}^{-1})^\UU \right)\mathfrak{R}(\gamma_\ZZZ)^{-1}=\mathfrak{rad}_{\qqq^k}^\ell\left(\mathcal{R}(\DDD^\circ)^{\UU^{\otimes\ell}} \right)
\]
\end{lem}

\begin{proof}
Since conjugation is a ring homomorphism, it suffices to show that
\[
\mathfrak{R}(\gamma_\ZZZ)\mathfrak{R}(\mathbf{A}^{-1})\mathfrak{R}(\gamma_\ZZZ)^{-1}\overset{?}{=}(-1)^\ell\mathfrak{R}\left(\Psi_\ZZZ^{\ell}(\DDD^\circ)\right)
\]
To that end, recall that the entries of $\mathbf{A}^{-1}$ lie in $\mathbb{V}\otimes\mathbb{V}^*\cong\mathbb{V}^{**}\otimes\mathbb{V}^*$.
When they are multiplied to an element of $V_\lambda^*\otimes V_\lambda\otimes S_\qqq\mathbb{V}$, the result lies in the direct sum of $V_\mu^*\otimes V_\mu\otimes S_\qqq\mathbb{V}$ where  $\lambda-\mu=\epsilon_i$ for some $i$.
For such a $\mu$, we have:
\begin{equation}
\mathfrak{R}(\gamma_\ZZ)\big|_{V_\mu^*\otimes V_\mu\otimes S_\qqq\mathbb{V}}\bigg/
\mathfrak{R}(\gamma_\ZZ)\big|_{V_\lambda^*\otimes V_\lambda\otimes S_\qqq\mathbb{V}}
=
\prod_{a=1}^\ell\left( 1-\qqq^{-2\lambda_i-2(n-i)-1}\ZZZ_a^{-1} \right)
\label{GammaRat}
\end{equation}

We can write the ratio of eigenvaues (\ref{GammaRat}) in terms of the ribbon element $\nu$.
Specifically, let us define an operator $\mathfrak{R}(\nu)$ on $\OO_G\,\widetilde{\otimes}\,S_\qqq\mathbb{V}$ where for $v^*\otimes v\otimes x\in V^*\otimes V\otimes S_\qqq\mathbb{V}$,
\[
\mathfrak{R}(\nu)\cdot\left( v^*\otimes v\otimes x \right)= (\nu v^*)\otimes v\otimes x
\]
Using (\ref{RibbonAct}), we compute:
\begin{align*}
\mathfrak{R}(\nu)\big|_{V^*_\mu\otimes V_\mu\otimes S_\qqq\mathbb{V}}\bigg/\mathfrak{R}(\nu)\big|_{V^*_\lambda\otimes V_\lambda\otimes S_\qqq\mathbb{V}}
&= \qqq^{-\langle\lambda-\epsilon_i,\lambda-\epsilon_i+2\rho\rangle+\langle\lambda,\lambda+2\rho\rangle}\\
&= \qqq^{2\lambda_i+n+1-2i}
\end{align*}
In comparison with (\ref{GammaRat}), we have:
\[
\mathfrak{R}(\gamma_\ZZZ)\mathfrak{R}(\mathbf{A}^{-1})\mathfrak{R}(\gamma_\ZZZ)^{-1}
=\bigg[\left( 1-\ZZZ_1^{-1}\qqq^{-n}\ad_{\mathfrak{R}(\nu)}^{-1} \right)\circ\cdots\circ\left( 1-\ZZZ_\ell^{-1}\qqq^{-n}\ad_{\mathfrak{R}(\nu)}^{-1} \right)\bigg]\left( \mathfrak{R}(\mathbf{A}^{-1}) \right)
\]
Finally, we apply Lemma A.8 from \cite{QHarish}:
\begin{align*}
\ad^{-1}_{\mathfrak{R}(\nu)}\left(\mathfrak{R}(\mathbf{A}^{-1})\right)&=\qqq^n\mathfrak{R}\left( \mathbf{A}^{-1}\mathbf{B}^{-1} \right)\\
\ad^{-1}_{\mathfrak{R}(\nu)}\left(\mathfrak{R}(\mathbf{B}^{-1})\right)&= \mathfrak{R}\left(\mathbf{B}^{-1}\right)
\end{align*}
We thus have
\begin{align*}
\mathfrak{R}(\gamma_\ZZZ)\mathfrak{R}(\mathbf{A}^{-1})\mathfrak{R}(\gamma_\ZZZ)^{-1}
&= \mathfrak{R}\left(\mathbf{A}^{-1}\left(\mathbf{I}-\ZZZ_1^{-1}\mathbf{B}^{-1}  \right)\cdots\left( \mathbf{I-\ZZZ_\ell^{-1}\mathbf{B}}^{-1} \right)\right)\\
&= (-1)^\ell\mathfrak{R}\left( \Psi_\ZZZ^\ell(\DDD^\circ) \right)
\end{align*}
by Lemma \ref{BigCobLem1} (cf. (\ref{PhiZ})).
\end{proof}

\begin{thm}
For $k>2n\ell$, the radial parts map induces an isomorphism:
\[
\mathfrak{r}^{-1}\circ\mathfrak{rad}_{\qqq^k}^\ell:\AAA_\ell(\ZZZ,\mathsf{Q}^{k-1})\rightarrow\SHH_n^\ell(\ZZZ,\qqq,\qqq^k)
\]
\end{thm}

\begin{proof}
By Lemma \ref{XDqk}, $\mathrm{im}(\mathfrak{rad}_{\qqq^k}^\ell)$ contains $\mathfrak{r}(\SHH_n^\ell(\ZZZ,\qqq,\qqq^k))$.
Recall the notation from the proof of Lemma \ref{PsiLem}:
\[
\AAA_\ell^+:=\left( \MM_\ell\big/I_{\ZZZ\mathsf{Q}^{k-1}}^+ \right)^{\UU^{\otimes\ell}}
\]
We have that $\mathfrak{rad}_{\qqq^k}^\ell(\AAA_\ell^+)$ contains $\mathfrak{r}(\SHH_n^\ell(\ZZZ,\qqq,\qqq^k)^-)$, and moreover, the filtration $F_\DD$ is sent to the filtration defined in Proposition \ref{CyclicPBW}.
Taking the associated graded, Corollary \ref{PBWCor} gives the opposite inequality to Lemma \ref{QTraceLem}(2), from which the theorem follows.
\end{proof}

\subsubsection{Etingof-Kirillov theory at generic $\ttt$}
Here, we review the necessary changes to \ref{EtKirqk} for considering the case of generic $\ttt$.
To discuss Etingof--Kirillov theory here, we will need to base change to $\KK:=\CC(\qqq,\ttt)$.
We will formally denote $\ttt=\qqq^\alpha$, although a logarithm for $\qqq$ is unnecessary and all the objects below can be defined in terms of $\ttt$.
Recall that at $\ttt=\qqq^k$, we had a vector space isomorphism:
\begin{align*}
\Lambda^\pm_n(\qqq)&\cong\left( \OO_G\otimes S_\qqq^{n(k-1)}\mathbb{V}\otimes\mathbb{1}_{-(k-1)} \right)^\UU\\
P_\lambda(\qqq,\qqq^k)&\mapsto {}_\cup\Phi_\lambda^k\in(V_{\lambda+(k-1)\delta}^*\otimes V_{\lambda+(k-1)\delta}\otimes S_\qqq^{n(k-1)}\mathbb{V}\otimes\mathbb{1}_{-(k-1)})^\UU
\end{align*}
Each tensorand on the right-hand-side will need to be replaced.

First, recall the isomorphism from \ref{WeylDiff} between $S_\qqq\mathbb{V}$ and a \textit{commutative} polynomial ring $\CC(\qqq)[\mathbf{z}_n]$.
This allowed us to write the $\UU$-equivariant $\WW$-action on $S_\qqq\mathbb{V}$ in terms of a homomorphism into a ring of $\qqq$-difference operators:
\[
\mathfrak{qdiff}:\WW\rtimes\UU\rightarrow\mathcal{D}\hbox{\it iff}\,^+_\qqq(\mathbf{z}_n)
\]
The ring $\mathcal{D}\hbox{\it iff}\,^+_\qqq(\mathbf{z}_n)$ acts on the $\KK$-vector space $W_\alpha$ spanned by the monomials
\[
\left\{ z_1^{m_1+\alpha-1}\cdots z_n^{m_n+\alpha-1}\,\middle|\, (m_1,\ldots, m_n)\in\ZZ^n \right\}
\]
(here, we can omit the $\alpha$ and just impose that the shift operator $\tau_{z_i,\qqq}$ multiplies by $\ttt$).
Precomposing with $\mathfrak{qdiff}$, we obtain a $\UU$-equivariant action of $\WW$ on $W_\alpha$.
We note that the zero weight space of $W_{\alpha}\otimes\mathbb{1}_{\qqq\ttt^{-1}}$ is one-dimensional.
As in the $\ttt=\qqq^k$ case, we pick a spanning vector $u$.

Next, the finite dimensional irreducible $\UU$-module $V^k_\lambda:=V_{\lambda+k\delta}$ will need to be replaced with the \textit{Verma module} $M^\alpha_\lambda:=M_{\lambda+(\alpha-1)\delta}$.
We note that we can define the highest weight of the Verma module in terms of $\ttt$ rather than $\alpha$.
The structures in \ref{Representations} can be applied to $M^\alpha_\lambda$ except those involving the coevaluation because it is infinite-dimensional.
Nonetheless, we can still define a weighted trace over $M^\alpha_\lambda$, which will return a power series.

\begin{thm}[\cite{EtKirQuant}]
For $\lambda\in P^+$, there exists a unique intertwiner 
\[\Phi_\lambda^\alpha\in \Hom_\UU(M_\lambda^\alpha, M_\lambda^\alpha\otimes W_\alpha\otimes\mathbb{1}_{\qqq\ttt^{-1}})\]
such that for a highest weight vector $v^\alpha_\lambda\in M_\lambda^\alpha$,
\[
\Phi_\lambda^\alpha(v^\alpha_\lambda)=v^\alpha_\lambda\otimes u +\hbox{\rm lower order terms}
\]
Let $\varphi_\lambda^\alpha$ denote its weighted trace over $M_\lambda^\alpha$.
We then have
\[
P_\lambda(\qqq,\ttt)=\frac{\varphi^\alpha_\lambda}{\varphi^\alpha_0}
\]
\end{thm}

\subsubsection{Harish-Chandra isomorphism at generic $\ttt$}
The lack of coevaluation maps for $M_\lambda^\alpha$ forces us to work with intertwiners rather than invariant vectors.
Consider the space
\begin{align*}
\mathrm{tHom}_\alpha&:=\bigoplus_{M\hbox{ {\tiny is highest weight}}}
\Hom_\KK\left( M,M\otimes W_\alpha\otimes\mathbb{1}_{\qqq\ttt^{-1}} \right)\bigg/\\
&\left\langle
(f\otimes 1)\circ\psi-\psi\circ f\,\middle|\,
\begin{array}{l}
\psi\in \Hom_{\KK}\left( M',M\otimes W_\alpha\otimes\mathbb{1}_{\qqq\ttt^{-1}} \right),\\
f\in\Hom_\UU(M,M')
\end{array}
\right\rangle
\end{align*}
The relations are analogous to the coend relation (\ref{Coend}).
We also set $\mathrm{Int}^+_\alpha\subset \mathrm{tHom}_\alpha$ to be the subspace spanned by the intertwiners $\{\Phi_\lambda^\alpha\}_{\lambda\in P^+}$.

Now, let us address the base change of quantum differential operators.
Let $\LL_\qqq:=\CC(\qqq)[\ttt^{\pm 1}]$ and recall the parameter $\mathsf{T}$ (\ref{QTParam}).
We denote by
\begin{align*}
\MM_{0,\LL}&:=\MM_0\otimes\LL_\qqq\\
%\MM_{\ell,\LL}^\star &:=\MM_\ell^\star\otimes \LL_\qqq\\
\AAA_{0}(\qqq^{-1}\ttt)_\LL&:=\left( \MM_{0,\LL}\big/I_{\qqq^{-1}\ttt} \right)^\UU\\
\AAA_{0}(\qqq^{-1}\ttt)_\KK&:=\AAA_0(\qqq^{-1}\ttt)_\LL\otimes\KK\\
\MM^\star_{\ell,\LL}&:=\MM^\star_\ell\otimes\LL_\qqq\\
\AAA_\ell(\ZZZ,\mathsf{T})_\LL&:=\left( \MM^\star_{\ell,\LL}\big/I_{\ZZZ\mathsf{T}} \right)^{\UU^{\otimes\ell}}\\
\AAA_\ell(\ZZZ,\mathsf{T})_\KK&:=\AAA_\ell(\ZZZ,\mathsf{T})_\LL\otimes\KK
\end{align*}

\begin{prop}[\cite{QHarish}]
There exists a subspace $\mathrm{tHom}_\alpha^D\subset \mathrm{tHom}_\alpha$ on which $\MM_{0,\LL}$ acts.
This subspace contains $\mathrm{Int}_\alpha^+$ and the action of $\MM_{0,\LL}^\UU$ preserves $\mathrm{Int}_\alpha^+$.
Moreover, the action of $\MM_{0,\LL}^\UU$ on $\mathrm{Int}_\alpha^+$ factors through $\AAA_0(\qqq^{-1}\ttt)_\LL$.
\end{prop}

Identifying $\mathrm{Int}^+_\alpha$ with $\Lambda_n^\pm(\qqq,\ttt)$ by sending $\Phi_\lambda^\alpha$ to $P_\lambda(\qqq,\ttt)$, we obtain the radial parts maps
\begin{align*}
%\mathfrak{rad}_{\ttt,\LL}^0:\AAA_0(\qqq^{-1}\ttt)_\LL&\rightarrow\mathrm{End}\left( \Lambda_n^\pm(\qqq,\ttt) \right)\\
\mathfrak{rad}_{\ttt}^0:\AAA_0(\qqq^{-1}\ttt)&\rightarrow\mathrm{End}\left( \Lambda_n^\pm(\qqq,\ttt) \right)
\end{align*}

\begin{prop}[\cite{QHarish}]
The radial parts map satisfies the following:
\begin{enumerate}
%\item $\mathfrak{rad}_{\ttt,\LL}^0\big|_{\ttt=\qqq^k}=\mathfrak{rad}_{\qqq^k}^0$;
\item the image of $\mathrm{qcoev}_{V_\mu}^\mathbf{A}(1):=\mathrm{qcoev}_{V_\mu}(1)\in\OO_G^\mathbf{A}$ is
\[
\mathfrak{rad}_\ttt^0(\mathrm{qcoev}_{V_\mu}^{\mathbf{A}}(1))=\mathfrak{r}\left(\ee s_\mu(X_1,\ldots, X_n)\ee\right)
\]
\item the image of $\mathrm{qcoev}_{V_\mu}^\mathbf{B}(1):=\mathrm{qcoev}_{V_\mu}(1)\in\OO_G^\mathbf{B}$ is
\[
\mathfrak{rad}_\ttt^0(\mathrm{qcoev}_{V_\mu}^{\mathbf{B}}(1))=t^{-|\mu|(n-1)}\mathfrak{r}\left(\ee s_\mu(Y_1,\ldots, Y_n)\ee\right)
\]
\end{enumerate}
\end{prop}

To define a radial parts map for $\ell>1$, we need the following:

\begin{prop}
The map from Lemma \ref{PsiLem}(1)
\[
\Phi_\ZZZ^\ell:\AAA_0(\qqq^{-1},\ttt)\rightarrow \AAA_\ell(\ZZZ,\mathsf{T})[\textstyle\det_\qqq(\XX^\circ)^{-1}]
\] 
is injective.
\end{prop}

\begin{proof}
Borrowing notation from the proof of Lemma \ref{PsiLem}, it suffices to show that the restricted map
\[
\Phi_\ZZZ^\ell:\AAA_0^{\mathrm{IV}}(\qqq^{-1}\ttt)\rightarrow\AAA_\ell^+(\ZZZ,\mathsf{T})
\]
is injective.
Moreover, the map above is defined for the corresponding algebras defined over $\LL_\qqq$, which we denote by 
\[
\Phi_{\ZZZ,\LL}^\ell:\AAA_0^{\mathrm{IV}}(\qqq^{-1}\ttt)_\LL\rightarrow\AAA_\ell^{+}(\ZZZ,\mathsf{T})_\LL
\]
Endowing $\AAA_0^{\mathrm{IV}}$ with a grading from Proposition \ref{QTraceLem0}, we only need to check injectivity on each finite-dimensional graded piece.
On the other hand, by Lemma \ref{PsiLem}(2), $\Phi_{\ZZZ,\LL}^\ell\big|_{\ttt=\qqq^k}$ is injective for all $k\in\ZZ$.
This implies that on each graded piece, the kernel of $\Phi_{\ZZZ,\LL}^\ell$ must be torsion and thus vanishes upon base change to $\KK$.
\end{proof}

Thus, for $\ell>1$, we can define 
\[\Psi_\ZZZ^\ell:\AAA_\ell(\ZZZ,\mathsf{T})\rightarrow\AAA_0(\qqq^{-1}\ttt)\]
For all $\ell>0$, we define the radial parts map by
\[
\mathfrak{rad}_{\ttt}^\ell:=\mathfrak{rad}_{\ttt}^0\circ\Psi_\ZZZ^\ell
\]

\begin{thm}
The radial parts map induces an isomorphism
\[
\mathfrak{r}^{-1}\circ\mathfrak{rad}_\ttt^\ell:\AAA_\ell(\ZZZ,\mathsf{T})\rightarrow \SHH_n^\ell(\ZZZ,\qqq,\ttt)
\]
\end{thm}

\begin{proof}
The proof that the $\mathrm{im}(\mathfrak{rad}_\ttt^\ell)$ contains $\mathfrak{r}(\SHH_n^\ell(\ZZZ,\qqq,\ttt))$ is similar to the case of $\ttt=\qqq^k$.
To show that the latter is the entire image, it suffices to show that $\AAA_\ell(\ZZZ,\mathsf{T})$ is generated by the three subalgebras
\[
\mathcal{R}(\XX^\circ)^{\UU^{\otimes\ell}}\cup\OO_G(\YLu^{(1)})^{\UU^{\otimes\ell}}\cup\mathcal{R}(\DDD^\circ)^{\UU^{\otimes\ell}}
\]
For that, it suffices to show generation for $\AAA_\ell^+(\ZZZ,\mathsf{T})$.
On each finite-dimesional filtered piece $F_{l_1,l_2}$ (cf. \ref{Class}) of $\AAA_\ell^+(\ZZZ,\mathsf{T})_\LL$, this is the case for $\ttt=\qqq^k$.
Generation over $\KK$ then follows from Nakayama's Lemma.

For injectivity, consider the map
\[
\mathfrak{rad}_{\ttt,\LL}^\ell:\AAA_\ell^+(\ZZZ,\mathsf{T})_\LL\rightarrow\mathrm{End}\left( \Lambda_n^\pm(\qqq,\ttt) \right)
\]
given by precomposing $\mathfrak{rad}_\ttt^\ell$ with the localization map.
On each filtered piece $F_{l_1,l_2}$, this map projects away the torsion submodule and then applies $\mathfrak{rad}_\ttt^\ell$.
Moreover, this map is injective on $F_{l_1,l_2}$ at infinitely many evaluations $\ttt=\qqq^k$, and thus the kernel must be torsion.
\end{proof}

\appendix

\section{Shifted quantum toroidal algebras}
In this appendix, we show that $\SHH_n^\ell(\ZZZ,\qqq,\ttt)$ is isomorphic to the image of a \textit{shifted quantum toroidal algebra} under Tsymbaliuk's \textit{GKLO homomorphism} \cite{TsymGKLO}.
The latter is expected to be isomorphic to a $K$-\textit{theoretic Coulomb branch} \cite{BFN,FTMult, FTShift}, and thus we make marginal progress towards the conjecture in \cite{BEF} concerning these Coulomb branches.
Our main reference is \cite{TsymGKLO}, although note that to be consistent with our DAHA conventions, we have squared the parameters $(\qqq,\ttt)$.

\subsection{Definition of shifted quantum toroidal algebras}
For a nonnegative integer $\ell$, the \textit{shifted quantum toroidal algebra} $U_{\qqq,\ttt}^\ell(\ddot{\gl}_1)$ is a $\KK$-algebra with generators
\[
\left\{ e_r, f_r \right\}_{r\in\ZZ}\cup\left\{\psi_s^+\right\}_{s\ge 0}\cup\left\{ \psi_{s}^{-} \right\}_{s\le \ell}\cup\left\{(\psi_0^+)^{-1},(\psi_{\ell}^-)^{-1}\right\}
\]
To describe its relations, we place the generators into \textit{currents}:
\begin{equation*}
\begin{aligned}
e(z)&:= \sum_{r\in\ZZ}e_rz^{-r},&
f(z)&:= \sum_{r\in\ZZ}f_rz^{-r},&
\psi^+(z)&:= \sum_{s\ge 0}\psi_s^+z^{-s},&
\psi^-(z)&:= \sum_{s\le\ell}\psi_{s}^-z^{-s}
\end{aligned}
\end{equation*}
The relations are then:
\begin{gather}
\nonumber
\left[ \psi^\epsilon(z),\psi^{\epsilon'}(w) \right]=0,\, \psi^+_0\left( \psi^+_0 \right)^{-1}=\psi^-_\ell\left( \psi^-_\ell \right)^{-1}=1,\\
\nonumber
\left(z-\qqq^2 w\right)\left(z-\ttt^2 w\right)\left(z-\qqq^{-2}\ttt^{-2}w\right)e(z)e(w)=\left(\qqq^2 z-w\right)\left(\ttt^2 z-w\right)\left(\qqq^{-2}\ttt^{-2} z-w\right)e(w)e(z),\\
\nonumber
\left(\qqq^2 z-w\right)\left(\ttt^2 z-w\right)\left(\qqq^{-2}\ttt^{-2} z-w\right)f(z)f(w)=\left(z-\qqq^2 w\right)\left(z-\ttt^2 w\right)\left(z-\qqq^{-2}\ttt^{-2}w\right)f(w)f(z),\\
\label{PsiE}
\left(z-\qqq^2 w\right)\left(z-\ttt^2 w\right)\left(z-\qqq^{-2}\ttt^{-2}w\right)\psi^\epsilon(z)e(w)=\left(\qqq^2 z-w\right)\left(\ttt^2 z-w\right)\left(\qqq^{-2}\ttt^{-2} z-w\right)e(w)\psi^\epsilon(z),\\
\label{PsiF}
\left(\qqq^2 z-w\right)\left(\ttt^2 z-w\right)\left(\qqq^{-2}\ttt^{-2} z-w\right)\psi^\epsilon(z)f(w)=\left(z-\qqq^2 w\right)\left(z-\ttt^2 w\right)\left(z-\qqq^{-2}\ttt^{-2}w\right)f(w)\psi^\epsilon(z),\\
\nonumber
\left[ e(z),f(w) \right]=\frac{\delta(z/w)}{\left(1-\qqq^2\right)\left(1-\ttt^2\right)\left(1-\qqq^{-2}\ttt^{-2}\right)}\left( \psi^+(z)-\psi^-(z) \right),\\
\nonumber
\mathrm{Sym}_{z_1,z_2,z_3}\frac{z_2}{z_3}\left[ e(z_1),\left[ e(z_2),e(z_3) \right] \right]=0,\\
\nonumber
\mathrm{Sym}_{z_1,z_2,z_3}\frac{z_2}{z_3}\left[ f(z_1),\left[ f(z_2),f(z_3) \right] \right]=0
\end{gather}
where $\epsilon,\epsilon'\in\{+,-\}$ and $\delta(z)$ is the usual delta-function
\[
\delta(z)=\sum_{r\in \ZZ}z^r
\]

\begin{prop}\label{TorGen}
For any pair of nonzero $\{\mathsf{c}_+,\mathsf{c}_-\}\subset\KK^\times$, quotient $U_{\qqq,\ttt}^\ell(\ddot{\gl}_1)\big/\langle\psi_0^+=\mathsf{c}_+,\psi_\ell^-=\mathsf{c}_-\rangle$ is generated by
\[
\{e_0, f_0\}\cup\left\{ \psi^+_s \right\}_{s\ge 0}\cup\left\{ \psi^-_s \right\}_{s\le \ell}\cup\left\{ (\psi_0^+)^{-1},(\psi_\ell^+)^{-1} \right\}
\]
\end{prop}

\begin{proof}
Unpacking (\ref{PsiE}) and (\ref{PsiF}), we have:
\begin{align*}
\left[\psi^+_1,e_{r}\right]&= \mathsf{c}_+\left( \qqq^{-2}+\ttt^{-2}+\qqq^2\ttt^2-\qqq^2-\ttt^2-\qqq^{-2}\ttt^{-2} \right)e_{r+1}\\
\left[\psi^-_{\ell-1},e_{r}\right]&= \mathsf{c}_-\left( \qqq^2+\ttt^2+\qqq^{-2}\ttt^{-2}-\qqq^{-2}-\ttt^{-2}-\qqq^2\ttt^2 \right)e_{r-1}\\
\left[\psi^+_1,f_{r}\right]&= \mathsf{c}_+\left( \qqq^2+\ttt^2+\qqq^{-2}\ttt^{-2}-\qqq^{-2}-\ttt^{-2}-\qqq^2\ttt^2 \right)f_{r+1}\\
\left[\psi^-_{\ell-1},f_{r}\right]&= \mathsf{c}_-\left( \qqq^{-2}+\ttt^{-2}+\qqq^2\ttt^2-\qqq^2-\ttt^2-\qqq^{-2}\ttt^{-2} \right)f_{r-1}\qedhere
\end{align*}
\end{proof}

\subsection{GKLO homomorphism}
Here, we will map $U_{\qqq,\ttt}^\ell(\ddot{\gl}_1)$ into the ring of $\qqq$-difference operators on the variables $\mathbf{x}_n:=\{x_1,\ldots, x_n\}$.
Let $\tau_{x_i,\qqq}$ denote the $\qqq$-shift operator on $x_i$:
\[
\tau_{x_i,\qqq}(x_j)=\qqq^{\delta_{i,j}}x_j
\]
We denote by $\mathcal{D}\hbox{\it iff}\,_\qqq^\circ(\mathbf{x}_n)$ the ring of $\qqq$-difference operators with coefficients in $\KK(\mathbf{x}_n)$.
The following was proved by Tsymbaliuk:
\begin{prop}[\cite{TsymGKLO}]
For $\ZZZ=(\ZZZ_1,\ldots,\ZZZ_\ell)\in(\CC^\times)^\ell$, the assignment
\begin{align}
\label{TsymE}
e(z)&\mapsto \frac{1}{\qqq^{-2}-1}\sum_{i=1}^n\delta\left( \frac{x_i}{z} \right)\prod_{a=1}^\ell\left( 1-\frac{\ZZZ_a}{\qqq x_i} \right)\prod_{\substack{1\le j\le n\\j\not=i}}\frac{x_i-\ttt^{-2}x_j}{x_i-x_j}\tau_{x_i,\qqq^2}^{-1}\\
\label{TsymF}
f(z)&\mapsto \frac{1}{1-\qqq^2}\sum_{i=1}^n\delta\left( \frac{\qqq^2 x_i}{z} \right)\prod_{\substack{1\le j\le n\\j\not=i}}\frac{x_i-\ttt^2 x_j}{x_i-x_j}\tau_{x_i,\qqq^2}\\
\label{TsymPsi}
\psi^\pm(z)&\mapsto \left[ \prod_{a=1}^\ell\left( 1-\frac{\ZZZ_a}{\qqq z} \right)\prod_{i=1}^n\frac{\left( z-\ttt^{-2} x_i \right)\left( z-\qqq^2\ttt^2 x_i \right)}{\left( z-x_i \right)\left( z-\qqq^2 x_i \right)} \right]^\mp
\end{align}
defines a homomorphism of $\KK$ algebras $\mathrm{T}_\ZZZ:U_{\qqq,\ttt}^\ell(\ddot{\gl}_1)\rightarrow\mathcal{D}\hbox{\it iff}\,_\qqq^\circ(\mathbf{x}_n)$.
In (\ref{TsymPsi}), the $\mp$ on the right hand side denotes taking the Laurent series expansion in $z^{\mp 1}$.
\end{prop}

We will take a closer look at (\ref{TsymE}) and (\ref{TsymF}) later.
For now, we note that (\ref{TsymPsi}) shows that the currents $\psi^{\pm}(z)$ generate operators given by multiplication with symmetric Laurent polynomials.
To see this, let us take logarithmic generators 
\[\{b_m\,|\, m\in\ZZ\backslash\{0\} \}\]
such that
\begin{equation*}
\begin{aligned}
\psi^+(z)&=\exp\left( \sum_{m>1}b_m\frac{z^{-m}}{m} \right),&
\psi^-(z)&=z^{-\ell}\exp\left( \sum_{m>1}b_{-m}\frac{z^{m}}{m} \right)
\end{aligned}
\end{equation*}
Formula (\ref{TsymPsi}) then yields that for $m>0$,
\begin{align*}
\mathrm{T}_\ZZZ(b_m)&= \left(1-\ttt^{-2m}\right)\left(1-\qqq^{2m}\ttt^{2m}\right)\left( x_1^m+\cdots+ x_n^m \right)-\qqq^{-m}\left(\ZZZ_1^m+\cdots+\ZZZ_\ell^m\right)\\
\mathrm{T}_\ZZZ(b_{-m})&= \frac{(-1)^\ell\ZZZ_1\cdots\ZZZ_\ell}{\qqq^\ell}\bigg(
\left( 1-\ttt^{2m} \right)\left( 1-\qqq^{-2m}\ttt^{-2m} \right)\left( x_1^{-m}+\cdots +x_n^{-m} \right)-\qqq^m\left( \ZZZ_1^{-m}+\cdots+\ZZZ_\ell^{-m} \right) \bigg)
\end{align*}
These clearly generate multiplication by power sums.
Finally, we note that $\mathrm{T}_\ZZZ(\psi_0^+)$ and $\mathrm{T}_\ZZZ(\psi_\ell^-)$ are constants and thus Proposition \ref{TorGen} applies to the image of $\mathrm{T}_\ZZZ$.

\subsection{Isomorphism with cyclotomic DAHA}
Before comparing $U_{\qqq,\ttt}^\ell(\ddot{\gl}_1)$ and $\SHH_n^\ell(\ZZZ,\qqq,\ttt)$, we will need some preparatory steps.

\subsubsection{Automorphisms}
The following is in 3.7 of \cite{ChereDAHA}:
\begin{prop}
The formulas below define a $\KK$-algebra automorphism $\tau_X$ of $\HH^0_n(\qqq,\ttt)$:
\begin{align*}
%\tau_Y(T_i)&= T_i,&  \tau_Y(X_1\cdots X_n)&= \qqq^{i}(X_1\cdots X_i)(Y_1\cdots Y_i),&\tau_Y(Y_i)&=Y_i, \\
\tau_X(T_i)&= T_i, & \tau_X(X_i)&= X_i, & \tau_X(Y_1\cdots Y_i)&=\qqq^{-i}(Y_1\cdots Y_i)(X_1\cdots X_i)
\end{align*}
Additionally, we have a $\CC$-algebra anti-involution $\varepsilon$ given by:
\begin{align*}
\varepsilon(T_i)&= T_i^{-1}, & \varepsilon(X_i)&= \ttt^{n-1}Y_i, & \varepsilon(\ttt^{n-1}Y_i)&= X_i, & \varepsilon(\qqq)&= \qqq^{-1}, & \varepsilon(\ttt)&= \ttt^{-1}
\end{align*}
%All three automorphisms are related by
%\[
%\tau_Y=\varepsilon\tau_X\varepsilon
%\]
\end{prop}
\noindent Note that $\varepsilon(\ee)$ is still a projector onto the ($\ttt$-deformed) trivial representation of the Hecke algebra, so $\varepsilon(\ee)=\ee$.
Thus, both $\varepsilon$ and $\tau_X$ restrict to $\SHH_n^0(\qqq,\ttt)$.

The automorphism $\tau_X$ is in some sense internal.
We can define the following \textit{Gaussian} in a suitable completion of $\HH_n(\qqq,\ttt)$:
\begin{equation*}
\begin{aligned}
\gamma_X&:=\exp\left( \sum_{i=1}^n\frac{\log(X_i)^2}{4\log(\qqq)} \right),&
%\gamma_Y&:=\varepsilon(\gamma_X)=\exp\left( -\sum_{i=1}^n\frac{\log(Y_i)^2}{4\log(\qqq)} \right)
\end{aligned}
\end{equation*}
%The following was proved in \cite{DiKedqt} (note that our $\qqq^2$ is their $\qqq$):
\begin{prop}[\cite{ChereDAHA,DiKedqt}]
We have $\tau_X= \ad_{\gamma_X^{-1}}$.
\end{prop}
\noindent Finally, we note that the image of the cyclotomic $\qqq$-Dunkl element $\varepsilon(D_1)$ can be written in terms of $\tau_X$ and hence $\ad_{\gamma_X^{-1}}$:
\begin{equation}
\varepsilon(D_1)=\bigg[\left( \ad_{\gamma_X}^{-1}-\ZZZ_1 \right)\circ\cdots\circ\left( \ad_{\gamma_X}^{-1}-\ZZZ_\ell \right)\bigg]\left(\ttt^{1-n}Y_1^{-1}\right)
\label{DunklGauss}
\end{equation}

\subsubsection{$\SHH_n^\ell(\ZZZ,\qqq,\ttt)$ and GKLO}
We need one more preparatory result.
The following is an analogue of Proposition \ref{TorGen}.

\begin{prop}\label{CycAppGen}
$\SHH_n^\ell(\ZZZ,\qqq,\ttt)$ is generated by
\[
\left\{ \ee X_1\ee,\, \ee D_1\ee\right\}\cup\ee\KK\left[ Y_1^{\pm 1},\ldots, Y_n^{\pm 1}\right]\ee
\]
\end{prop}

\begin{proof}
The proof is similar to that of Lemma 5.2 of \cite{FFJMM}.
Here, instead we try to generate $P_{a,k,b}$ from (\ref{Pakb}).
\end{proof}

\begin{thm}
For $\ZZZ=(\ZZZ_1,\ldots,\ZZZ_\ell)\in(\CC^\times)^\ell$, let $\ZZZ^{-1}:=(\ZZZ_1^{-1},\ldots,\ZZZ_\ell^{-1})$.
The polynomial representation $\mathfrak{r}$ gives an isomorphism:
\[
\mathfrak{r}:\varepsilon\left( \SHH_n^\ell(\ZZZ,\qqq,\ttt) \right)\rightarrow \mathrm{T}_{\ZZZ^{-1}}\left( U_{\qqq,\ttt}^\ell(\ddot{\gl}_1) \right)
\]
\end{thm}

\begin{proof}
It is evident that
\[
\mathfrak{r}\left( \ee\KK[X_1^{\pm 1},\ldots, X_n^{\pm1}] \right)=\mathrm{T}_{\ZZZ^{-1}}\left( \langle \psi_r^+,\psi_{\ell-r}^-\, |\, r\in\ZZ_{\ge 0}\rangle \right)
\]
Thus, by Propositions \ref{TorGen} and \ref{CycAppGen}, it suffices to show that
\begin{align}
\label{Yf}
\mathfrak{r}\left( \ee Y_1\ee \right)&= C_1\mathrm{T}_{\ZZZ^{-1}}(f_0)\\
\label{De}
\mathfrak{r}\left( \ee \varepsilon(D_1)\ee \right)&= C_2\mathrm{T}_{\ZZZ^{-1}}(e_0)
\end{align}
for some nonzero constants $C$ and $D$.
The first equality (\ref{Yf}) is well known due to the formulas for the first Macdonald operators (cf. \cite{MacBook}):
\begin{align}
\nonumber
\ttt^{n-1}\mathfrak{r}\left( \ee(Y_1+\cdots +Y_n)\ee \right)&= \sum_{i=1}^n\prod_{\substack{1\le j\le n\\ j\not=i}}\frac{x_i-\ttt^2 x_j}{x_i-x_j}\tau_{x_i,\qqq^2}\\
\ttt^{1-n}\mathfrak{r}\left( \ee(Y_1^{-1}+\cdots +Y_n^{-1})\ee \right)&= \sum_{i=1}^n\prod_{\substack{1\le j\le n\\ j\not=i}}\frac{x_i-\ttt^{-2} x_j}{x_i-x_j}\tau_{x_i,\qqq^2}^{-1}
\label{DualMac}
\end{align}
The formula for $\mathrm{T}_{\ZZZ^{-1}}(e_0)$ is not too far from the right-hand-side of (\ref{DualMac}).
By Lemma 2.7 from \cite{DiKedqt}, we have
\[
\tau_{x_i,\qqq^2}\mathfrak{r}(\gamma_X)=\qqq x_i\mathfrak{r}(\gamma_X)\tau_{x_i,\qqq^2}
\]
It follows that $\mathrm{T}_{\ZZZ^{-1}}(e_0)$ is a nonzero scalar multiple of
\[
\bigg[ \big( \ad_{\mathfrak{r}(\gamma_X)}^{-1}-\ZZZ_1 \big)\circ\cdots\circ\big( \ad_{\mathfrak{r}(\gamma_X)}^{-1}-\ZZZ_\ell \big) \bigg]\left( \sum_{i=1}^n\prod_{\substack{1\le j\le n\\ j\not=i}}\frac{x_i-\ttt^{-2} x_j}{x_i-x_j}\tau_{x_i,\qqq^2}^{-1} \right)
\]
Equality (\ref{De}) then follows from (\ref{DunklGauss}).
\end{proof}

%%fakesection

\bibliographystyle{alpha}
\bibliography{QHarish}

\end{document}